\documentclass[11pt]{amsart}
\usepackage{amsmath}
\usepackage{amssymb}
\usepackage{mathrsfs}
\usepackage{comment}
\usepackage{enumerate}

\usepackage{tikz}
\usepackage{cancel}
\usepackage{mathtools}
\mathtoolsset{showonlyrefs}

\theoremstyle{plain}
\newtheorem{thm}{Theorem}[section]
\newtheorem{lem}[thm]{Lemma}
\newtheorem{prop}[thm]{Proposition}
\newtheorem{cor}[thm]{Corollary}

\theoremstyle{definition}

\newtheorem{rem}[thm]{Remark}

\numberwithin{equation}{section}

\newcommand{\lda}{\lambda}
\newcommand{\om}{\Omega}                \newcommand{\pa}{\partial}
\newcommand{\va}{\varepsilon}            \newcommand{\ud}{\mathrm{d}}
\newcommand{\be}{\begin{equation}}      \newcommand{\ee}{\end{equation}}
                 
\newcommand{\Lda}{\Lambda}              
\newcommand{\R}{\mathbb{R}}              
\renewcommand{\H}{\mathbb{H}}

\begin{document}

\title[The Mass-Angular Momentum Inequality]
{The Mass-Angular Momentum Inequality \\
for Multiple Black Holes}

\author[Han]{Qing Han}
\address{Department of Mathematics\\
University of Notre Dame\\
Notre Dame, IN 46556, USA} \email{qhan@nd.edu}

\author[Khuri]{Marcus Khuri}
\address{Department of Mathematics\\
Stony Brook University\\
Stony Brook, NY 11794, USA}
\email{marcus.khuri@stonybrook.edu}

\author[Weinstein]{Gilbert Weinstein}
\address{Physics Department and Department of Mathematics\\
Ariel University\\
Ariel, Israel 40700}
\email{gilbertw@ariel.ac.il}

\author[Xiong]{Jingang Xiong}
\address{School of Mathematical Sciences\\
Laboratory of Mathematics and Complex Systems, MOE\\
Beijing Normal University\\
Beijing 100875, China}
\email{jx@bnu.edu.cn}

\thanks{Q. Han acknowledges the support of NSF Grant DMS-2305038. M. Khuri acknowledges the support of NSF Grants DMS-2104229 and DMS-2405045. G. Weinstein acknowledges the support of ISF Grant 1119/26. J. Xiong acknowledges the partial support of NSFC Grants 12325104 and 12271028.}

\begin{abstract}
This is the second in a series of two papers to settle the conjectured mass-angular momentum inequality for multiple black holes. 
More precisely, we show that for a complete, simply connected, maximal initial data set for the Einstein equations with multiple ends that are either asymptotically flat or asymptotically cylindrical, the ADM mass of any asymptotically flat end admits a lower bound given by the square root of total angular momentum, under the assumption of nonnegative energy density and axisymmetry. Moreover, equality is achieved in this mass lower bound only for a constant time slice of an extreme Kerr spacetime. 
The proof is based on a novel flow of singular harmonic maps with hyperbolic plane target, under which the renormalized harmonic map energy is monotonically nonincreasing. Relevant properties of the flow are achieved through a refined asymptotic analysis of solutions to the harmonic map equations and their linearization.
\end{abstract}


\maketitle

\section{Introduction}
\label{sec1} \setcounter{equation}{0}
\setcounter{section}{1}

In \cite{Penrose1} Penrose outlined heuristic arguments that give rise to the Penrose inequality \cite{Bray,HuiskenIlmanen},
as well as a conjectured lower bound \cite[Section 8]{Mars} for the ADM mass/energy $m$ of a spacetime in terms of ADM angular momentum $\mathcal{J}$, which takes the form
\begin{equation}\label{1}
m\geq\sqrt{|\mathcal{J}|}.
\end{equation}
In this paper we rigorously establish this inequality for any number of black holes, using a new flow of singular harmonic maps.

The reasoning that leads to \eqref{1} requires conservation of angular momentum, and to achieve this it is typically assumed that the spacetime is axisymmetric and satisfies certain conditions on matter fields. In fact, counterexamples have been constructed by Huang-Schoen-Wang \cite{HuangSchoenWang} when the axisymmetric assumption is removed. The original motivation for the Penrose inequality holds equally well for the mass-angular momentum inequality, namely they both serve as necessary conditions for the weak cosmic censorship \cite{Penrose,Penrose2} and the final state \cite{Klainerman} conjectures. Therefore, while a counterexample to the inequality would be detrimental for at least one of these latter two conjectures, confirmation of \eqref{1} only adds to the prevailing belief in their generic validity. Furthermore, inequality \eqref{1} may be interpreted as a refinement of the positive mass theorem \cite{ScYa,W}, in which a precise contribution to the total mass is expressed via the rotation of black holes. This manuscript is the second in a series of two papers to establish the mass-angular momentum inequality for multiple black holes; the first paper in this series is \cite{HKWX}.

The appropriate mathematical setting in which to study this inequality is that of an initial data set $(M,g,k)$ for the Einstein equations, consisting of a smooth connected 3-manifold $M$ with Riemannian metric $g$ and a symmetric 2-tensor $k$ representing the second fundamental form of an embedding into spacetime. These quantities satisfy the constraint equations
\begin{equation}\label{const}
16\pi\mu = R+(\mathrm{Tr}_{g}k)^{2}-|k|_{g}^{2},\quad\quad\quad
8\pi J = \operatorname{div}_{g}(k-(\mathrm{Tr}_{g}k)g),
\end{equation}
where $R$ is the scalar curvature of $g$, and $\mu$, $J$ denote the matter energy and momentum densities respectively. A natural hypothesis related to the dominant energy condition is the assumption of nonnegative energy density $\mu\geq 0$, which when combined with the \textit{maximal slice} condition $\mathrm{Tr}_g k=0$ guarantees nonnegative scalar curvature.
Moreover, it will be assumed throughout that the data are \textit{axially symmetric}. This means that a subgroup isomorphic to $U(1)$ is present within the isometry group of the Riemannian manifold, and that the extrinsic curvature is invariant under the $U(1)$ action. In particular, if $\eta$ is the associated Killing field generating the symmetry and $\mathfrak{L}_{\eta}$ denotes Lie differentiation, then
\begin{equation}\label{sym}
\mathfrak{L}_{\eta}g=\mathfrak{L}_{\eta}k=0.
\end{equation}

The initial data set $(M,g,k)$ will be \textit{asymptotically flat with multiple ends}, in that there exists a compact set $\mathcal{C}$, called the core, and an integer $N \geq 2$ such that $M\setminus\mathcal{C}=\cup_{i=0}^{N} M_{i}$, where the pairwise disjoint ends $M_i$ are either asymptotically flat or asymptotically cylindrical, with $i=0$ designating an asymptotically flat end. The other ends for $1\leq i\leq N$ represent individual black holes. Recall that an $i$th end is referred to as asymptotically flat if it is diffeomorphic to $\mathbb{R}^{3}\setminus B_1$, and in the resulting Cartesian coordinates with $\pmb{\delta}$ denoting the Euclidean metric we have
\begin{equation}\label{asymptotflat}
|\partial^\ell (g-\pmb{\delta})|_{\pmb{\delta}}=O(r^{-q-\ell}),\quad \ell\leq 6,\quad\quad\quad
|k|_{\pmb{\delta}}=O(r^{-q-1}),
\end{equation}
for some $q>\tfrac{1}{2}$ where $r$ is the Euclidean radial coordinate; additionally the scalar curvature is required to be integrable $R\in L^1(M_i)$. We note that the condition $\ell\leq 6$ is stronger than the typical requirement, and that this is due to the need for Brill coordinates \cite[Remark 3.2]{Chrusciel1}, (see also \cite{JaraczKhuriSokolowsky,Sokolowsky}).
Each asymptotically flat end comes equipped with the ADM energy, which is well-defined \cite{Bartnik,Chrusciel} and given by
\begin{equation}
m=\lim_{r\rightarrow\infty}\frac{1}{16\pi}\int_{S_{r}}\star_{\pmb{\delta}}\left(\mathrm{div}_{\pmb{\delta}}g-d\mathrm{Tr}_{\pmb{\delta}}g\right),
\end{equation}
where $\star_{\pmb{\delta}}$ is the Euclidean Hodge star operator, and $S_r$ is the coordinate sphere of radius $r$. Furthermore, axial symmetry implies that the ADM angular momentum of the end is aligned along the symmetry axis and thus may be described by a single number
\begin{equation}\label{AMdef}
\mathcal{J}=\lim_{r\rightarrow\infty}\frac{1}{8\pi}\int_{S_{r}}\left(k-(\mathrm{Tr}_{g} k)g\right)(\eta,\nu) dA,
\end{equation}
where $\nu$ and $dA$ are the unit outer normal and area element of $S_r$. Notice that the integral is invariant over any surface homologous to a coordinate sphere when the \textit{swirl} vanishes $J(\eta)=0$. This latter condition will be assumed in order to obtain a twist potential, and therefore make contact with harmonic maps. Moreover, it also shows that the total angular momentum is a sum of the individual angular momenta $\mathcal{J}_i$ evaluated at the auxiliary ends, that is $\mathcal{J}=\sum_{i=1}^{N}\mathcal{J}_i$. In particular, note that this implies such initial data without black holes must have vanishing ADM angular momentum.

Consider the warped cylindrical geometry $(\mathbb{R}_{+}\times S^2 ,\bar{g})$ with model metric
$\bar{g}=\lambda^2 d\bar{t}^2+ h$, where $\lambda$ is a positive smooth function independent of the radial coordinate $\bar{t}$, and $h$ is a metric on the 2-sphere. It is assumed further that $\bar{g}$ admits a Killing vector field that is tangent to $S^2$ and has periodic orbits. An end $(M_i,g)$ is then referred to as \textit{asymptotically cylindrical} if it is diffeomorphic to $\mathbb{R}_+\times S^2$, and in the coordinates provided by the diffeomorphism
\begin{equation}
|\bar{\nabla}^\ell(g-\bar{g})|_{\bar{g}}=O(e^{-q\bar{t}}), \quad \ell\leq 7,
\end{equation}
for some $q>0$ where $\bar{\nabla}$ denotes covariant differentiation with respect to $\bar{g}$. Again the large number of derivatives $\ell\leq 7$ is included for the existence of Brill coordinates \cite{JaraczKhuriSokolowsky,Sokolowsky}.  We note that extreme Kerr black holes possess an asymptotically cylindrical end in the sense described above.

In this setting, with the additional hypotheses of completeness, simple connectivity, and the absence of asymptotically cylindrical ends, the mass-angular momentum inequality \eqref{1} was initially studied by Dain and established for a single black hole through the combined work of Chru\'{s}ciel, Costa, Dain, Schoen, and Zhou in \cite{Chrusciel1,ChruscielCosta,Costa,Dain,SchoenZhou}. The survey \cite{DainGabachClement} details many of these developments, and \cite{BrydenKhuriSokolowsky} addresses situations where completeness and simple connectivity may be removed. The case of multiple black holes was taken up by Chru\'{s}ciel-Li-Weinstein \cite{CLW} (see also Khuri-Weinstein \cite{KhuriWeinstein} for the inclusion of charge) who proved the lower bound
\begin{equation}\label{initialmb}
m\geq \mathcal{F}(\mathcal{J}_{1},\ldots,\mathcal{J}_{N},z_1,\ldots,z_N),
\end{equation}
where the function $\mathcal{F}$ is proportional to a renormalized harmonic map energy depending on the angular momenta $\mathcal{J}_{i}$ and positions $z_i$ associated with the $N$ black holes. Our main result shows that this function is bounded below by the desired quantity $\sqrt{|\mathcal{J}|}$. The proof is based on a novel flow of singular harmonic maps with hyperbolic plane target, under which the renormalized harmonic map energy is monotonically nonincreasing. Relevant properties of the flow are achieved through a refined asymptotic analysis of solutions to the harmonic map equations and their linearization. 

\begin{thm}\label{maintheorem}
Let $(M, g, k)$ be a complete, simply connected, axially symmetric, maximal initial data set for the Einstein equations with one designated asymptotically flat end and finitely many other asymptotically flat or asymptotically cylindrical ends. Assume further that the nonnegative energy density condition is satisfied $\mu\geq 0$, and the momentum density vanishes in the direction of rotation $J(\eta)=0$. Then
\begin{equation*}
m\geq \sqrt{|\mathcal{J}|},
\end{equation*}
with equality if and only if the data arise from a constant time slice of an extreme Kerr black hole.
\end{thm}

As mentioned above, the proof is based on a flow of singular harmonic maps. More precisely, we observe that the first variation of renormalized harmonic energy under perturbations of the black hole singularities within the maps, is determined by the conical singularities along the axis in the associated stationary vacuum spacetimes. From this we are able to define a flow of the singular harmonic maps, along with the corresponding stationary vacuum black holes solutions, which moves the black holes guided by the conical singularities and results in a monotone nonincreasing renormalized energy. It is shown that the flow exists up to a maximal time, at which point two different scenarios can occur: the black holes either collide,  or stagnate. In the first case, the limiting renormalized energy exists and is shown to be greater than or equal to the renormalized energy of the configuration after collision. Thus, the flow may be restarted at the new collision configuration while preserving monotonicity through the singular time. In the case of a stagnating flow, we find that the limiting harmonic map gives rise to a regular multi-Kerr spacetime, in contradiction to the black hole uniqueness theorem \cite{HKWX_Unique}. With these observations, an induction argument on the number of black holes is then used to establish the main result.

This paper is organized as follows. In the next section, background material and the set up relating a preliminary
mass lower bound to the harmonic map energy are described. A flow of harmonic maps via movement of singular points is introduced in Section \ref{secflow}, and statements of the results concerning families of singular harmonic maps and their linearization are also discussed. The flow is studied in more detail in Section \ref{monotonicity}, where it is shown that the renormalized energy is monotonically nonincreasing along the evolution. Collisions of the singular points along the flow are analyzed in Section \ref{collision}, where it is proven that energy monotonicity is preserved in this situation. The proof of the main theorem is provided in Section \ref{sec6}. Section \ref{sec:linear-hm} is dedicated to an asymptotic analysis of the linearized local harmonic map system at horizons, while Section \ref{sec:diff-dp} establishes differentiability properties of the harmonic maps with respect to parameters. 
Finally, an appendix is included to collect miscellaneous auxiliary lemmas.

\subsection*{Acknowledgements}
The authors would like to thank Demetrios Christodoulou for helpful comments. 
This project was made possible in part by a SQuaRE at the American Institute of Mathematics. The authors thank AIM for providing a supportive and mathematically rich environment. 

\section{Background and Set Up}
\label{sec2} \setcounter{equation}{0}
\setcounter{section}{2}

Let $(M,g)$ be a complete, simply connected, axially symmetric Riemannian 3-manifold which is asymptotically flat with $N+1$ ends. Then it is shown in \cite[Theorem 2.9, Remark 3.2]{Chrusciel} (only asymptotically flat ends) and \cite[Theorem 1.0.5]{Sokolowsky} (including asymptotically cylindrical ends) that $M\cong\mathbb{R}^{3}\setminus\{p_{1},\dots,p_N\}$, and that there exists a global (cylindrical)
Brill coordinate system $(\rho,z,\phi)$ on $M$ with $\rho\geq 0$ and $z\in\mathbb{R}$ parameterizing a half-plane orbit space, and where $\phi\in[0,2\pi)$ is associated with the rotational Killing field so that $\eta=\partial_{\phi}$. Moreover, the points $p_{i}$ (referred to as punctures) represent individual black holes and all lie on the $z$-axis $\Gamma$, where $|\eta|=0$. In these coordinates the metric may be expressed to exhibit the Riemannian submersion structure
\begin{equation}\label{brill}
g=e^{-2\mathbf{U}+2\sigma}(d\rho^{2}+dz^{2})+\rho^{2}e^{-2\mathbf{U}}(d\phi+A_{\rho}d\rho+A_{z}dz)^{2},
\end{equation}
where all coefficient functions are independent of $\phi$ and $\rho e^{-\mathbf{U}}(d\phi+A_{\rho}d\rho+A_{z}dz)$ is the dual 1-form to fiber directions $|\eta|^{-1}\eta$. In the designated asymptotically flat end $M_0$, as Euclidean distance
$r=\sqrt{\rho^{2}+z^{2}}\rightarrow\infty$ the coefficient functions satisfy the decay\footnote{The notation
$\mathbf{h}=O_{\ell}(r^{-q})$ asserts that $|\partial_{r}^{l}\mathbf{h}|\leq Cr^{-q-l}$ for all $l\leq \ell$.}
\begin{equation}\label{asymdecaycoef}
\mathbf{U}=O_{\ell-3}(r^{-q}),\text{ }\text{ }\text{ }\text{ }\sigma=O_{\ell-4}(r^{-q}),\text{ }\text{ }\text{
}\text{ }A_{\rho}=\rho O_{\ell-3}(r^{-q-2}),\text{ }\text{ }\text{ }\text{ }A_{z}=O_{\ell-3}(r^{-q-1}),
\end{equation}
for $\ell\leq 6$.
In the remaining ends associated with punctures we have $r_{i}\rightarrow 0$, where $r_{i}$ is the Euclidean distance to $p_{i}$, and the coefficients satisfy the following decay in the asymptotically flat case
\begin{equation}\label{punctureasymf}
\mathbf{U}=2\log r_{i}+O_{\ell-4}(r_{i}^{q}),\text{ }\text{ }\text{ }\text{ }\sigma=O_{\ell-4}(r_{i}^{q}),\text{
}\text{ }\text{ }\text{ }A_{\rho}=\rho O_{\ell-3}(r_{i}^{q}),\text{ }\text{ }\text{ }\text{
}A_{z}=O_{\ell-3}(r_{i}^{q+1}),
\end{equation}
for $\ell\leq 6$, while in the asymptotically cylindrical case
\begin{equation}\label{punctureasymc}
\mathbf{U}=\log r_{i}\!+\!\hat{\mathbf{U}}(\theta_i)\!+\!O_{\ell-4}(r_i^q),\text{ }\text{ }\sigma=\hat{\sigma}(\theta_i)\!+\!O_{\ell-4}(r_i^q),\text{
}\text{ }A_{\rho}=\rho O_{\ell-3}(r_{i}^{q}),\text{ }\text{
}A_{z}=O_{\ell-3}(r_{i}^{q+1}),
\end{equation}
for $\ell\leq 7$, where $\hat{\mathbf{U}}$ and $\hat{\sigma}$ are some functions of the polar angle $\theta_i$ centered at $p_i$.

To obtain a twist potential observe that a calculation produces
\begin{equation}
d\star \left(k(\eta)\wedge \eta\right)=-\iota_{\eta}\star J(\eta)=0,
\end{equation}
where $\star$ is with respect to $(M,g)$, interior product is denoted by $\iota$, and we use the same notation $\eta$ for both the vector field and its dual 1-form. Since $M$ is simply connected, there then exists a smooth potential function $\mathbf{v}$ such that
\begin{equation}
d\mathbf{v}=\star \left(k(\eta)\wedge \eta\right).
\end{equation}
Notice that the $z$-axis with punctures removed may be decomposed into a set of pairwise disjoint open intervals $\Gamma\setminus \{p_{1},\dots,p_N\}=\cup_{j=1}^{N+1}\Gamma_j$,
for which $\mathbf{v}=c_j$ is constant on each $\Gamma_j$. The intervals are called \textit{axis rods}, whereas the corresponding values of $\mathbf{v}$ are referred to as \textit{potential constants}, and they determine the angular momentum of each black hole. More precisely in the maximal case
\begin{equation}\label{angpotconst}
\mathcal{J}_{i}=\lim_{r_i \rightarrow 0}\frac{1}{8\pi}\int_{S_{r_i}}k(\eta,\nu)dA=\frac{1}{4}(c_{i+1} -c_{i}),
\end{equation}
where $S_{r_i}$ is a coordinate sphere centered at $p_i$ and $\nu$ is the unit normal pointing towards the designated asymptotically flat end. Here the labeling of intervals is such that $\Gamma_{i+1}$, $\Gamma_{i}$ lie directly above and below $p_i$, respectively. Furthermore, using the orthonormal frame
\begin{equation}
e_{1}=e^{\mathbf{U}-\sigma}(\partial_{\rho}-A_{\rho}\partial_{\phi}),
\text{ }\text{ }\text{ }\text{ } e_{2}=e^{\mathbf{U}-\sigma}(\partial_{z}-A_{z}\partial_{\phi}),
\text{ }\text{ }\text{ }\text{ } e_{3}=\rho^{-1} e^{\mathbf{U}}\partial_{\phi},
\end{equation}
we find
\begin{equation}
k(e_{1},e_{3})
=-|\eta|^{-2}e_2(\mathbf{v}),\quad\quad
k(e_{2},e_{3})=|\eta|^{-2}e_1(\mathbf{v}).
\end{equation}
It follows that
\begin{equation}\label{aoinfoianioh}
|k|_{g}^{2}\geq 2\left(k(e_{1},e_{3})^{2}
+k(e_{2},e_{3})^{2}\right)
=2\frac{e^{6\mathbf{U}-2\sigma}}{\rho^{4}}|\nabla \mathbf{v}|^2,
\end{equation}
where the norm $|\cdot|$ is with respect to the Euclidean metric.

\subsection{Relation between mass and harmonic map energy.} Recall that in Brill coordinates the scalar curvature may be expressed (\cite{Brill}, \cite{Dain}) as
\begin{equation}
2e^{-2\mathbf{U}+2\sigma}R=8\Delta \mathbf{U}-4\Delta_{\rho,z}\sigma-4|\nabla \mathbf{U}|^{2}-\rho^{2}e^{-2\sigma}
\left(A_{\rho,z}-A_{z,\rho}\right)^{2},
\end{equation}
where $\Delta$ is the Euclidean Laplacian on $\mathbb{R}^{3}$ and $\Delta_{\rho,z}=\partial_{\rho}^{2}
+\partial_{z}^{2}$. Integrating this formula by parts yields the following mass formula \cite[(3.9)]{Chrusciel1} (\cite[(4.5.15)]{Sokolowsky} including asymptotically cylindrical ends)
\begin{equation}
m=\frac{1}{32\pi}\int_{\mathbb{R}^{3}}\left(2e^{-2\mathbf{U}+2\sigma}R+4|\nabla \mathbf{U}|^{2}
+\rho^{2}e^{-2\sigma}(A_{\rho,z}-A_{z,\rho})^{2}\right)d\mathbf{x}.
\end{equation}
Observe that the Hamiltonian contraint \eqref{const}, the nonnegative energy density and maximality conditions, together with \eqref{aoinfoianioh} imply that
\begin{equation}\label{aofnaoingoih}
R=16\pi\mu+|k|_{g}^{2}\geq 2\frac{e^{6\mathbf{U}-2\sigma}}{\rho^{4}}|\nabla \mathbf{v}|^2,
\end{equation}
and therefore
\begin{equation}\label{masslowerbound1}
m\geq  \frac{1}{8\pi}\int_{\mathbb{R}^{3}}\left(|\nabla \mathbf{U}|^{2}
+\frac{e^{4\mathbf{U}}}{\rho^{4}}|\nabla \mathbf{v}|^2\right)d\mathbf{x} =: \frac{1}{8\pi}\mathcal{E}(\bar{\Phi}).
\end{equation}
Notice that with the substitutions $\mathbf{U}=\mathbf{u}+\ln\rho$, and using that $\ln\rho$ is harmonic on $\mathbb{R}^3 \setminus \Gamma$, the functional on the right-hand side is up to boundary terms the harmonic map energy of a map $\bar{\Phi}=(\mathbf{u},\mathbf{v}):\mathbb{R}^{3}\setminus\Gamma\rightarrow
\mathbb{H}^{2}$, where the metric $d\mathbf{u}^2 +e^{4\mathbf{u}}d\mathbf{v}^2$ of the hyperbolic plane (with Gaussian curvature $-4$) is expressed in horospherical coordinates. We refer to $\mathcal{E}(\bar{\Phi})$ as the \textit{renormalized energy} of the map $\bar{\Phi}$, since the infinite terms involving $\nabla\ln\rho$ have been removed.

In \cite[Proposition 2.1]{CLW} (see also \cite[Corollary 3.2]{KhuriWeinstein}) it is shown that there exists a unique singular harmonic map $\Phi=(u,v):\mathbb{R}^3 \setminus \Gamma\rightarrow\mathbb{H}^2$ having the same potential constants as $\bar{\Phi}$, which is asymptotic to the relevant extreme Kerr harmonic map near each puncture $p_i$ as well as at infinity, and for which $u= -\ln\rho +O(1)$ upon approaching the rods $\Gamma_i$. The harmonic map equations are given by
\begin{equation}\label{equationaklfahg1}
\Delta u-2e^{4u}|\nabla v|^2=0,\quad\quad \Delta v+4\nabla u\cdot\nabla v=0.
\end{equation}
Furthermore, it is proven in \cite{CLW} and \cite[Theorem 4.1]{KhuriWeinstein} that this map minimizes\footnote{The assumption \cite[(1.5)]{KhuriWeinstein} indicates derivative decay of $k$, however this is not used. Therefore \eqref{asymptotflat} is sufficient.} the renormalized energy so that $\mathcal{E}(\bar{\Phi})\geq\mathcal{E}(\Phi)$, and hence
\begin{equation}\label{masslowerb2}
m\geq \frac{1}{8\pi}\mathcal{E}(\Phi)=:\mathcal{F}(\mathcal{J}_{1},\ldots,\mathcal{J}_{N},z_1,\ldots,z_N).
\end{equation}
This relies on the observation that asymptotically cylindrical ends carry less energy than asymptotically flat ends. We also note that no additional decay assumption on $k$ is required along an asymptotically cylindrical ends to achieve \eqref{masslowerb2}. This is due to the fact that $R=O(1)$ and \eqref{aofnaoingoih} imply
$|\nabla v|
\leq C\rho^2r_i^{-3}$ along the $i$th end,
which is precisely the cylindrical-end estimate of \cite[(4.10)]{KhuriWeinstein} with $\lambda=2$; the remaining hypotheses of \cite[Theorem 4.1]{KhuriWeinstein} are immediately verified.

The remainder of this manuscript is dedicated to proving that the function $\mathcal{F}$ is bounded below by $\sqrt{|\mathcal{J}|}$. This entails two separate analyses to obtain asymptotics near punctures and other regimes, namely one for the nonlinear harmonic map system \eqref{equationaklfahg1} carried out in \cite{HKWX} and another for the linearized equations carried out in Section \ref{sec:diff-dp} below. Moreover, a flow of harmonic maps $\Phi_t$, for $t\in\mathbb{R}$, is developed based on the movement of punctures $p_{i}(t)$ in the direction of conical singularity angle defects of an associated stationary vacuum spacetime. The renormalized energy $\mathcal{E}(\Phi_t)$ will be shown to be monotonically nonincreasing along the flow and to converge, after possible jumps when punctures collide, to a value not less than the desired lower bound of the main theorem.

\section{The Flow of Punctures and Further Harmonic Map Analysis}
\label{secflow} \setcounter{equation}{0}
\setcounter{section}{3}

Fix potential constants $c_1,\dots,c_{N+1}$ that give rise to the nonzero angular momenta $\mathcal{J}_i$ as given by \eqref{angpotconst}, and let $\mathbf{z}=(z_1,\dots,z_N)$ be the $z$-coordinates for $N$ distinct punctures
$p_i$ on the $z$-axis $\Gamma\subset \R^3$. As discussed above, for each such data set there is a unique singular harmonic map $\Phi_{\mathbf{z}}=(u_{\mathbf{z}},v_{\mathbf{z}})\colon\R^3\setminus\Gamma\to\H^2$ with prescribed blow-up at punctures designated by $\mathbf{z}$, and having these given potential constants. The quantity $\mathbf{z}$ may be viewed as a multivariate parameter for the harmonic maps $\Phi_{\mathbf{z}}$. Let $\mathbf{z}_0=(z_{0,1},\dots, z_{0,N})$ with $z_{0,1}<z_{0,2} <\dots<z_{0,N}$, and consider a ball $B_{\varepsilon}^N(\mathbf{z}_0)\subset\mathbb{R}^N$ with this center, having radius $\varepsilon>0$ small enough to ensure that the closed balls $\overline{B_{\varepsilon}(p_i^0)}\subset\mathbb{R}^3$ remain pairwise disjoint, where $p_i^0=(0,0,z_{0,i})\in\mathbb{R}^3$ are the punctures associated with $\mathbf{z}_0$.
According to \cite[Theorems 2.1 and 2.2]{HKWX} we have the following expansions for the harmonic maps $\Phi_{\mathbf{z}}$ in a neighborhood of each puncture with $z$-coordinate $z_i$, using polar coordinates $(r_i,\theta_i,\phi)$ centered at this puncture, namely
\begin{equation}\label{3.1}
U_{\mathbf{z}}  =u_{\mathbf{z}} +\ln\rho=\ln r_i + \bar{U}_{\mathbf{z},i}(\theta_i) +\tilde{U}_{\mathbf{z},i}(r_i,\theta_i) ,\quad\quad
v_{\mathbf{z}} =\bar{v}_{\mathbf{z},i}(\theta_i) + \tilde{v}_{\mathbf{z},i}(r_i,\theta_i),
\end{equation}
where there exist constants $b_i\in(-1,1)$ and $a_i=2\mathcal{J}_i >0$ such that the barred functions yield the renormalized tangent maps
\begin{align} \label{eq:tange-map}
\begin{split}
\bar{U}_{\mathbf{z},i} &=\bar{U}(\theta_i,b_i) = -\frac{1}{2}\ln\left(\frac{2a_i \sqrt{1-b_i^2}}{1+\cos^2 \theta_i +2b_i \cos\theta_i}\right),\\
\bar{v}_{\mathbf{z},i} &=\bar{v}(\theta_i,b_i)= a_i\left(\frac{b_i +b_i \cos^2 \theta_i +2\cos\theta_i}{1+\cos^2 \theta_i +2b_i\cos\theta_i} \right).
\end{split}
\end{align}
Although $a_i$ and the angular momentum are assumed positive here, no generality is lost as the generic case of these local expansions may be obtained from this one by replacing $v_{\mathbf{z}}$ with $\pm v_{\mathbf{z}}+c$, for some appropriately chosen constant $c$; this transformation yields an isometry of the target hyperbolic space, and hence does not affect the harmonic map equations. Notice that each $(\bar{u}_{\mathbf{z},i}=\bar{U}_{\mathbf{z},i}-\ln\sin \theta_i, \bar{v}_{\mathbf{z},i})$ is harmonic from $\mathbb{S}^2\setminus \{N,S\} \to \mathbb{H}^2$ where $N$, $S$ represent north and south poles, and thus provides the tangent map at the puncture. The constants $b_i$ will be referred to as \textit{tangent map parameters}. Furthermore the error terms satisfy
\begin{equation}\label{3.3}
\begin{split}
|\tilde{U}_{\mathbf{z},i}|+r_i|\partial_{r_i} \tilde{U}_{\mathbf{z},i}|+|\partial_{\theta_i}\tilde{U}_{\mathbf{z},i}| &=O(r_i^{\beta_i}),\\ (\sin\theta_i)^{-(3+\varsigma)} \left(|\tilde{v}_{\mathbf{z},i}|+r_i |\partial_{r_i} \tilde{v}_{\mathbf{z},i}|\right)+(\sin\theta_i)^{-(2+\varsigma)}|\partial_{\theta_i}\tilde{v}_{\mathbf{z},i}|&=O(r_{i}^{\beta_i}),
\end{split}
\end{equation}
for some $\beta_i \in(0,1)$ where $\varsigma\in (0,1)$ is arbitrary. Similar estimates hold for higher derivatives up to and including order three. It should be pointed out that while $\theta_i$ and $b_i$ are functions of $\mathbf{z}$, the $a_i$ are independent of $\mathbf{z}$. Furthermore, it follows from \cite[Theorem 2.3]{HKWX} that the relevant expansions near infinity (as $r\rightarrow\infty$) are given by
\begin{equation}\label{asyminfin}
u_{\mathbf{z}}=-\ln\rho +\frac{c_1}{r}+O(r^{-2}),\quad\quad
v_{\mathbf{z}}=c_2 \cos\theta(3-\cos^{2}\theta)+c_3 +O\left(\frac{(\sin\theta)^{3+\varsigma}}{r}\right),
\end{equation}
for some constants $c_1$, $c_2$, $c_3$ with corresponding fall-off for derivatives. Note that there is no constant term in the expansion of $u_{\mathbf{z}}$ due to asymptotic flatness.

The following result addresses the smooth dependence problem for the singular harmonic maps and their corresponding tangent maps. The notation $\partial_{z_i}$ will be used to denote differentiation with respect to the components of the parameter $\mathbf{z}$, which should be distinguished from $\partial_z$ which represents domain differentiation with respect to the $z$-coordinate of $\mathbb{R}^3$.

\begin{thm}\label{thm:smoothness} For any $\varsigma\in(0,1)$ the map $B_{\varepsilon}^N(\mathbf{z}_0)\rightarrow C^{3,\varsigma}_{loc}(\mathbb{R}^3 \setminus \cup_{i=1}^{N}\overline{B_{\varepsilon}(p_i^0)},\mathbb{R}^2)$ given by $\mathbf{z}\mapsto (U_{\mathbf{z}},v_{\mathbf{z}})$ is smooth, and the map $B_{\varepsilon}^N(\mathbf{z}_0)\rightarrow C^{3,\varsigma}(\mathbb{S}^2,\mathbb{R}^2)$ given by $\mathbf{z}\mapsto (\bar{U}_{\mathbf{z},i},\bar{v}_{\mathbf{z},i})$ is smooth for each $i=1,\dots,N$. Moreover, for any $i,j\in \{1,\dots,N\}$ and $l+\kappa \leq 3$ the following estimates hold
\begin{align} \label{eq:thm3.1-1}
\begin{split}
|\nabla^{\kappa}( \partial_{z_j} u_{\mathbf{z}} + \pmb{\delta}_{ij}\pa_{z} u_{\mathbf{z}})| + (\sin \theta_i )^{\kappa-3-\varsigma } |\nabla ^\kappa ( \partial_{z_j} v_{\mathbf{z}} +\pmb{\delta}_{ij} \pa_{z} v_{\mathbf{z}})|
&\le C r_i^{-\kappa} \quad \mbox{in }B_{\va/2}(p_i)\setminus\Gamma,\\[2mm]
|\nabla_{\mathbb S^2}^\kappa \partial_{z_j}\bar{u}_{\mathbf{z},i} |+ (\sin\theta_i)^{\kappa-3-\varsigma } |\nabla_{\mathbb S^2}^\kappa \partial_{z_j}\bar{v}_{\mathbf{z},i} |
&\le C  \quad\text{on }\mathbb S^2 \setminus\{N,S\}, \\[2mm]
|\nabla^{\kappa}( \partial_{z_j} u_{\mathbf{z}} )| + (\sin \theta )^{\kappa-3-\varsigma } |\nabla ^\kappa ( \partial_{z_j} v_{\mathbf{z}} )| \le C r^{-(\kappa+1)} \quad &\text{in }\mathbb{R}^3\setminus \left(\Gamma\cup_{i=1}^N B_{\va/2}(p_i)\right),
\end{split}
\end{align}
and
\begin{align}\label{eq:thm3.1-2}
\begin{split}
&|(r_i \partial_{r_i})^l \nabla_{\mathbb{S}^2}^\kappa [(\partial_{z_j} u_{\mathbf{z}} + \pmb{\delta}_{ij}\pa_{z} u_{\mathbf{z}}) -\partial_{z_j}\bar{u}_{\mathbf{z},i}]|\\
&\qquad+e^{(3+\varsigma-\kappa)u_{\mathbf{z}}}|(r_i \partial_{r_i})^l \nabla_{\mathbb{S}^2}^\kappa [( \partial_{z_j} v_{\mathbf{z}} + \pmb{\delta}_{ij}\pa_{z} v_{\mathbf{z}})-\partial_{z_j}\bar{v}_{\mathbf{z},i}]|
\le C r_i^{\bar\beta} \quad \mbox{in }B_{\va/2}(p_i)\setminus\Gamma,
\end{split}
\end{align}
for some constants $C$ and $\bar\beta>0$ depending on $\varepsilon$, $c_1,\dots, c_{N+1}$, and $\sum_{i=1}^N|z_i|$, with $C$ depending also on $\varsigma$. 
\end{thm}

The proof of this result is presented in Section \ref{sec:diff-dp}, and relies on two primary elements. These include
convexity estimates of the renormalized energy, derived by Schoen-Zhou \cite{SchoenZhou} and Khuri-Weinstein \cite{KhuriWeinstein}, which provide the starting point $C^{1/2}$ regularity of $\mathcal{E}(\Phi_{\mathbf{z}})$. Additionally, a refined asymptotic analysis based on \cite{HKWX} is employed, which is detailed in Section \ref{sec:linear-hm}.

Let $\varepsilon_0 >0$ be sufficiently small and consider a smooth curve of punctures $\mathbf{z}(t) \subset B_{\varepsilon}^N(\mathbf{z}_0)$ for $t\in (-\va_0,\va_0)$, with $\mathbf{z}(0)=\mathbf{z}_0$. Then Theorem \ref{thm:smoothness} implies that the family of harmonic maps $\Phi_t=(u_t,v_t):=(u_{\mathbf{z}(t)},v_{\mathbf{z}(t)})$ is smooth in $t$ away from the axis. We emphasize that $\theta_i$ and $b_i$ are functions of $t$, whereas $a_i$ are independent of $t$. The next result essentially follows from Theorem \ref{thm:smoothness}, and its proof is also given in Section \ref{sec:diff-dp}.

\begin{cor}\label{cor:linest}
Let $(\dot U_t,\dot v_t)$ be the derivative with respect to $t$ of the renormalized family of harmonic maps.
In a neighborhood of each puncture $p_i(t)$, as $r_i(t) \rightarrow 0$, the following expansions hold
\begin{align}\label{abcd}
\begin{split}
\dot{U}_t 
&=\left(-\frac{z-z_i}{r_i^2} + \frac{\rho}{r_i^2}\partial_{\theta_i}\bar{U}\right)\dot{z}_i+\dot{b}_i\partial_{b_i}\bar{U}
+O(r_i^{\beta_i -1}),\\
\dot{v}_t 
&=\frac{\rho}{r_i^2}(\partial_{\theta_i}\bar{v})\dot{z}_i+\dot{b}_i\partial_{b_i}\bar{v} +O(r_i^{\beta_i -1} \sin\theta_i).
\end{split}
\end{align}
Furthermore, there exists a constant $C$ such that
\begin{equation}\label{estrho}
|\dot  U_t| + \frac{e^{2U_t}}{\rho^{2}}|\dot v_t| \leq C    \text{ }\text{ as $\rho\rightarrow 0$ away from punctures},
\end{equation}
and
\begin{equation}\label{est}
|\dot  U_t| + \frac{e^{2U_t}}{\rho^{2}}|\dot v_t| \leq Cr^{-1} \text{ }
\text{ as $r\to\infty$}.
\end{equation}
\end{cor}

The desired curve of punctures arises from a flow guided by the tangent map parameters, where the potential constants are held fixed according to the prescribed angular momenta. More precisely, we
require the curve to satisfy the autonomous system
\begin{equation}\label{flowdef}
\frac{dz_i}{dt}=-b_i(z_1,\ldots,z_N), \quad i=1,\ldots,N,
\end{equation}
which we will refer to as the \textit{puncture flow}. In the next section it is shown that when an initial condition is given, there is a unique $C^2$ solution for a maximal time interval. Furthermore, it will also be established that the renormanlized energy is monotone along the flow, in fact when smooth the flow implies
\begin{equation}\label{gonaoinoinihjjj}
\frac{d}{dt}\mathcal{E}(\Phi_t)=- \sum_{i=1}^{N}f(b_i)b_i \leq 0
\end{equation}
for some increasing function $f$ with $f(0)=0$. Monotonicity in the presence of collisions is proven in Section \ref{collision}. The relation to extreme black hole uniqueness arises in the case that all $b_i$ vanish. More precisely, with a harmonic map $\Phi=(u,v)$ as described, one may build a stationary vacuum spacetime on $\mathbb{R}\times \left(\mathbb{R}^3 \setminus \{p_1,\ldots,p_N\}\right)$ with metric in Weyl coordinates given by
(see \cite[Section 2]{HKWX})
\begin{equation} \label{4d-g}
\boldsymbol{g}=-e^{2U}d\tau^2 +\rho^{2}e^{-2U}(d\phi +w d\tau)^2 +e^{-2U+2\boldsymbol{\alpha}}(d\rho^2 +dz^2),
\end{equation}
where $w$ is obtained from $dw=2e^{4u}\iota_{\eta}\star_{\pmb{\delta}} dv$ and $\boldsymbol{\alpha}$ satisfies
\begin{equation}\label{alphaeq}
\partial_{\rho}\boldsymbol{\alpha}=\rho\left[(\partial_{\rho}U)^2 \!-\!(\partial_z U)^2+e^{4u}((\partial_{\rho}v)^2 \!-\!(\partial_z v)^2)\right],\quad
\partial_z \boldsymbol{\alpha} =2\rho\left(\partial_{\rho} U \partial_z U +e^{4u} \partial_{\rho} v \partial_z v \right).
\end{equation}
According to \cite[Theorem 2.4]{HKWX}, the tangent map parameter at $p_i$ is related to the logarithmic angle defects $\mathbf{b}_{i+1}$, $\mathbf{b}_i$ of the neighboring axis rods for the constructed spacetime, via the formula
\begin{equation}\label{fffoanoinilnhj}
\mathbf{b}_{i+1}-\mathbf{b}_{i}=\ln\left(\frac{1+b_i}{1-b_i}\right).
\end{equation}
Thus if all $b_i=0$, which indicates a stationary point for the flow of punctures, then the associated spacetime would be devoid of conical singularities. This leads to a regular stationary vacuum solution with multiple degenerate black holes, which would be in contradiction to the extreme black hole uniqueness theorem.  For convenience, we state here the version used in this paper.

\begin{thm}\label{bhu}
Let $\Phi_{\mathbf{z}}:\mathbb{R}^3 \setminus\Gamma \rightarrow\mathbb{H}^2$ be a singular harmonic map associated with 
a collection of $N>1$ distinct punctures located on the $z$-axis at $\mathbf{z}=(z_1,\dots,z_N)$, and having potential constants giving rise to nonzero angular momenta $\mathcal{J}_i$ at each such location. Then at least one of the tangent map parameters $b_i$, $i=1,\dots,N$ must be nonzero, in particular $b_1 <0$ and $b_{N}>0$.
\end{thm}

This version follows as a special case of \cite[Theorem 1.1]{HKWX_Unique}, using the relation \eqref{fffoanoinilnhj}.

\section{Energy Monotonicity}
\label{monotonicity}
\setcounter{equation}{0}
\setcounter{section}{4}

In this section we will show that the renormalized harmonic map energy of harmonic maps moving according to the flow \eqref{flowdef} is monotonically nonincreasing, whenever the flow is smooth. The first step is to establish a short time existence result in the next proposition. This follows from standard ODE theory, as the functions $b_i$ used in the definition of the flow are well-behaved. In particular, they are smooth with respect to movement of punctures due to the tangent map statement of Theorem \ref{thm:smoothness}. We will also establish a version of this statement in the following result using a different method.

\begin{prop}\label{blipschitz}
The tangent map parameters $b_i(\mathbf{z})$, $i=1,\ldots,N$ are continuously differentiable in $\mathbf{z}$ as long as the $z_i$ remain pairwise distinct. Consequently, given an initial condition $\mathbf{z}(0)$ with $z_i(0)<z_{i+1}(0)$ there exists a maximal time $T>0$ (possibly infinite) and a unique $C^2\left([0,T),\mathbb{R}^N \right)$ solution $\mathbf{z}(t)$ of the initial value problem corresponding to the puncture flow \eqref{flowdef}. The solution satisfies $z_i(t)<z_{i+1}(t)$ for all $i$ and $t\in[0,T)$, and the maximal time, if finite, is characterized by the property that 
\begin{equation}\label{faoijofinoihyhj}
\liminf_{t\rightarrow T}(z_{j+1}(t)-z_j(t))=0,
\end{equation}
for some $j\in\{1,\ldots,N-1\}$.
\end{prop}


\begin{proof}
Below it will be proven that
\begin{equation}\label{faoijofinoihyhj10}
\limsup_{t\rightarrow T}(z_{j+1}(t)-z_j(t))=\infty,\quad\text{ or }\quad \liminf_{t\rightarrow T}(z_{j+1}(t)-z_j(t))=0.
\end{equation}
We note that due to the last statement of Theorem \ref{bhu}, the first possibility of \eqref{faoijofinoihyhj10} never occurs.

Consider the harmonic map $\Phi_{\mathbf{z}}=(u_{\mathbf{z}},v_{\mathbf{z}})$ with punctures $p_i$ designated by $\mathbf{z}$, and having fixed potential constants. We may view $\mathbf{z}$ as a multivariate parameter for the harmonic maps $\Phi_{\mathbf{z}}$. Let $\mathbf{z}(0)=\mathbf{z}_0$ have distinct components, and take a ball $B_{\varepsilon}^N(\mathbf{z}_0)\subset\mathbb{R}^N$ where $\varepsilon>0$ is small enough to ensure that the components of each point remain distinct. According to Theorem \ref{thm:smoothness}, the map $B_{\varepsilon}^N(\mathbf{z}_0)\rightarrow C^{3,\varsigma}(\mathbb{R}^3 \setminus \Gamma_{\varepsilon},\mathbb{R}^2)$ given by $\mathbf{z}\mapsto (U_{\mathbf{z}},v_{\mathbf{z}})$ is continuously differentiable, where $U_{\mathbf{z}}=u_{\mathbf{z}}+\ln\rho$ and $\Gamma_{\varepsilon}=\Gamma \cap \left(\cup_{i=1}^{N} B_{\varepsilon}(p_i^0)\right)$.

By \cite[Theorem 2.4]{HKWX} we have the following relation between the tangent map parameter at the $i$th puncture and the difference of logarithmic angle defects associated with the two neighboring axis rods
\begin{equation}\label{9jnhininkankhhj}
b_i(\mathbf{z})= \tanh\left(\frac{\mathbf{b}_{i+1}(\mathbf{z})-\mathbf{b}_i(\mathbf{z})}{2}\right),
\end{equation}
which follows from \eqref{fffoanoinilnhj}. Moreover, \cite[(2.12)]{HKWX} shows that the logarithmic angle defect of an axis rod agrees with the spacetime metric coefficient $\boldsymbol{\alpha}_{\mathbf{z}}$ of \eqref{alphaeq} restricted to the rod; this value is constant along the rod. Thus, if $\gamma_i$ is a semi-circle in the $\rho z$-half plane centered at puncture $p_i^0$ of radius $2\varepsilon$ and connecting points $\tilde{p}_i$, $\tilde{p}_{i+1}$ of the neighboring axes $\Gamma_i$, $\Gamma_{i+1}$, then
\begin{equation}\label{poniaonoinhh}
\mathbf{b}_{i+1}(\mathbf{z})-\mathbf{b}_i(\mathbf{z}) =\boldsymbol{\alpha}_{\mathbf{z}}(\tilde{p}_{i+1})-\boldsymbol{\alpha}_{\mathbf{z}}(\tilde{p}_i)=\int_{\gamma_i} d\boldsymbol{\alpha}_{\mathbf{z}}=:\mathcal{I}_i(\mathbf{z}).
\end{equation}
Note that $\varepsilon$ may be chosen small enough so that the curves $\gamma_i$ enclose no other initial puncture except $p_i^0$.
It suffices then to show that $\mathcal{I}_i(\mathbf{z})$ is continuously differentiable for $\mathbf{z}\in B_{\varepsilon}^N(\mathbf{z}_0)$.

In what follows the dependence on $\mathbf{z}$ will be suppressed for clarity,
and a dot will be used to denote partial differentiation with respect to the $\mathbf{z}$-parameters. Observe that by formally differentiating the expression \eqref{poniaonoinhh} we obtain
\begin{align}\label{fnoinainoignoinqh}
\begin{split}
\frac{1}{2}\dot{\mathcal{I}}_i(\mathbf{z})&=\int_{\gamma_i}\!\rho\left[\partial_{\rho}U\partial_{\rho}\dot{U}\!-\!\partial_z U\partial_z \dot{U}
\!+\!2\dot{U} e^{4u}\left((\partial_{\rho}v)^2 \!-\!(\partial_z v)^2 \right)\!+\!e^{4u}\left(\partial_{\rho} v\partial_{\rho}\dot{v}
\!-\!\partial_z v\partial_{z}\dot{v}\right)\right]d\rho\\
&+\int_{\gamma_i}\rho\left[\partial_{\rho} \dot{U} \partial_z U +\partial_{\rho}U \partial_z \dot{U}
+4\dot{U} e^{4u}\partial_{\rho}v \partial_{z}v +e^{4u}\left(\partial_{\rho}\dot{v} \partial_z v
+\partial_{\rho} v \partial_z \dot{v}\right)\right]dz,
\end{split}
\end{align}
with the help of \eqref{alphaeq}. Although $e^{4u}$ blows-up upon approach to the axis $\Gamma$, this formula is valid. To see this, let $s\in[0,1]$ parameterize the curve $\gamma_i$ and denote by $h_{\mathbf{z}}(s)$ the integrand of $\mathcal{I}_i(\mathbf{z})$, then the following properties are satisfied. First, it is clear that $h_{\mathbf{z}}$ is an integrable function of $s$ for each $\mathbf{z}\in B_{\varepsilon}^N(p_0)$. Secondly, for almost all $s\in[0,1]$ the partial derivative $\dot{h}_{\mathbf{z}}$ exists for all $\mathbf{z}\in B_{\varepsilon}^N(\mathbf{z}_0)$, by virtue of Theorem \ref{thm:smoothness}; in fact, this holds for all $s$ except at the end points $s=0,1$ and the expression for $\dot{h}_{\mathbf{z}}$ is as in \eqref{fnoinainoignoinqh}.
Thirdly, there is an integrable function $\bar{h}$ of $s$ such that $|\dot{h}_{\mathbf{z}}(s)|\leq \bar{h}(s)$ for all $\mathbf{z}\in B_{\varepsilon}^N(\mathbf{z}_0)$ and $s\in(0,1)$. The function $\bar{h}$ may be taken to be an appropriate constant away from the endpoints, and near $s=0,1$ it suffices to choose $\bar{h}(s)\sim \rho(s)^{-1/2}$. The reason such an $\bar{h}$ dominates is due to Theorem \ref{thm:smoothness} which implies that $\dot{U}$, $\nabla \dot{U}$, and $\nabla\dot{v}$ are continuous on $[0,1] \times B_{\varepsilon}^{N}(\mathbf{z}_0)$, as well as \cite[Theorem 1.1]{LT} which yields the same for $U$, $\nabla U$, and $\nabla v$ in addition to the asymptotics $v=c_i +O_1(\rho^{7/2})$ as $\rho\rightarrow 0$ near $\tilde{p}_i$ (similarly for $\tilde{p}_{i+1}$). In particular, the only terms within \eqref{fnoinainoignoinqh} which potentially blow-up upon approach to $\Gamma$, admit the estimate
\begin{equation}
\rho e^{4u}\left|\partial_{\rho} v\partial_{\rho}\dot{v}
-\partial_z v\partial_{z}\dot{v}\right|\leq C\rho^{-1/2}
\end{equation}
where the constant $C$ is independent of $\mathbf{z}\in B_{\varepsilon}^{N}(\mathbf{z}_0)$, showing that $\bar{h}$ bounds the integrand as desired. We may now 
differentiate under the integral and establish \eqref{fnoinainoignoinqh}. Furthermore, similar arguments show that $\dot{\mathcal{I}}_i(\mathbf{z})$ is continuous on the $\varepsilon$-ball. It then follows from \eqref{9jnhininkankhhj} and \eqref{poniaonoinhh} that $b_i(\mathbf{z})$ is continuously differentiable as long as the components of $\mathbf{z}$ remain distinct.

Standard ODE theory now yields local in time existence and uniqueness of a $C^2$ solution $\mathbf{z}(t)$, in which the components remain pairwise distinct. If $T\in (0,\infty)$ denotes the maximal time of existence, and neither of the behaviors in \eqref{faoijofinoihyhj10} occur, then the functions $b_i(\mathbf{z}(t))$ are uniformly bounded on $[0,T]$. It follows that the solution may be smoothly extended to $t=T$, and local existence may be applied once more to contradict the maximality of $T$. Thus, we conclude that either the solution exists for all time or the maximal time of existence may be characterized by \eqref{faoijofinoihyhj10}.
\end{proof}

We will now proceed to show that the renormalized energy $\mathcal{E}(\Phi_t)$ is monotonic along the flow \eqref{flowdef}. This will require several preliminary lemmas. Let $\epsilon,\delta,\lambda>0$ be parameters with $\delta\ll\epsilon$ and define domains of the form
\begin{equation}
\Omega_{\epsilon,\delta,\lambda}^t=\{(\rho,z,\phi)\in \mathbb{R}^3 \mid \rho>\delta, \, r_i>\epsilon, \, i=1,\dots,N\}\cap B_{\lambda},
\end{equation}
where $r_i$ is the Euclidean distance to the puncture $p_i(t)$. This domain is the region inside the ball of radius $r=\lambda$ centered at the origin, outside the cylinder of radius $\delta$ around the $z$-axis, and outside the balls of radius $\epsilon$ centered at the punctures $p_i(t)$. The boundary of this region is the union of a large sphere with discs removed at the north and south poles $S_{\lambda,\delta}$, a collection of cylinders $C_{\delta,\epsilon,\lambda,i}^t$, $i=1,\dots,N+1$ in one-to-one correspondence with axis rods, and small spheres with discs removed at the north and south poles $S_{\epsilon,\delta,i}^t$, $i=1,\dots,N$. Figure~\ref{domain} provides a pictorial representation that is projected onto the $\rho z$-half plane and reflected across the axis. Note that $C_{\delta,\epsilon,\lambda,i}^t$ is independent of $\lambda$ when $i=2,\dots,N$. Consider the renormalized harmonic map energy restricted to these domains, denoted $\mathcal{E}_{\Omega^t_{\epsilon,\delta,\lambda}}(\Phi_t)$. The strategy will be to differentiate this quantity with respect to $t$, and then integrate by parts thereby reducing the problem to an analysis of boundary terms. Although a special class of exhaustion domains is used for this purpose, which is chosen to simplify computations, the main result (Theorem \ref{decreaselaw}) does not depend on this choice since the renormalized energy density is integrable on $\mathbb{R}^3$.

The movement of $\Omega_{\epsilon,\delta,\lambda}^t$ in time may be obtained from the action of a 1-parameter family of diffeomorphisms applied to a fixed initial domain, such that the associated flow vector field takes the form $Z=h\partial_z$ where $h$ is a smooth function on $\mathbb{R}^3$ which vanishes except near the punctures $p_i(t)$, and in a sufficiently small neighborhood of these points $h=\dot{z}_i(t)$. According to Corollary \ref{cor:linest} and Proposition \ref{blipschitz}, the 1-parameter family of harmonic maps $\Phi_t =(u_t,v_t)$ is continuously differentiable in $t$ away from the axis, and $\dot{\Phi}_t \in C^2(\Omega^t_{\epsilon,\delta,\lambda},\mathbb{R}^2)$. Since the domains are bounded and removed from the axis, we may differentiate under the integral 
and utilize the harmonic map equations to find
\begin{align}\label{4.7}
\begin{split}
\frac{d}{dt}\mathcal E_{\Omega^t_{\epsilon,\delta,\lambda}}(\Phi_t)
&= \int_{\Omega^t_{\epsilon,\delta,\lambda}} \left(2\nabla \dot{U}_t \cdot\nabla U_t + 4 \frac{e^{4U_t}}{\rho^4}\dot{U}_t |\nabla v_t|^2  +2\frac{e^{4U_t}}{\rho^4}\nabla \dot{v}_t \cdot\nabla v_t\right)d\mathbf{x}\\
&\quad +\int_{\partial\Omega^t_{\epsilon,\delta,\lambda}}\left(|\nabla U_t|^2 +\frac{e^{4U_t}}{\rho^4}|\nabla v_t|^2\right)
\iota_{Z}d\mathbf{x}\\
&= \int_{\partial\Omega^t_{\epsilon,\delta,\lambda}} \!\!\!\!\!\!\!2\left(\dot U_t\,\nabla_\nu U_t
+ \frac{e^{4U_t}}{\rho^{4}} \dot v_t \nabla_\nu v_t \right)dA
+\int_{\partial\Omega^t_{\epsilon,\delta,\lambda}}\!\!\!\!\left(|\nabla U_t|^2 +\frac{e^{4U_t}}{\rho^4}|\nabla v_t|^2\right)
\iota_{Z}d\mathbf{x}\\
&= \underbrace{\int_{S_{\lambda,\delta}}}_{\mathbf{I}_1(\lambda,\delta)} + \sum_{i=1}^{N+1}\underbrace{ \int_{C^t_{\delta,\epsilon,\lambda,i}}}_{\mathbf{I}^t_2(\delta,\epsilon,\lambda,i)}
+ \sum_{i=1}^N\underbrace{\int_{S^t_{\epsilon,\delta,i}}}_{\mathbf{I}^t_3(\epsilon,\delta,i)} 2\left(\dot U_t\,\nabla_\nu U_t
+ \frac{e^{4U_t}}{\rho^{4}} \dot v_t \nabla_\nu v_t \right) dA\\
&\quad +\sum_{i=1}^N\underbrace{\int_{S^t_{\epsilon,\delta,i}}}_{\mathbf{I}^t_4(\epsilon,\delta,i)} \left(|\nabla U_t|^2 +\frac{e^{4U_t}}{\rho^4}|\nabla v_t|^2\right)\iota_{Z}d\mathbf{x},
\end{split}
\end{align}
where $\nu$ is the unit outer normal to the boundary, and $\iota_{Z}$ denotes interior product. In the next four lemmas, exhaustion limits for each of the integrals $\mathbf{I}_1$, $\mathbf{I}^t_2$, $\mathbf{I}^t_3$, and $\mathbf{I}_4^t$ will be evaluated.

\begin{figure}
\includegraphics[width=8.5cm]{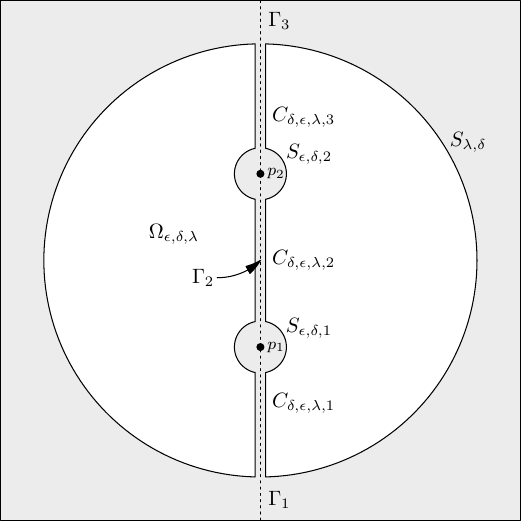}
\caption{Domain of integration and boundary components.}\label{domain}
\end{figure}

\begin{lem} \label{I1}
For each $\delta>0$ we have $\lim_{\lambda\rightarrow\infty} \mathbf{I}_1(\lambda,\delta)=0$.
\end{lem}

\begin{proof}
It suffices to show that the integrand is $o(r^{-2})$. By~\cite[Theorem~2.3]{HKWX} and \eqref{est} of Corollary \ref{cor:linest}, there exists a constant $C$ such that for all $r$ large enough
\begin{equation}
|U_t|\leq C,\qquad |\nabla U_t|+ \frac{|\nabla v_t|}{\rho^{2}} \leq \frac{C}{r^{2}},\qquad |\dot{U}_t|+\frac{|\dot{v}_t|}{\rho^{2}} \leq \frac{C}{r},
\end{equation}
from which the desired result follows.
\end{proof}

\begin{lem}\label{I2}
For each $\epsilon>0$ and $i=1,\dots, N+1$ we have
\begin{equation}
\lim_{\delta\rightarrow 0}\lim_{\lambda\rightarrow\infty} \mathbf{I}^t_2(\delta,\epsilon,\lambda,i)=0.
\end{equation}
\end{lem}

\begin{proof}
First consider the case $i=2,\dots,N$, in which $\mathbf{I}^t_2(\delta,\epsilon,\lambda,i)=\mathbf{I}^t_2(\delta,\epsilon,i)$ is independent of $\lambda$.
Since the areas of the relevant cylinders $C^t_{\delta,\epsilon,\lambda,i}$ are $O(\delta)$ as $\delta\rightarrow 0$, it suffices to show that the integrands are $O(1)$. For fixed $\epsilon>0$, observe that \cite[Theorem~2.1]{HKWX}, \cite[Theorem 1.1]{LT}, and \eqref{estrho} of Corollary \ref{cor:linest} imply the existence of a constant $C$ such that
\begin{equation}
|U_t|\leq C,\qquad |\nabla U_t| + \frac{|\nabla v_t|}{\rho^2}\leq C,\qquad |\dot{U}_t|+\frac{|\dot{v}_t|}{\rho^2}\leq C,
\end{equation}
yielding the desired result.

Consider now the case when $i=1,N+1$. Since the areas of the relevant cylinders $C^t_{\delta,\epsilon,\lambda,i}$ are $O(\delta \lambda)$, it suffices to show that the integrands are $O(r^{-2})$. For fixed $\epsilon>0$, observe that \cite[Theorem~2.3]{HKWX} and \eqref{est} of Corollary \ref{cor:linest} imply that for all $r$ large enough
\begin{equation}
|U_t|\leq C,\qquad |\nabla U_t| + \frac{|\nabla v_t|}{\rho^2}\leq \frac{C}{r^2},\qquad |\dot{U}_t|+\frac{|\dot{v}_t|}{\rho^2}\leq \frac{C}{r},
\end{equation}
where $C$ is independent of $\delta$. The desired result now follows.
\end{proof}

\begin{lem}\label{I3}
For each $i=1,\dots, N$ we have
\begin{equation}
\lim_{\epsilon\to0}\lim_{\delta\rightarrow 0} \mathbf{I}^t_3(\epsilon,\delta,i) = -f_3(b_i)b_i,
\end{equation}
where $f_3(b)$ is defined by~\eqref{f(b)}.
\end{lem}

\begin{proof}
According to \cite[Theorem 2.1 and 2.2]{HKWX}, expansions for the harmonic map $\Phi_t$ are given by \eqref{3.1}-\eqref{3.3} in a neighborhood of each puncture.
Next observe that by Proposition \ref{blipschitz} we have $\dot{b}_i \partial_{b_i}\bar{U}_t =O(1)$ and $\dot{b}_i\partial_{b_i}\bar{v}_t =O(1)$ as $r_i\rightarrow 0$. Thus, \eqref{abcd} of Corollary \ref{cor:linest} implies that on $S^t_{\epsilon,\delta,i}$ the following asymptotic profiles hold
\begin{align}\label{oniaoinoinhjj}
\begin{split}
\dot{U}_t \partial_{r_i}U_t &= \left[\left(-\frac{(z-z_i)}{r_i^2} + \frac{\rho}{r_i^2}\partial_{\theta_i}\bar{U}_t\right)\dot{z}_i+\dot{b}_i\partial_{b_i}\bar{U}_t
+O(r_i^{\beta_i -1})\right]\left[\frac{1}{r_i}+ O(r_{i}^{\beta_i -1}) \right]\\
&= -\frac{(z-z_i)}{r_{i}^3}\dot{z}_i+\frac{\rho}{r_i^3}(\partial_{\theta_i}\bar{U}_t)\dot{z}_i +O(r_{i}^{\beta_i -2}),
\end{split}
\end{align}
\begin{align}
\begin{split}
\frac{e^{4U_t}}{\rho^4} \dot{v}_t \partial_{r_i}v_t &=
\frac{1}{\sin^4\theta_{i}}\left[\left(\frac{1+\cos^{2}\theta_i +2b_i \cos\theta_i}{2a_i \sqrt{1-b_i^2}}\right)^2 +O(r_i^{\beta_i})\right]\\
&\cdot\left[\frac{\sin\theta_i}{r_i} (\partial_{\theta_i} \bar{v}_t)\dot{z}_i+\dot{b}_i \partial_{{b}_i}\bar{v}_t +O(r_{i}^{\beta_i -1}\sin\theta_{i})\right]
\underbrace{O(r_{i}^{\beta_i -1} (\sin\theta_i)^{7/2} )}_{\partial_{r_i} v_t}\\
&=O((\sin\theta_{i})^{-1/2}r_{i}^{\beta_i -2}).
\end{split}
\end{align}
We now let $\delta\rightarrow 0$ and integrate over the spheres $S^t_{\epsilon,i}$ on which $r_i=\epsilon$, and note that the first term on the right-hand side of \eqref{oniaoinoinhjj} integrates to zero due to symmetry. Since
\begin{equation}\label{poijaoinfoianoinionhhh}
\partial_{\theta_i}\bar{U}=-\frac{\sin\theta_i\, (\cos\theta_i + b_i)}{1+\cos^2\theta_i + 2b_i\cos\theta_i}
\end{equation}
by \eqref{eq:tange-map}, it follows that
\begin{align}
\begin{split}
\lim_{\delta\rightarrow 0}\int_{S^t_{\epsilon,\delta,i}} 2\left(\dot{U}_t \nabla_{\nu}U_t +\frac{e^{4U_t}}{\rho^4} \dot{v}_t \nabla_{\nu}v_t\right)dA&=-4\pi \dot{z}_i \int_{0}^{\pi}(\partial_{\theta_i}\bar{U}) \sin^2 \theta_i d\theta_i +O(\epsilon^{\beta_i})\\
&=-f_3(b_i)b_i+O(\epsilon^{\beta_i}),
\end{split}
\end{align}
where $f_3$ is defined in Lemma \ref{flemma} and we have used \eqref{flowdef}. Letting $\epsilon\rightarrow 0$ now produces the desired result.
\end{proof}

\begin{lem}\label{I4}
For each $i=1,\dots, N$ we have
\begin{equation}
\lim_{\epsilon\to0}\lim_{\delta\rightarrow 0} \mathbf{I}^t_4(\epsilon,\delta,i) = -f_4(b_i)b_i,
\end{equation}
where $f_4(b)$ is defined by~\eqref{f(b)}.
\end{lem}

\begin{proof}
Observe that on $S^t_{\epsilon,\delta,i}$ a direct computation produces
\begin{equation}
\iota_{Z}d\mathbf{x}=\dot{z}_i dx\wedge dy=-\dot{z}_i r_i^2 \sin\theta_i \cos\theta_i d\phi_i \wedge d\theta_i.
\end{equation}
Furthermore, \eqref{3.1}-\eqref{3.3} imply that in a neighborhood of the $i$th puncture
\begin{align}
\begin{split}
|\nabla U_t|^2 &=(\partial_{r_i} U_t)^2 +r_{i}^{-2} (\partial_{\theta_i}U_t)^2\\
&=\left(r_{i}^{-1}+O(r_{i}^{\beta_i -1})\right)^2 +r_i^{-2}\left(\partial_{\theta_i}\bar{U}_t +O(r_i^{\beta_i})\right)^2\\
&=r_i^{-2}\left(1+(\partial_{\theta_i}\bar{U}_t)^2\right)+O(r_i^{\beta_i -2}),
\end{split}
\end{align}
and
\begin{align}
\begin{split}
\frac{e^{4U_t}}{\rho^4}|\nabla v_t|^2 &=e^{4u_t}\left((\partial_{r_i} v_t)^2 +r_i^{-2}(\partial_{\theta_i}v_t)^2 \right)\\
&=\left[e^{4\bar{u}_t}+O((\sin\theta_i)^{-4}r_i^{\beta_i})\right]\\
&\quad \cdot\left[O((\sin\theta_i)^{3+\varsigma}r_i^{\beta_i -1})^2 +r_i^{-2}\left(\partial_{\theta_i}\bar{v}_t 
+O((\sin\theta_i)^{2+\varsigma}r_i^{\beta_i})\right)^2 \right]\\
&=r_i^{-2}e^{4\bar{u}_t}(\partial_{\theta_i}\bar{v}_t)^2 +O(r_i^{\beta_i -2}).
\end{split}
\end{align}
Therefore, integrating over the surface with respect to the induced boundary orientation yields
\begin{align}
\begin{split}
&\lim_{\delta\rightarrow 0}\int_{S^t_{\epsilon,\delta,i}}\left(|\nabla U_t|^2 +\frac{e^{4U_t}}{\rho^4}|\nabla v_t|^2\right)\iota_{Z}d\mathbf{x}\\
=&-2\pi \dot{z}_i \int_{0}^{\pi}\left(1+(\partial_{\theta_i}\bar{U})^2 +e^{4\bar{u}}(\partial_{\theta_i}\bar{v})^2 \right)\sin\theta_i \cos\theta_i d\theta_i +O(\epsilon^{\beta_i})\\
=&- f_4(b_i) b_i +O(\epsilon^{\beta_i}).
\end{split}
\end{align}
In the last step we used \eqref{poijaoinfoianoinionhhh}, together with \eqref{eq:tange-map}, and the relation $e^{4\bar{u}}\partial_{\theta_i}\bar{v}=-(2a_i \sin\theta_i)^{-1}$ from \cite[(4.5)]{HKWX},  to obtain
\begin{equation}
1+(\partial_{\theta_i}\bar{U})^2 +e^{4\bar{u}}(\partial_{\theta_i}\bar{v})^2=\frac{2(1+b_i \cos\theta_i)}{1+\cos^2\theta_i +2b_i \cos\theta_i}.
\end{equation}
Letting $\epsilon\rightarrow 0$ now produces the desired result.
\end{proof}

\begin{rem} \label{uniform}
We point out that the estimates in the proofs of the preceding three lemmas are uniform in $t$, for $t$ in a sufficiently small neighborhood of any given value $t_0$ for which $\mathbf{z}(t_0)$ has pairwise distinct components.
\end{rem}

\begin{lem}\label{flemma}
Let
\begin{equation} \label{f(b)}
f_3(b) = 4\pi \int_0^\pi \frac{\sin^3\theta\, (\cos\theta + b)}{1+\cos^2\theta + 2b\cos\theta} d\theta, \quad\quad
f_{4}(b)=-4\pi\int_{0}^{\pi}\frac{(1+b\cos\theta)\sin\theta\cos\theta}{1+\cos^2 \theta+2b\cos\theta} d\theta.
\end{equation}
Then 
\begin{equation}
f(b):= (f_3+f_4)(b)=4\pi \operatorname{arctanh} b
\end{equation}
for all $b\in (-1,1)$. 
\end{lem}

\begin{proof}
Using \(\sin^2\theta=1-\cos^2\theta\), the numerator of the $f$ integrand may be
written as
\begin{equation}
4\pi\sin\theta
\left(
b-\cos^3\theta-2b\cos^2\theta
\right).
\end{equation}
Setting \(u=\cos\theta\) therefore gives
\begin{equation}
\frac{f}{4\pi}=
\int_{-1}^1
\frac{b-u^3-2bu^2}{1+u^2+2bu}\,du  
=
\int_{-1}^1
\left(
-u+\frac{u+b}{u^2+2bu+1}
\right)\,du.
\end{equation}
A simple evaluation yields
\begin{equation}
\frac{f}{4\pi}
=
\frac12
\left[
\log\left(u^2+2bu+1\right)
\right]_{-1}^{1}
=
\frac12\log\left(\frac{1+b}{1-b}\right)=\operatorname{arctanh} b. \qedhere
\end{equation}
\end{proof}

We are now ready to establish the main result of this section. We will utilize an elementary fact concerning the interchange of derivatives and limits. Recall that if $\{h_n\}$ is a sequence of continuously differentiable functions on a bounded interval $I$ such that $\{h_{n}(t_*)\}$ converges for some $t_* \in I$, and $\{h'_{n}\}$ is uniformly convergent on $I$, then $\{h_n\}$ converges uniformly on the same interval to a continuously differentiable function $h$ with $h'(t)=\lim_{n\rightarrow\infty}h'_{n}(t)$ for all $t\in I$. 

\begin{thm}\label{decreaselaw}
Let $\mathbf{z}(t)$ be a solution of the puncture flow \eqref{flowdef} for $t\in[0,T)$ provided by Proposition \ref{blipschitz},
and consider the 1-parameter family of harmonic maps $\Phi_t=\Phi_{\mathbf{z}(t)}$. Then the flow of renormanlized energy $\mathcal{E}(\Phi_t)$ is continuously differentiable and monotonically nonincreasing on $[0,T)$.
\end{thm}


\begin{proof}
Set $\Omega^t_{\epsilon,\delta}=\Omega^t_{\epsilon,\delta,\infty}$, and observe that since the renormalized energy densities of the harmonic maps are globally integrable, we have that $\mathcal{E}_{\Omega^t_{\epsilon,\delta,\lambda}}(\Phi_t)\rightarrow\mathcal{E}_{\Omega^t_{\epsilon,\delta}}(\Phi_t)$ as $\lambda\rightarrow\infty$ for all $t\in [0,T)$. Then with the aid of Remark \ref{uniform} and a standard calculus result \cite[Theorem 7.17]{Rudin}, upon letting $\lambda\rightarrow\infty$ in
\eqref{4.7} we may interchange derivative and limit and apply Lemma \ref{I1} to find that
$\mathcal{E}_{\Omega^t_{\epsilon,\delta}}(\Phi_t)$ is continuously differentiable in $t$ with
\begin{equation}
\frac{d}{dt}\mathcal E_{\Omega^t_{\epsilon,\delta}}(\Phi_t)=\lim_{\lambda\to\infty}\frac{d}{dt}\mathcal E_{\Omega^t_{\epsilon,\delta,\lambda}}(\Phi_t)=\lim_{\lambda\rightarrow\infty}\sum_{i=1}^{N+1} \mathbf{I}^t_2(\delta,\epsilon,\lambda,i) + \sum_{i=1}^N \mathbf{I}^t_3(\epsilon,\delta,i)
+ \sum_{i=1}^N \mathbf{I}^t_4(\epsilon,\delta,i).
\end{equation}
Next set $\Omega^t_{\epsilon}=\Omega^t_{\epsilon,0}$ and note that in a similar manner, using now Lemma~\ref{I2}, we have that $\mathcal{E}_{\Omega^t_{\epsilon}}(\Phi_t)$ is continuously differentiable in $t$ with
\begin{equation}
\frac{d}{dt}\mathcal E_{\Omega^t_{\epsilon}}(\Phi_t)=\lim_{\delta\to 0}\frac{d}{dt}\mathcal E_{\Omega^t_{\epsilon,\delta}}(\Phi_t)=\lim_{\delta\rightarrow 0}\sum_{i=1}^N \left(\mathbf{I}^t_3(\epsilon,\delta,i) +\mathbf{I}^t_4(\epsilon,\delta,i)\right).
\end{equation}
Finally, interchanging derivative and limit one more time while letting $\epsilon\to0$ shows that $\mathcal{E}(\Phi_t)$ is continuously differentiable in $t$, and employing Lemmas \ref{I3} and \ref{I4} yields
\begin{equation}\label{derivative}
\frac{d}{dt}\mathcal E(\Phi_t)=\lim_{\epsilon\to0}\frac{d}{dt}\mathcal E_{\Omega^t_{\epsilon}}(\Phi_t)
=\lim_{\epsilon\to0} \lim_{\delta\rightarrow 0} \sum_{i=1}^N \left(\mathbf{I}^t_3(\epsilon,\delta,i) + \mathbf{I}^t_4(\epsilon,\delta,i)\right)=
-\sum_{i=1}^N f(b_i)\, b_i \leq 0,
\end{equation}
where the last inequality follows from Lemma~\ref{flemma}.
\end{proof}


\section{Collision of Punctures}
\label{collision}
\setcounter{equation}{0}
\setcounter{section}{5}

Consider the 1-parameter family of harmonic maps $\Phi_t =(u_t,v_t):=(u_{\mathbf{z}(t)},v_{\mathbf{z}(t)})$ with fixed potential constants, defined by the flow of punctures \eqref{flowdef}. Let $z_i(t)$, $1\leq i \leq N$ denote the $z$-coordinate of the $i$th puncture at time $t$, and assume that $z_1(0)<\cdots<z_N(0)$. According to Theorem \ref{thm:smoothness} and Proposition \ref{blipschitz}, this family of maps is continuously differentiable in time away from the axis for $t\in[0,T)$. At time $T$ (which may be infinite), at least one of the possibilities from \eqref{faoijofinoihyhj} occurs. In the latter case, namely when
$\liminf_{t\to T}\left(z_{j+1}(t)-z_j(t)\right)=0$ for some $j\in \{1,\ldots, N-1 \}$, we say that two (or more) adjacent punctures $z_j<z_{j+1}$ \textit{collide} at time $T$. Note that more than two punctures can collide at time $T$ in a single location, or in multiple locations simultaneously. In this section we will treat the case in which no scattering occurs, that is, when the first possibility of \eqref{faoijofinoihyhj} does not happen. With this setting, after possibly translating in the $z$-direction, we may assume without loss of generality that the set of punctures remains within a fixed compact set for all time.
Hence, there is a sequence of times $t_n\to T$ such that $\lim_{n \rightarrow\infty} z_i(t_n)=z_{i}^*$ for all $i$, and $z_j(t_n)$ eventually collides with $z_{j+1}(t_n)$ for some $j$. Note that the neighboring $j$th punctures collide if and only if
$z_{j}^*=z_{j+1}^*$, in which case we call $z_{j}^*$ a \textit{collision location}. 
The collection of values $\{z_{1}^*,\ldots, z_{N}^* \}$, which has at most $N-1$ distinct elements, is referred to as the \textit{collision configuration}. Consider a harmonic map $\Phi_*$ having the set of punctures given by $\{z_{1}^*,\ldots, z_{N}^* \}$ and with the same potential constants on the rods that remain; this will be referred to as the \textit{collision configuration map}. We point out that different sequences $t_n\to T$ may result in different collision configurations and collision configuration maps. Nevertheless, the limiting renormalized energy is the same for all sequences, that is $\lim_{t\to T} \mathcal E(\Phi_{t})$ exists in light of the monotonicity provided by Theorem \ref{decreaselaw}. The proposition below contains the primary result of this section, which will allow the flow to `pass through' the maximal time of existence when a collision happens, while maintaining the desired monotonicity property with the use of collision configuration maps. 

\begin{prop}\label{prop:collision}
Let $\Phi_t$ be the 1-parameter family of harmonic maps with fixed potential constants defined by the flow of punctures \eqref{flowdef}, in which a collision occurs at the limiting time $T$, with no scattering. If $\Phi_*$ is a collision configuration map described above, then
\begin{equation} \label{2-collision-limit}
\mathcal E(\Phi_*) \le \lim_{t\to T} \mathcal E(\Phi_{t}).
\end{equation}
\end{prop}


\begin{proof} 
Let $\{t_n\}$ be a sequence of times with $t_n\to T$, which leads to a collision configuration $\{z_{1}^*,\ldots, z_{N}^* \}$ and the associated collision configuration map $\Phi_*$. For convenience denote $\Phi_n =\Phi_{t_n}$. We claim that 
\begin{equation}\label{2-collision-convergence-Phi}
\Phi_n\to\Phi_* \text{ in $C^3$ away from the axis},
\end{equation} 
and that for any $\epsilon>0$ there is an $n_\epsilon$ large enough such that the renormalized energies satisfy
\begin{equation}\label{2-collision}
\mathcal E(\Phi_*) < \mathcal E(\Phi_{n_\epsilon})+\epsilon.
\end{equation}
A simple limit as $\epsilon\to 0$ in \eqref{2-collision} then yields \eqref{2-collision-limit}.
The rest of the proof is devoted to the proof of \eqref{2-collision-convergence-Phi} and \eqref{2-collision}. 

For each $n$, consider a harmonic map $\Psi_n$ having the same punctures as $\Phi_n$ except for the colliding punctures which are replaced by the corresponding set of collision locations, and having the same potential constants on the rods that remain. The maps $\Psi_n$ are called \textit{auxiliary collision configuration maps}.
The punctures at collision locations within $\Psi_n$ have absorbed all the colliding punctures from $\Phi_n$, and it follows from \eqref{angpotconst} that the resulting angular momentum at any collision location is the sum of all the angular momenta from punctures colliding at that point. With $\Phi_*$ given as the harmonic map associated with the collision configuration, then observe that Theorem \ref{thm:smoothness} implies 
\begin{equation}\label{2-collision-energy-convergence}
\Psi_n\to\Phi_* \text{ in $C^3$ away from the axis},
\end{equation} 
since no collisions occur in the sequence $\Psi_n$. 

We will construct a \emph{model map} $\Xi_n$ for $\Phi_n$ which coincides with $\Psi_n$ outside a fixed compact set. This will lead to an upper bound on the hyperbolic distance $d_{\H^2}(\Phi_n,\Psi_n)$ which is independent of $n$, and will allow us to argue that $\Phi_n\to\Phi_{*}$ uniformly on compact subsets of $\mathbb{R}^3 \setminus\Gamma$. Recall that a model map \cite[Definition 4.1]{KKRW} for $\Phi$ is a map $\Xi\colon \R^3\setminus\Gamma\to\H^2$ which is \emph{asymptotic} to $\Phi$ in the sense that $d_{\H^2}(\Phi,\Xi)$ is bounded and tends to zero at infinity, and such that its \textit{tension} $\tau(\Xi)=\operatorname{Tr}\pmb{\nabla} d\,\Xi$ is bounded and satisfies $|\tau(\Xi)|\leq-\Delta w$, for some positive function $w\in C^2(\mathbb{R}^3)$ which tends to zero at infinity. In order not to encumber the notation, we will omit the subscript $n$ until the $L^\infty$-bound has been obtained. Here $\pmb\nabla$ denotes the covariant derivative on the pullback bundle $T^*(\R^3\setminus\Gamma) \otimes \Xi^{-1} T\H^2$.

The construction of the model map, and the estimation of its tension will proceed in three cases. We will first consider the case where there is only one collision, and only two colliding punctures. We will then consider the case where there is one collision location but possibly more than two colliding punctures. Finally, we will treat the general case.\smallskip

\textit{Case 1.} Consider the situation in which only two punctures collide, and denote for convenience the two colliding punctures by $p_1$, $p_2$ with $z$-coordinates $z_1<z_2$. Without loss of generality it may be assumed that the midpoint between these two punctures occurs at the origin. Set $\eta=(z_2-z_1)/4$ and let $B_i=B_\eta(p_i)$ be the ball of radius $\eta$ centered at $p_i$; clearly $B_1\cap B_2=\emptyset$. For any $\delta>0$, $n$ may be taken large enough and hence $\eta$ small enough to guarantee that the union of these two balls lies within the $\delta/4$-ball centered at the origin, 
and to guarantee that the collision location lies within $B_{\delta/4}$.  We define the model map by
\begin{equation}
\Xi = (u,v) = \chi_0\Theta_0 + \chi_1\Theta_1 + \chi_2\Theta_2,
\end{equation}
where $\Theta_i=(u_i,v_i)$, $i=1,2$ are the corresponding extreme Kerr maps associated with each $p_i$, and we have denoted the extreme multi-Kerr auxiliary collision configuration map $\Psi$ by $\Theta_0=(u_0,v_0)$ for uniformity. Furthermore, $0\leq\chi_i\leq1$ are cut-off functions described as follows. Let $(r,\theta,\phi)$ be polar coordinates centered at the origin, and take $0\leq \tilde\chi_1,\tilde\chi_2\leq 1$ to be functions of $\theta$ alone such that $\tilde\chi_1+\tilde\chi_2=1$, and $\tilde\chi_i=1$ in $B_i$. For instance these may be chosen so that $\tilde\chi_1=0$ for $0\leq\theta\leq\pi/4$, and $\tilde\chi_1=1$ for $3\pi/4\leq\theta\leq\pi$. Then take $\chi_0$ to be a nonnegative function of $r$ alone such that $\chi_0=0$ in $B_{\delta/2}$, and $\chi_0=1$ outside $B_\delta$, and set $\chi_i=(1-\chi_0)\tilde\chi_i$, $i=1,2$. See Figure \ref{domain1} for a diagram of the regions involved as well as other regions detailed below. 

\begin{figure}
\includegraphics[height=9cm]{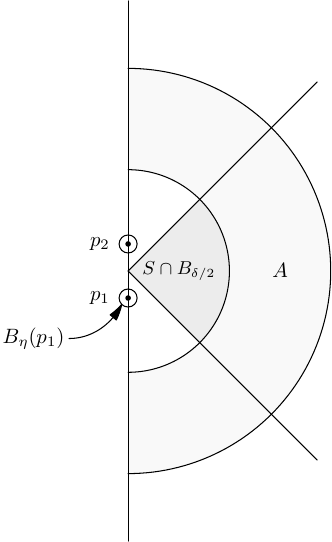}
\caption{Case 1: collision of two punctures.}  \label{domain1}
\end{figure}

First note that the tension $\tau(\Xi)$ vanishes except in an annulus $A$ and the portion of a sector $S\cap B_{\delta/2}$ where
\begin{equation}
S=\{(r,\theta,\phi)\in\mathbb{R}^3 \mid \pi/4<\theta<3\pi/4\},\quad\quad\quad A=\{(r,\theta,\phi)\in\mathbb{R}^3 \mid \delta/2<r<\delta\}.
\end{equation}
In regions where it is nonzero we seek to bound the four quantities labelled by Roman numerals from
\begin{equation} \label{eq:tension-est}
\begin{aligned}
|\tau(\Xi)| &= \sqrt{ |\Delta u - 2e^{4u} |\nabla v|^2 |^2
+ e^{4u} |\Delta v + 4\nabla u \cdot\nabla v|^2} \\
&\leq |\Delta u| + 2e^{4u} |\nabla v|^2 + e^{2u} |\Delta v| + 4e^{2u} |\nabla u \cdot \nabla v| \\
&=:\mathbf{I} + \mathbf{II} + \mathbf{III} + \mathbf{IV},
\end{aligned}
\end{equation}
where the tension norm is with respect to the hyperbolic space metric. We begin by estimating the tension in $S \cap B_{\delta/2}$, term by term. Observe that in this region $\chi_0=0$, hence $\chi_i =\tilde{\chi}_i$ for $i=1,2$ and $\Xi$ is a combination of only $\Theta_1$ and $\Theta_2$. Since $\tilde{\chi}_2=1-\tilde{\chi}_1$ we find that
\begin{equation}\label{I}
\mathbf{I} \leq |\Delta \chi_1| |u_1-u_2| + 2|\nabla \chi_1||\nabla(u_1-u_2)| + \chi_1|\Delta u_1| + \chi_2|\Delta u_2|.
\end{equation}
The first term of \eqref{I} may be estimated by $C/r^2$, where here and in what follows $C$ represents a constant independent of $\eta$ (or rather $n$) and of $\delta$. This follows from $|\Delta \chi_1|\leq C/r^2$ and $|u_1 -u_2|\leq C$, with the latter inequality arising from \eqref{3.1} together with the leading term inequality $|\ln r_1 - \ln r_2|\leq \ln 5$ outside $B_1\cup B_2$ which is shown in Corollary~\ref{cor:r1-r2}; $r_1$ and $r_2$ denote the distances to $p_1$ and $p_2$ respectively.
Similarly, the second term satisfies the same bound since $|\nabla \chi_1|\leq C/r$, and $|\nabla u_1-\nabla u_2|\leq C(1/r_1+1/r_2)\leq C/r$ outside $B_1\cup B_2$. Moreover, \eqref{3.1}-\eqref{3.3} imply that the last two terms satisfy $|\Delta u_i|\leq C/r^2$ outside $B_1\cup B_2$. It follows that $\mathbf{I} \leq C/r^2$. Term $\mathbf{III}$ in \eqref{eq:tension-est} is handled analogously to $\mathbf{I}$ and thus satisfies the same estimate, using the fact that $e^{2u}|v_1-v_2|$, $r e^{2u}|\nabla v_i|$, and $r^2e^{2u}|\Delta v_i|$ are uniformly bounded in $S \cap B_{\delta/2}$ as a consequence of Lemma~\ref{lemma:v} below. 

Consider now term $\mathbf{IV}$ and use the expression
\begin{equation}
\nabla u = (u_1-u_2)\nabla\chi_1 + \chi_1\nabla u_1 + \chi_2\nabla u_2,
\end{equation}
as well as a similar expression for $v$ to find
\begin{align}\label{du*dv}
\begin{split}
\nabla u \cdot \nabla v &=
(u_1-u_2)(v_1-v_2)|\nabla\chi_1|^2 + \chi_1^2 \nabla u_1 \cdot \nabla v_1
+\chi_2^2 \nabla u_2 \cdot \nabla v_2\\
&+(u_1-u_2)(\chi_1\nabla\chi_1 \cdot\nabla v_1 + \chi_2\nabla\chi_1 \cdot\nabla v_2)\\
&+(v_1-v_2)(\chi_1\nabla\chi_1 \cdot\nabla u_1 + \chi_2\nabla\chi_1 \cdot\nabla u_2)\\
&+\chi_1\chi_2(\nabla u_1 \cdot\nabla v_2+\nabla u_2 \cdot\nabla v_1).
\end{split}
\end{align}
When multiplied by $e^{2u}$ each term on the right-hand side of \eqref{du*dv} is bounded by $C/r^2$ in $S\cap B_{\delta/2}$, as can be seen from the uniform boundedness in this region of the following quantities discussed in the previous paragraph: $|u_1-u_2|$, $e^{2u}|v_1-v_2|$, $r |\nabla u_i|$, $r e^{2u}|\nabla v_i|$ for $i=1,2$. Term $\mathbf{II}$ of \eqref{eq:tension-est} may be handled similarly replacing $\nabla u$ with $e^{2u}\nabla v$. Therefore we conclude that
\begin{equation}\label{eq:C/r^2}
|\tau(\Xi)| \leq \frac{C}{r^2} \quad \text{ in } B_{\delta/2}
\end{equation}
for some constant $C$ independent of $\eta$ and $\delta$, since the tension vanishes on $B_{\delta/2}\setminus S$.

Turning now our attention to the annulus $A$, note that away from the axis in this region $\widetilde\Theta=\tilde\chi_1\Theta_1+\tilde\chi_2\Theta_2=(\tilde u,\tilde v)$ is bounded. We can perform a computation similar to \eqref{eq:tension-est}, \eqref{I} utilizing
\begin{equation}
\Xi=\chi_0\Theta_0+(1-\chi_0)\widetilde\Theta
\end{equation}
with $u_1$ replaced by $\tilde u$ and $u_2$ replaced by $u_0$, to obtain quantities analogous to $\mathbf{I}$--$\mathbf{IV}$ which will be denoted with the same notation. The analogue of \eqref{I} becomes
\begin{equation} \label{I.1}
\mathbf{I} \leq |\Delta\chi_0||\tilde u-u_0|  + 2|\nabla\chi_0 \cdot \nabla(\tilde u-u_0)| + (1-\chi_0)|\Delta \tilde u| + \chi_0|\Delta u_0|.
\end{equation}
As a consequence of Lemma~\ref{lem:r0-ri} we have that $|\tilde u - u_0|$ is uniformly bounded outside $B_{\delta/2}$, and moreover $|\Delta\chi_0|\leq C/\delta^2\leq C/r^2$ in $A$. To estimate the gradient term note that \eqref{3.1}-\eqref{3.3} yield
\begin{equation}
\tilde u - u_0=\tilde{\chi}_1 \ln r_1 +\tilde{\chi}_2 \ln r_2 -\ln r +O_1(1),
\end{equation}
and since $\tilde{\chi}_i$ are independent of $r$ while $\chi_0$ is a function of $r$ alone, we obtain
\begin{equation}
|\nabla(\tilde u-u_0)\cdot\nabla\chi_0|\leq \left(|\tilde{\chi}_1 \nabla\ln r_1 +\tilde{\chi}_2 \nabla\ln r_2 -\nabla\ln r| +O(r^{-1})\right)|\nabla\chi_0|\leq \frac{C}{r^2}.
\end{equation}
Furthermore, as indicated above $|\Delta u_i|\leq C/r^2$, $i=1,2$ outside of $B_1 \cup B_2$ and hence on $A$; the same is true also for $i=0$. Therefore the last two terms of \eqref{I.1} admit the inverse square estimate by routine computations. Putting this together then produces $\mathbf{I}\leq C/r^2$ on $A$. As before, term $\mathbf{II}$ admits the desired estimate using a version of \eqref{du*dv}, together with the fact that $e^{2u}|\tilde{v}-v_0|$ and $r e^{2u}(|\nabla \tilde{v}|+|\nabla v_0|)$ are uniformly bounded in $A$ by Lemma~\ref{lemma:v1} below. However, since the estimates for $\Delta v_0$ in this same lemma involve a nonzero $\varsigma$, we cannot directly follow the approach used in $S\cap B_{\delta/2}$ for $\mathbf{III}$. Instead, terms $\mathbf{III}$ and $\mathbf{IV}$ of \eqref{eq:tension-est} will be replaced by $e^{2u}|\Delta v+4\nabla u\cdot\nabla v|$. Observe that a calculation similar to that of \eqref{I.1} in which $u$ is replaced by $v$, combined with the estimates of Lemma~\ref{lemma:v1} yields
\begin{equation}
e^{2u}\Delta v=e^{2u}\chi_0 \Delta v_0 +O\left(\tfrac{1}{r^2}\right)
\end{equation}
in $A$. Moreover a computation, along with \eqref{3.1}-\eqref{3.3}, and again Lemma~\ref{lemma:v1} show that
in this region
\begin{align}
\begin{split}
\nabla u\cdot\nabla v&= \chi_0 \nabla u_0 \cdot\nabla v_0 +(v_0 -\tilde{v})\nabla u_0 \cdot\nabla\chi_0
+(1-\chi_0)\nabla u_0 \cdot\nabla\tilde{v}\\
&+\nabla[(1-\chi_0)(\tilde{U}-U_0)]\cdot\nabla[\tilde{v}+\chi_{0}(v_0 -\tilde{v})]\\
&=\chi_0 \nabla u_0 \cdot\nabla v_0 +O\left(\tfrac{e^{-2u}}{r^2}\right),
\end{split}
\end{align}
where $\tilde{U}=\tilde{u}+\ln\rho$ and $U_0=u_0+\ln\rho$. Therefore since $(u_0,v_0)$ is harmonic we find that
\begin{equation}
e^{2u}|\Delta v+4\nabla u\cdot\nabla v|\leq \frac{C}{r^2},
\end{equation}
where $C$ is independent of $\eta$ and $\delta$. Thus, the tension of $\Xi$ satisfies the same estimate of \eqref{eq:C/r^2} within the annulus $A$.

At this stage notation indicating the dependence on $n$ will be reintroduced. We may summarize what has been established with the following properties of the model map tension
\begin{equation}\label{eq:tension-bound}
|\tau(\Xi_n)|\leq \frac C{r^2} \quad \text{ on }B_{\delta}\setminus\Gamma,
\qquad |\tau(\Xi_n)|=0 \quad \text{ on }(\mathbb{R}^3 \setminus B_{\delta})\setminus\Gamma,
\end{equation}
where $C$ is a constant independent of $n$ and $\delta$. Let 
\begin{equation}
\Lambda(\Phi_n,\Xi_n)=\sqrt{1+d_{\H^2}(\Phi_n,\Xi_n)^2} 
\end{equation}
and observe that by \cite[Lemma 1]{WeinsteinGR} we have
\begin{equation}\label{oanoinininhhj}
\Delta\Lambda(\Phi_n,\Xi_n) \geq -|\tau(\Xi_n)| \geq 
\begin{cases}
-\frac C{r^2} & \text{ on } B_{\delta}\setminus \Gamma \\
0 & \text{ on } (\mathbb{R}^3 \setminus B_{\delta} )\setminus\Gamma
\end{cases}.
\end{equation}
Now define a radial function
\begin{equation}
w(r)=\begin{cases}
C\ln \left(\dfrac{\delta}{r}\right) +C & \text{ on } B_{\delta} \\
\dfrac{C \delta}{r} & \text{ on } \mathbb{R}^3 \setminus B_{\delta} 
\end{cases},
\end{equation}
and note that $w\in C^{1,1}(\R^3\setminus\{0\})$ with
\begin{equation}
\Delta w= \begin{cases}
-\dfrac{C}{r^2} & \text{ on }B_{\delta}\\
0 & \text{ on } \mathbb{R}^{3}\setminus B_{\delta}
\end{cases},
\end{equation}
from which it follows that
\begin{equation}\label{aoihnoinoinqoihnpoiqjhpoijnpioh}
\Delta\left(\Lambda(\Phi_n,\Xi_n)-w\right) \geq 0 \quad \text{ on }\mathbb{R}^3 \setminus \left(\Gamma \cup \partial B_{\delta}\right).
\end{equation}
Although this inequality is satisfied only outside the axis $\Gamma \cup \partial B_{\delta}$, \cite[Lemma~8]{WeiDuke} implies it is satisfied weakly on all of $\R^3$ so that the maximum principle applies. Since $\Phi_n$ and $\Xi_n$ are asymptotic by Lemma \ref{aoijfoainoinhk} and Remark \ref{foanfoinhghhh}, we find that $\Lambda(\Phi_n,\Xi_n)-w\to1$ at infinity. Moreover, $\Lambda(\Phi_n,\Xi_n)-w <0$ on the boundary of a small enough ball centered at the origin. Hence, the maximum principle then yields the $L^\infty$-bound
\begin{equation} \label{eq:pointwise-bound}
d_{\H^2}(\Phi_n,\Psi_n) = d_{\H^2}(\Phi_n,\Xi_n) \leq \sqrt{w(w+2)} = \frac{\sqrt{C\delta(C\delta +2r)}}{r} \quad\text{ on }\R^3\setminus (B_\delta\cup\Gamma),
\end{equation}
where $C$ is independent of $n$ and $\delta$.\smallskip

\textit{Case 2.} Consider the situation in which there is only a single collision location with multiple punctures colliding, and for convenience denote the $z$-coordinates of the $N_1$ punctures $p_i$ colliding at the one location by $z_1<\dots<z_{N_1}$. Without loss of generality it may be assumed that $\sum_{i=1}^{N_1} z_i=0$. Let $z_{0,i}=(z_{i+1}+z_i)/2$, $\eta=1/4 \cdot\min_{1\leq i\leq N_1 -1}(z_{i+1}-z_i)$, and $B_i=B_\eta(p_i)$. Clearly the balls $B_1,\dots,B_{N_1}$ are mutually disjoint. Now set $\vartheta=\frac{\pi}{2(N_1 -1)}$ and for $i=1,\ldots,N_1 -1$ define the sectors
\begin{equation}
S_i=\{(\tilde{r}_i,\tilde{\theta}_i,\phi)\in\mathbb{R}^3 \mid \tfrac{3\pi}{4}-i\vartheta < \tilde{\theta}_i < \tfrac{3\pi}{4}-(i-1)\vartheta\}, 
\end{equation}
where $(\tilde{r}_i,\tilde{\theta}_i,\phi)$ are polar coordinates centered at the point $\tilde{p}_i \in\Gamma$ with $z$-coordinate $z_{0,i}$. The sectors $S_i$ are mutually disjoint, and are also disjoint from the balls $\cup_{i=1}^{N_1} B_i$. We will denote the components of the complement of $\cup_{i=1}^{N_1 -1} S_i$ by $T_1 ,\ldots , T_{N_1}$ in order of increasing $z$. For each $i=1,\ldots, N_1$ construct cut-off functions $\tilde\chi_i\geq 0$ satisfying the following properties: $\tilde{\chi}_i$ is a function of only $\tilde{\theta}_i$ in $S_i$ and is a function of only $\tilde{\theta}_{i-1}$ in $S_{i-1}$, with $\tilde\chi_i=1$ in $T_i$, $\tilde\chi_i+\tilde\chi_{i+1}=1$ in $S_i$, and $\sum_{i=1}^{N_1}\tilde\chi_i=1$ globally. Motivated by the arguments from Case~1, the model map will be taken in the form
\begin{equation}
\Xi = (u,v) = \chi_0\Theta_0 + \chi_1\Theta_1 + \dots + \chi_{N_1}\Theta_{N_1},
\end{equation}
where $\chi_0$ is a nonnegative function of $r$ alone which vanishes in $B_{\delta/2}$ and equals $1$ outside $B_\delta$, and $\chi_i=(1-\chi_0)\tilde\chi_i$ so that $\sum\chi_i =1$. As before $n$ may be taken large enough to guarantee that $\cup_{i=1}^{N_1} B_i$ and the collision location lie within $B_{\delta/4}$, where $r$ respectively $B_{\delta}$ are the Euclidean distance and $\delta$-ball centered at the origin which serves as the center of mass of the punctures. Moreover $\Theta_i$, $i=1,\ldots,N_{1}$ denotes the extreme Kerr harmonic map associated with each colliding puncture, and $\Theta_0$ is the extreme multi-Kerr auxiliary collision configuration map $\Psi$. 

The tension $\tau(\Xi)$ vanishes except possibly in the regions $S_i \cap B_{\delta/2}$ and $A=B_{\delta}\setminus \overline{B}_{\delta/2}$. Note that in each sector $S_i$ only two of the cut-off functions, $\tilde{\chi}_i$ and $\tilde{\chi}_{i+1}$, are ever nonconstant. Therefore, a straightforward generalization of the arguments leading to 
the tension estimate \eqref{eq:C/r^2} in $B_{\delta/2}$ apply here to yield the same bound with $1/r^2$ replaced by $\sum_{i=1}^{N_1 -1}1/\tilde{r}_{i}^2$. Similarly, minor modifications to the methods of Case 1 provide the same estimate with $1/r^2$ in the annulus $A$. Since $\tilde{r}_i$ is comparable with $r$ inside $A$ it follows that
\begin{equation}\label{eq:tension-bound11}
|\tau(\Xi_n)|\leq C\sum_{i=1}^{N_1 -1}\frac{1}{\tilde{r}_i^2} \quad \text{ on }B_{\delta}\setminus\Gamma,
\qquad |\tau(\Xi_n)|=0 \quad \text{ on }(\mathbb{R}^3 \setminus B_{\delta})\setminus\Gamma,
\end{equation}
where $C$ is a constant independent of $n$ and $\delta$.  Define radial functions $w_i\in C^{1,1}(\R^3\setminus\{\tilde{p}_i\})$, $i=1,\ldots N_1 -1$ by
\begin{equation}
w_i(\tilde{r}_i)=\begin{cases}
C\ln \left(\dfrac{2\delta}{\tilde{r}_i}\right) +C & \text{ on } B_{2\delta}(\tilde{p}_i) \\
\dfrac{2C \delta}{\tilde{r}_i} & \text{ on } \mathbb{R}^3 \setminus B_{2\delta}(\tilde{p}_i) 
\end{cases},
\end{equation}
and note that as in \eqref{oanoinininhhj}-\eqref{aoihnoinoinqoihnpoiqjhpoijnpioh} we have
\begin{equation}
\Delta\left(\Lambda(\Phi_n,\Xi_n)-\sum_{i=1}^{N_1 -1}w_i\right) \geq 0 \quad \text{ on }\mathbb{R}^3 \setminus \left(\Gamma \cup_{i=1}^{N_1 -1}\partial B_{2\delta}(\tilde{p}_i)\right).
\end{equation}
The weak maximum principle argument may then be applied as before to obtain
\begin{equation} \label{eq:pointwise-bound1}
d_{\H^2}(\Phi_n,\Psi_n) = d_{\H^2}(\Phi_n,\Xi_n) \leq \sqrt{w (w+2)} \leq\frac{\sqrt{C\delta(C\delta +2r)}}{r} \quad\text{ on }\R^3\setminus (B_{3\delta}\cup\Gamma),
\end{equation}
where $w=\sum_{i=1}^{N_1 -1}w_i$ for some constant $C$ independent of $n$ and $\delta$.\smallskip

\textit{Case 3.} Consider now the situation in which there are simultaneous collisions at $k>1$ separate locations, with punctures $p_i$ having $z$-coordinates $z_1<\dots<z_{N_1}$ colliding at the first location, those labelled by $z_{N_1+1}<\dots<z_{N_2}$ colliding at the second location and so on, up to those labelled by $z_{N_{k-1}+1}<\dots<z_{N_k}$ colliding at the $k$th location. For convenience of notation the indices moving from one collision group to another are consecutive, however there may be noncolliding punctures located between the collision groups. Taking $\eta$ to be $1/4$ of the minimum distance between any two colliding punctures, we can repeat the arguments of Case~2 inside the balls $B_{\delta}(\bar{p}_j)$, $j=1,\ldots,k$ where $\bar{p}_j \in\Gamma$ is the point with $z$-coordinate given by the mean of $z_{N_{j-1}+1},\dots,z_{N_j}$, and $\delta$ is chosen so that the larger balls $B_{3\delta}(\bar{p}_j)$
are mutually disjoint; here $N_0 =0$. When $n$ is sufficiently large, $\eta$ is small enough to guarantee that each $B_{\eta}(p_i)$, for $N_{j-1}+1\leq i\leq N_j$, is contained within the corresponding ball $B_{\delta/4}(\bar{p}_j)$. The resulting model map $\Xi_n$ transitions within each $B_{\delta/2}(\bar{p}_j)$ from one extreme Kerr solution to another across sector regions, and it further transitions across the associated $k$ annuli out to the extreme multi-Kerr auxiliary collision configuration map $\Psi_n$ on the complement of the $\delta$-balls. The methods of Case 2 lead to a tension estimate
\begin{align}\label{eq:tension-boundfajopijnh}
\begin{split}
|\tau(\Xi_n)|&\leq C\!\!\!\sum_{i=N_{j-1}+1}^{N_{j}-1}\frac{1}{\tilde{r}_i^2} \quad \text{ on }B_{\delta}(\bar{p}_j)\setminus\Gamma,\text{ }j=1,\ldots,k,\\
|\tau(\Xi_n)|&=0 \quad \text{ on }\left(\mathbb{R}^3 \setminus \cup_{j=1}^{k} B_{\delta}(\bar{p}_j)\right)\setminus\Gamma,
\end{split}
\end{align}
where the $\tilde{r}_i$ denote Euclidean distance to the sector vertices.
As above, this leads via the maximum principle to the distance bound
\begin{equation} \label{eq:pointwise-bound12}
d_{\H^2}(\Phi_n,\Psi_n) = d_{\H^2}(\Phi_n,\Xi_n) \leq\frac{\sqrt{C\delta(C\delta +2\bar{r})}}{\bar{r}} \quad\text{ on }\R^3\setminus \left(\cup_{j=1}^{k}B_{3\delta}(\bar{p}_j)\cup\Gamma\right),
\end{equation}
for some constant $C$ independent of $n$ and $\delta$ where $\bar{r}=\min\{\bar{r}_1,\ldots,\bar{r}_k\}$ with $\bar{r}_j$ representing the distance to $\bar{p}_j$. For convenience we will denote the region on which \eqref{eq:pointwise-bound12} holds by $D_{\delta}$.

Having achieved the $L^\infty$-bound, we are now able to prove \eqref{2-collision-convergence-Phi} and \eqref{2-collision}. According to \eqref{2-collision-energy-convergence}, the sequence $\Psi_n$ converges on $\mathbb{R}^3 \setminus\Gamma$ to the collision configuration map $\Phi_*$. 
Let $\widetilde{\Phi}_*$ denote the limit of $\Phi_n$, then we claim that $\widetilde{\Phi}_* =\Phi_*$. Since both of these maps are harmonic, this conclusion would follow immediately if they were known to be asymptotic. However, the estimates above do not show this property due to the blow-up at certain points on the axis. Instead, we can use \eqref{eq:pointwise-bound12} directly to show that $\Phi_n\to\Phi_*$. Indeed, if $\delta_2 \gg\delta_1 >0$ then~\eqref{eq:pointwise-bound12} implies that $d_{\H^2}(\Phi_n,\Psi_n)=O(\sqrt{\delta_1})$ on $D_{\delta_2}$. In particular, by keeping $\delta_2>0$ fixed and sending $\delta_1$ to zero as $n\rightarrow\infty$, we find that $\Phi_n$ and $\Psi_n$ must converge to the same limit $\Phi_*$ on $D_{\delta_2}$. Since $\delta_2$ is arbitrary, it follows that $\widetilde{\Phi}_* =\Phi_*$. This establishes \eqref{2-collision-convergence-Phi}.

To complete the proof of the proposition let $\epsilon>0$ be given, and set $\mathcal{B}_{\delta}=\cup_{j=1}^{k} B_{\delta}(p_j^*)$ where $p_j^*$ denote the $k$ distinct collision points having $z$-coordinates $z_j^*$. Since $\Phi_*$ has finite renormalized energy, $\delta$ may be chosen small enough to guarantee that $\mathcal{E}_{\mathcal{B}_{\delta}}(\Phi_*)<\epsilon/2$.
Furthermore, since $\widetilde{\Phi}_* =\Phi_*$ 
we have that 
$\mathcal E_{\R^3\setminus\mathcal{B}_{\delta}}(\Phi_n) \to \mathcal E_{\R^3\setminus \mathcal{B}_{\delta}}(\Phi_*)$ as $n\rightarrow\infty$. To see this in more detail, note that since the difference of energy densities converges to zero in $L^1$ on any fixed compact subset of $\R^3\setminus\mathcal{B}_{\delta}$ with the aid of \cite[Theorem 2.1]{HKWX}, the same holds if the fixed set is replaced by a slow sequence of exhausting domains. This fact, combined with estimates for the difference of energy densities in the complement of the exhausting domains obtained from the expansions \cite[Theorem 2.3]{HKWX}, produces the desired conclusion. Thus, there exists a sufficiently large $n_{\epsilon}$ such that
\begin{equation}
\mathcal E(\Phi_*) =
\mathcal E_{\mathcal{B}_{\delta}}(\Phi_*) + \mathcal E_{\R^3\setminus \mathcal{B}_{\delta}}(\Phi_*)
<\epsilon + \mathcal E_{\R^3\setminus \mathcal{B}_{\delta}}(\Phi_{n_{\epsilon}})
\leq \epsilon + \mathcal E(\Phi_{n_{\epsilon}}),
\end{equation}
which yields \eqref{2-collision}.
\end{proof}

The proof of Proposition \ref{prop:collision} relies on two lemmas which we now establish. In what follows the notation will be consistent with that utilized in the proof of this proposition.

\begin{lem} \label{lemma:v}
Let $\eta,\delta>0$ be small parameters with $\delta> 12\eta$. Consider two extreme Kerr harmonic maps $(u_i,v_i)$, $i=1,2$ with punctures located at $p_1=(0,0,-2\eta)$ and $p_2=(0,0,2\eta)$ in Cartesian coordinates, and set $(u,v)=\sum_{i=1}^2\chi_i (u_i,v_i)$ where $\chi_i$ are cut-off functions from Case 1 in the proof of Proposition~\ref{prop:collision}. Then there exists a constant $C$ independent of $\eta$ and $\delta$ such that
\begin{equation}\label{aohonoinhjj}
e^{2u}|v_1-v_2| \leq C, \quad\quad
r e^{2u}|\nabla v_i| \leq C, \quad\quad
r^2 e^{2u} |\Delta v_i| \leq C,
\end{equation}
in sector $S \cap B_{\delta}$ for $i=1,2$.
\end{lem}

\begin{proof}
Let $(r_i,\theta_i,\phi)$ be polar coordinates in $\mathbb{R}^3$ centered at $p_i$, $i=1,2$ and let $a_i>0$ be parameters, then recall that the extreme Kerr harmonic map \cite[Appendix B]{KhuriWeinstein} with puncture $p_i$ and angular momentum $a_i^2$ is given by
\begin{equation}
u_i=-\ln\sin\theta_i -\frac{1}{2}\ln\left((r_i +a_i)^2 +a_i^2 +\frac{2a_i^2 (r_i +a_i)\sin^2 \theta_i}{(r_i +a_i)^2 +a_i^2 \cos^2 \theta_i}\right),
\end{equation}
\begin{equation}
v_i=a_i^2 \cos\theta_i (3-\cos^2 \theta_i)+\frac{a_i^4 \cos\theta_i \sin^4 \theta_i}{(r_i +a_i)^2 +a_i^2 \cos^2 \theta_i}.
\end{equation}
The extreme Kerr map with angular momentum $-a_i^2$ may be obtained by replacing $v_i$ with $-v_i$. Furthermore, for the purposes of this lemma we may assume without loss of generality that the angular momentum at $p_1$ is $a/2$ and that the angular momentum at $p_2$ is $1/2$ for some $a\neq 0$, so that
\begin{equation}\label{foanofinoiqh}
v_1 = \frac a2\cos\theta_1 (3-\cos^2\theta_1) - a -1 + O(\sin^4\theta_1),\quad
v_2 = \frac12 \cos\theta_2 (3-\cos^2\theta_2) + O(\sin^4\theta_2).
\end{equation}
Generality is achieved by noting that for any constant $c$ the map $(u,v)\rightarrow (u+c , e^{-2c}v)$ is an isometry in the target hyperbolic space. Even though such `hyperbolic translations' do not preserve asymptotic flatness of the associated spacetime, the estimates in \eqref{aohonoinhjj} will be preserved. Note that the potential constants for $v_1$ and $v_2$ agree (and equal $-1$) on the axis rod between the two punctures.

Using $\rho=r\sin\theta=r_i\sin\theta_i$ and $z=r\cos\theta=r_1 \cos\theta_1 -2\eta=r_2\cos\theta_2 +2\eta$, we find 
in the region $S\cap B_{\delta/2}$ (where $\theta_1 <\pi/2$ and $\theta_2 >\pi/2$) that
\begin{equation}
\cos\theta_{1,2}= \frac{\pm 1}{\sqrt{1+\tan^2\theta_{1,2}}} = \pm \left(1+\left(\frac{r\sin\theta}{r\cos\theta\pm 2\eta}\right)^2 \right)^{-1/2},
\end{equation}
and consequently if $x=2\eta/r$ then
\begin{equation}\label{aboifnoianoginoqaiahg}
\sin^2 \theta_{1,2} =1-\cos^2 \theta_{1,2} =\frac{\sin^2 \theta}{1\pm 2x\cos\theta +x^2}.
\end{equation}
It follows that
\begin{equation}\label{aonfoinaoinhqah}
v_1 -v_2=W(x,\theta)+O\left(\frac{\sin^4 \theta}{(1+2x\cos\theta +x^2)^2}+\frac{\sin^4 \theta}{(1-2x\cos\theta +x^2)^2}\right),
\end{equation}
where
\begin{align}
\begin{split}
W(x,\theta)&=\frac{a|\cos\theta+x|}{2\sqrt{1+2x\cos\theta+x^2}}\left(3-\frac{(\cos\theta+x)^2}{1+2x\cos\theta+x^2}\right) - a - 1\\
& \quad+\frac{|\cos\theta-x|}{2\sqrt{1-2x \cos\theta+ x^2}}\left(3-\frac{(\cos\theta-x)^2}{1-2 x \cos\theta+x^2}\right).
\end{split}
\end{align}
Moreover
\begin{equation}\label{fonOINIOQHH}
u = \chi_1 u_1 + \chi_2 u_2 = -\ln\sin\theta + \frac{\chi_1}{2}\ln(1+2x \cos\theta+ x^2) + \frac{\chi_2}{2}\ln(1-2x \cos\theta+ x^2) + O(1),
\end{equation}
and hence
\begin{equation}\label{afonoih}
e^{2u}|v_1-v_2| \leq C \frac{(1-2x \cos\theta+ x^2)^{\chi_1}(1+2x \cos\theta+ x^2)^{\chi_2}}{\sin^2\theta} W(x,\theta)+O(1)
\end{equation}
in $S \cap B_{\delta}$, for some constant $C$ independent of $\eta$ and $\delta$. By expanding in powers of $x$ and noting several cancellations, a calculation shows that 
\begin{equation}
|W(x,\theta)|\leq \frac{C}{x^2} \quad\text{ for all }x\geq 1 \text{ and }\pi/4 < \theta< 3\pi/4.
\end{equation}
We also observe that within the region given by $x\leq 1$ and $\pi/4 < \theta< 3\pi/4$, the function $W(x,\theta)$ remains uniformly bounded. These facts, together with $\chi_1 +\chi_2 =1-\chi_0 \leq 1$, imply that the first inequality of \eqref{aohonoinhjj} is valid.

Consider now the second inequality of \eqref{aohonoinhjj} and notice that for $i=1,2$ a direct computation yields
\begin{equation}\label{goinaqoinoiqh}
|\nabla v_i| \leq C\frac{ \sin^3 \theta_i}{r_i}.
\end{equation}
Next use $u_i=-\ln\sin\theta_i +O(1)$ and $\chi_1+\chi_2=1-\chi_0$ to find
\begin{equation}
e^{2u}|\nabla v_i|\leq e^{2|1-\chi_0-\chi_i||u_1 -u_2|}e^{2(1-\chi_0)u_i}|\nabla v_i|\leq C\frac{e^{2|u_1 -u_2|}\sin \theta_i}{r_i}.
\end{equation}
Moreover, for all $x>0$ and $\pi/4 < \theta< 3\pi/4$ we obtain  
\begin{equation}\label{foanoianh}
e^{2(u_1 -u_2)} \leq C\frac{\sin^2 \theta_2}{\sin^2 \theta_1}\leq C\frac{1 +2x\cos\theta+x^2}{1 -2x\cos\theta +x^2}\leq 10C
\end{equation}
with the help of \eqref{aboifnoianoginoqaiahg}, and a similar inequality holds for $e^{2(u_2 -u_1)}$. Since 
$1/r_i \leq C/r$ for some uniform constant $C$ at all points of $S$ by Corollary \ref{cor:r1-r2}, we conclude that the second inequality of \eqref{aohonoinhjj} is valid.

Consider now the third inequality of \eqref{aohonoinhjj} and note that for $i=1,2$ we have
\begin{equation}\label{gonqaoingoih}
|\Delta v_i| \leq C\frac{\sin^2 \theta_i}{r^2_i}.
\end{equation}
As above this gives rise to
\begin{equation}
e^{2u}|\Delta v_i|\leq e^{2|1-\chi_0-\chi_i||u_1 -u_2|}e^{2(1-\chi_0)u_i}|\Delta v_i|\leq C\frac{e^{2|u_1 -u_2|}}{r^2_i},
\end{equation}
and the desired estimate follows from \eqref{foanoianh}.
\end{proof}

\begin{lem} \label{lemma:v1}
Let $\eta,\delta>0$ be small parameters with $\delta> 12\eta$. Using the notation from Case 1 in the proof of Proposition~\ref{prop:collision}, consider a extreme multi-Kerr harmonic map $(u_0,v_0)$ having a puncture at the collision location in $B_{\delta/4}$ with all other punctures occurring outside a fixed radius, and set $(\tilde{u},\tilde{v})=\sum_{i=1}^2\tilde{\chi}_i (u_i,v_i)$ using two extreme Kerr harmonic maps 
$(u_i,v_i)$, $i=1,2$ having punctures located at $p_1=(0,0,-2\eta)$ and $p_2=(0,0,2\eta)$ in Cartesian coordinates. Then there exists a constant $C$ independent of $\eta$ and $\delta$ such that
\begin{equation}\label{eq:v1}
(\sin\theta)^{-1-\varsigma}e^{2u}|\tilde{v}-v_0| \leq C, \qquad \quad
r e^{2u}\left((\sin\theta)^{-1}|\nabla \tilde{v}|+(\sin\theta)^{-\varsigma}|\nabla v_0|\right)  \leq C, 
\end{equation}
\begin{equation}\label{eq:v2}
r^2 e^{2u} \left(|\Delta \tilde{v}|+ (\sin\theta)^{1-\varsigma} |\Delta v_0|\right) \leq C,
\end{equation}
in annulus $A=B_{\delta}\setminus \overline{B}_{\delta/2}$ where $u=\chi_0 u_0 +(1-\chi_0)\tilde{u}$, and $\varsigma\in(0,1)$ is arbitrary but fixed.
\end{lem}

\begin{proof}
Consider the first expression of \eqref{eq:v1} in $A\cap S$, and observe that 
\begin{equation}
e^{2u}|\tilde{v}-v_0|\leq e^{2u}|\tilde{\chi}_2 (v_1 -v_2)|+e^{2u}|v_1 -v_0|.
\end{equation}
Applying Lemma \ref{lemma:v} in this region yields
\begin{equation}
e^{2u}|\tilde{\chi}_2 (v_1 -v_2)|\leq e^{2\chi_0 u_0}\left(e^{2(\chi_1 u_1 +\chi_2 u_2)}|v_1-v_2|\right)\leq C,
\end{equation}
where we have also used that $e^{2\chi_0 u_0}$ remains uniformly bounded due to \eqref{3.1}-\eqref{3.3} combined with the restriction of angles inherent within $A\cap S$. It should be noted that even though the collision location which serves as the puncture for the map $(u_0,v_0)$ does not necessarily coincide with the origin, its distance to the origin tends to zero as $\eta\rightarrow 0$, so that polar coordinate angles measured from the collision location are comparable to those measured from the origin. Moreover, $v_1 -v_0$ may be computed analogously to \eqref{aonfoinaoinhqah}, and arguments as in Lemma \ref{lemma:v} may be utilized to estimate $e^{2(\chi_0 u_0 +\chi_1 u_1)}|v_1 -v_0|$ inside $A\cap S$. Again since $e^{2\chi_2 u_2}$ is controlled here we obtain a bound for $e^{2u}|v_1 -v_0|$, and hence the desired inequality for the first inequality of \eqref{eq:v1} follows in $A\cap S$. In this process, the expansion \eqref{3.1}-\eqref{3.3} allows us to effectively replace $(u_0,v_0)$ with a single extreme Kerr. Similar considerations show that the proof of Lemma \ref{lemma:v} also yields the desired estimates for the remaining quantities of \eqref{eq:v1} and \eqref{eq:v2}, in $A\cap S$. In fact, $\varsigma$ plays no role here and may be taken to be zero.


Consider now the region $A\setminus S=A_+ \cup A_-$, where $A_{\pm}$ are the portions of the annulus with $0\leq \theta\leq \pi/4$ and $3\pi/4\leq\theta\leq \pi$, respectively. We will treat the $A_+$ region, and note that similar arguments hold for $A_-$. First observe that in this region $\tilde{\chi}_1 =0$, and therefore $\tilde{v}=v_2$. We may assume without loss of generality that the potential constants for $v_1$ and $v_2$ are as in the proof of  Lemma \ref{lemma:v}. Then \eqref{3.3} and \eqref{foanofinoiqh} produce
\begin{equation}
v_0 =\frac{a+1}{2}\cos\theta_0(3-\cos^2 \theta_0)-a +O(\sin^{3+\varsigma}\theta_0),\quad 
v_2 = \frac12 \cos\theta_2 (3-\cos^2\theta_2) + O(\sin^4\theta_2),
\end{equation}
where $(r_0,\theta_0,\phi)$ are polar coordinates centered at the collision location and $\varsigma\in(0,1)$. Using the identity
\begin{equation}
\frac{1}{2}\cos\theta(3-\cos^2 \theta)=1-6 \sin^4(\theta/2)+4\sin^6(\theta/2),
\end{equation}
it follows that
\begin{equation}
\tilde{v}-v_0 =O(\sin^{3+\varsigma}\theta_0 +\sin^4 \theta_2).
\end{equation}
Next utilize \eqref{3.1} and similar computations as in \eqref{fonOINIOQHH} to find
\begin{equation}
u=\chi_0 u_0 +(1-\chi_0)u_2=-\ln\sin\theta +O(1),
\end{equation}
and hence within $A_+$ we have
\begin{equation}
e^{2u}|\tilde{v}-v_0|\leq C\left(\frac{\sin^{3+\varsigma}\theta_0}{\sin^2 \theta} +\frac{\sin^4 \theta_2}{\sin^2 \theta} \right)
\end{equation}
for some uniform constant $C$. Furthermore, \eqref{aboifnoianoginoqaiahg} implies
\begin{equation}
\frac{\sin^2 \theta_2}{\sin^2 \theta} =(1-2x\cos\theta +x^2)^{-1}=\left((x-\cos\theta)^2 +\sin^2 \theta\right)^{-1}
\leq (1/\sqrt{2} -1/3)^{-2}
\end{equation}
since $x=2\eta/r <4\eta/\delta<1/3$ in $A_+$, and a similar estimate holds if $\theta_2$ is replaced by $\theta_0$. Hence, the first inequality of \eqref{eq:v1} is satisfied in $A_+$. Moreover, these observations also show that the derivative inequalities of \eqref{eq:v1} and \eqref{eq:v2} involving $\tilde{v}$ follow from \eqref{goinaqoinoiqh} and \eqref{gonqaoingoih}, together with the fact that $r_i$, $i=1,2$ are comparable with $r$ in $A$ by Lemma \ref{lem:r0-ri}. Lastly, the derivative inequalities involving $v_0$ may be obtained by replacing \eqref{goinaqoinoiqh} and \eqref{gonqaoingoih} with
\begin{equation}
|\nabla v_0|\leq C\frac{\sin^{2+\varsigma}\theta_0}{r_{0}},\quad\quad\quad |\Delta v_0|\leq C\frac{\sin^{1+\varsigma}\theta_{0}}{r_0^2},
\end{equation}
which is due to \eqref{3.1}-\eqref{3.3}. We conclude that the desired estimates are valid on all of $A$.
\end{proof}

\section{Proof of the Main Theorem}
\label{sec6} \setcounter{equation}{0}
\setcounter{section}{6}

In this section we will complete the proof of Theorem~\ref{maintheorem}. Let $(M,g,k)$ be an initial data set satisfying the requisite hypotheses of the theorem. According to \eqref{masslowerb2} we have the following relation between the ADM mass $m$ and the renormalized energy of a harmonic map $\Phi=(u,v):\mathbb{R}^3 \setminus \Gamma\rightarrow\mathbb{H}^2$ which encodes the angular momenta of the data, namely
\begin{equation}
m\geq \frac{1}{8\pi}\mathcal{E}(\Phi).
\end{equation}
For convenience in what follows, we will write $\bar{\mathcal{E}}(\Phi)=\tfrac{1}{8\pi}\mathcal{E}(\Phi)$.
Suppose there are $N$ punctures associated with the map $\Phi$, and that it has total angular momentum $\mathcal{J}$.

\subsection{Inequality}\label{parti}
It will be proven by induction on $N$ that 
\begin{equation}\label{epsiloninequality}
\bar{\mathcal{E}}(\Phi)\geq \sqrt{|\mathcal{J}|}.
\end{equation}
First note that the case $N=1$ is true, since $\Phi$ must correspond to an extreme Kerr harmonic map with angular momentum $\mathcal{J}$, and therefore $1/8\pi$ times the renormalized energy of the map agrees with $\sqrt{|\mathcal{J}|}$. This latter statement follows from the arguments leading to \eqref{masslowerb2} when applied to extreme Kerr initial data. More precisely, in this situation the ADM mass agrees with the appropriate multiple of renormalized energy, and it is known by direct calculation that the mass is given by $\sqrt{|\mathcal{J}|}$ for extreme Kerr.

Suppose now that \eqref{epsiloninequality} holds for any extreme multi-Kerr harmonic map with $n$ punctures and total angular momentum $\mathcal{J}$, where $1\leq n\leq N-1$, and let $\Phi_0$ be such a harmonic map with $N$ punctures. We start the flow defined in Section~\ref{secflow} with initial condition $\Phi_0$. By Proposition \ref{blipschitz} the flow $\Phi_t$ exists and is regular up to a maximal time $T$ (possibly infinite), and according to Theorem~\ref{decreaselaw} the renormalized energy $\mathcal{E}(\Phi_t)$ is nonincreasing along this flow. We find that there are two possibilities to contend with at time $T$: either a collision occurs, or no collision occurs.
Note that in the first case $T$ can be either finite or infinite, while in the second case $T$ will be infinite. Consider the first of these cases in which a collision occurs at time $T$. By Proposition~\ref{prop:collision} we have that 
\begin{equation}
\lim_{t\rightarrow T}\bar{\mathcal{E}}(\Phi_{t}) \geq \bar{\mathcal{E}}(\Phi_{*}),
\end{equation}
for some collision configuration map $\Phi_{*}$ which has at most $N-1$ punctures and total angular momentum $\mathcal{J}$. By the induction hypothesis we then obtain the desired inequality
\begin{equation}
\bar{\mathcal{E}}(\Phi_0) \geq \lim_{t\rightarrow T}\bar{\mathcal{E}}(\Phi_{t}) \geq \bar{\mathcal{E}}(\Phi_{*})\geq \sqrt{|\mathcal J|}.
\end{equation}

Now consider the second case, in which no collisions occur at time $T=\infty$. Since the function $t\mapsto \mathcal{E}(\Phi_t)$ is continuously differentiable and nonincreasing on $[0,\infty)$ we find that
\begin{equation}
\int_{0}^{\infty}\Big|\frac{d}{dt}\mathcal{E}(\Phi_t)\Big| dt = \mathcal{E}(\Phi_0)-\lim_{t\rightarrow \infty}\mathcal{E}(\Phi_t)\leq \mathcal{E}(\Phi_0).
\end{equation}
It follows that there exists a sequence of times $t_j \rightarrow\infty$ such that $\tfrac{d}{dt}\mathcal{E}(\Phi_{t_j})\rightarrow 0$. Next observe that the $N$ punctures along these times converge to some $N$ distinct fixed limits $p^*_i$, $i=1,\dots, N$. 
The limiting harmonic map $\Phi^*$ must be stationary with respect to the flow, since the associated tangent map parameters satisfy $b_{i}^*=0$ for $i=1,\dots,N$ by virtue of their continuity (Proposition \ref{blipschitz}) and their relation to the first variation of the renormalized energy \eqref{gonaoinoinihjjj}. 
This contradicts Theorem \ref{bhu}, showing that the second case does not occur.

This completes the induction argument, and therefore also the inequality portion of Theorem~\ref{maintheorem}.

\subsection{Rigidity}
It remains to establish the rigidity statement for the mass-angular momentum inequality.
Consider the case of equality when $m=\sqrt{\mathcal{|J|}}$. This implies equality of renormalized energies $\mathcal{E}(\bar{\Phi})=\mathcal{E}(\Phi)$, for the map $\bar{\Phi}$ obtained from the initial data $(M,g,k)$ and the associated extreme multi-Kerr harmonic map $\Phi$. Due to the gap bound \cite[Theorem 4.1]{KhuriWeinstein}, it follows that $\bar{\Phi}=\Phi$. Moreover, several other restrictions on $g$ and $k$ are obtained from equality in \eqref{masslowerbound1}, as detailed at the end of Section 2 in \cite{KhuriWeinstein}. We may then construct from $\Phi$ an axisymmetric stationary vacuum spacetime with an asymptotically flat end and $N$ degenerate horizons, from which $(M,g,k)$ arises as a constant time slice. Since the given initial data set is void of conical singularities, this spacetime is regular. If $N>1$, a contradiction is obtained from Theorem \ref{bhu}. We conclude that $N=1$, and \cite[Theorem 1.1]{KhuriWeinstein} yields the desired result that $(M,g,k)$ must arise from a constant time slice of an extreme Kerr black hole.

\section{Linearized Harmonic Maps Near Punctures}
\label{sec:linear-hm}
\setcounter{equation}{0}
\setcounter{section}{7}

This section represents an extension of the study of isolated singularities that was presented in \cite[Theorem 2.1]{HKWX}, to higher order expansions. The results of this and the next section are needed to prove Theorem \ref{thm:smoothness} and Corollary \ref{cor:linest}. Let $\Phi=(u, v) $ be an $H_{\mathrm{loc}}^{1}$-harmonic map from $B_1\setminus \Gamma$, the Euclidean unit ball minus the $z$-axis, to $\mathbb{H}^2$. It will not be assumed here that this map is axisymmetric.\footnote{Note that $\Phi$ in the current setting has a different meaning from that in \cite{HKWX}.} 
Suppose that
\begin{equation}
\sup_{B_1\setminus \Gamma }|u+\ln \sin\theta |<\infty ,\quad\quad
|\nabla (u+\ln \rho)|^2+e^{4u} |\nabla v|^2 \in L^1_{\mathrm{loc}} (B_1\setminus \{0\}),
\end{equation}
with
\begin{equation}\label{oinhiqoinh}
v=a\text{ on }B_1\cap \Gamma_+ \quad\text{ and }\quad v =-a\text{ on }B_1\cap\Gamma_- \quad\text{ in the trace sense},
\end{equation}
for some positive constant $a$. We will use the notation $x=(x_1,x_2,x_3)$, $r=|x|$ and $\sin \theta= \frac{\rho}{r}$.  By Theorem 2.1 of \cite{HKWX},
$( u+\ln \sin\theta, v ) \in C^{3,\alpha}(B_1\setminus \{0\})$ for any $\alpha\in(0,1)$,
and there exists a harmonic map $\bar{\Phi}=(\bar{u}, \bar{v})$: $\mathbb{S}^2 \setminus \{N,S\} \to \mathbb{H}^2 $ satisfying
\begin{equation}
(\bar{u}+\ln \sin\theta, \bar{v}) \in C^{3,\alpha}(\mathbb{S}^2) , \quad \bar{v}(N)=a, \quad \bar{v}(S)= -a,
\end{equation}
such that
\begin{equation}\label{qaohoihqhq}
\sup_{\mathbb{S}^2 \setminus\{N,S\}}\left(|(r\partial_r)^l \nabla_{\mathbb{S}^2}^k \left(u(r)-\bar{u}\right)|
+e^{(3+\alpha-k)\bar{u}}|(r\partial_r)^l \nabla_{\mathbb{S}^2}^k \left(v(r)-\bar{v}\right)|\right)
\le \bar{C} r^{\beta}
\end{equation}
 for $l+k\leq 3$, where $\bar C$ and $\beta$ are positive constants and $\Phi(r)=\Phi(r,\cdot)$.
The map $\bar{\Phi}$ is the tangent map of $\Phi$ at the origin, and
the constant $\beta$ depends only on $\bar{\Phi}$.

In the following, we fix a harmonic map $\Phi=(u,v)$ and its tangent map $\bar\Phi=(\bar u,\bar v)$ as above. Let $L=(L_1, L_2)$ be the linearized harmonic map operator at $\Phi$, which is given by
\begin{align}
\begin{split}
L_{1}\varphi &= \Delta \varphi_1 - 8 e^{4u} |\nabla  v|^2 \varphi_1 -4 e^{4u} \nabla v \cdot \nabla \varphi_2,\\
L_{2}\varphi&= \Delta \varphi_2+4    \nabla u \cdot \nabla \varphi_2+4    \nabla v \cdot \nabla \varphi_1,
\end{split}
\end{align}
for $\varphi=(\varphi_1,\varphi_2)$ with $\varphi_1\in C^2(B_1)$ and $\varphi_2\in  C^2_c(B_1\setminus \Gamma)$.
Here, $\Delta$ and $\nabla$ are with respect to the flat metric in $\mathbb R^3$.
The spherical part $\mathcal{T}=(\mathcal{T}_{1},\mathcal{T}_{2})$ of $L$ at $\bar{\Phi}$ is given by
\begin{align}
\begin{split}
 \mathcal{ T}_{1}  \varphi&= \Delta_{g_0} \varphi_1
 - 8 e^{4\overline{u}} |\nabla_{g_0}  \overline{v}|^2 \varphi_1
 -4 e^{4\overline{u}} \nabla_{g_0} \overline{v} \cdot \nabla_{g_0} \varphi_2,\\
 \mathcal{T}_{2}  \varphi&=  \Delta_{g_0} \varphi_2
 +4 \nabla_{g_0} \overline{u}\cdot \nabla_{g_0} \varphi_2 + 4    \nabla_{g_0} \overline{v} \cdot \nabla_{g_0} \varphi_1,
\end{split}
\end{align}
where $g_0 =d\theta^2 +\sin^2 \theta d\phi^2$ is the round metric on $\mathbb{S}^2$.

Denote by $L^2(\mathbb{S}^2, e^{4\bar{u}})$ the subspace of $L^2(\mathbb{S}^2)$
consisting of all functions $f$ with the bounded norm
\begin{equation}
\|f\|_{L^{2}(\mathbb{S}^2, e^{4\bar{u}} )}=\Big(\int_{\mathbb{S}^2} e^{4\bar{u}}  f^2\,  d  vol_{g_0} \Big)^{1/2},
\end{equation}
and by $ H_0^{1}(\mathbb{S}^2, e^{4\bar{u}} )$ the closure of
$C_c^\infty(\mathbb{S}^2\setminus \{N,S\})$ under the norm
\begin{equation}
\|f\|_{H^{1}(\mathbb{S}^2, e^{4\bar{u}} )}
=\Big(\int_{\mathbb{S}^2}e^{4\bar{u}}  (|\nabla_{g_0} f|^2 +f^2)\,  d  vol_{g_0} \Big)^{1/2}.
\end{equation}
The inner products associated with these norms will be denoted by the braces $\langle\cdot, \cdot\rangle$. Introduce the bilinear form
\begin{align}\label{eq:bilinear-form}\begin{split}
\overline{\mathcal{B}}[\varphi, \psi]&= \int_{\mathbb{S}^2} \big(\nabla_{g_0} \varphi_1\cdot   \nabla_{g_0} \psi_1
+e^{4\bar{u}} \nabla_{g_0} \varphi_2 \cdot  \nabla_{g_0} \psi_2
+8 e^{4 \bar{u}} |\nabla_{g_0}    \bar{v}|^2 \varphi_1 \psi_1\\
&\qquad\qquad  +4 e^{4\bar{u}}\psi_1 \nabla_{g_0}  \bar{v}\cdot \nabla_{g_0} \varphi_2
-4 e^{4\bar{u}} \psi_2\nabla_{g_0}  \bar{v}\cdot\nabla_{g_0}  \varphi_1 \big)\, d vol_{g_0},
\end{split}\end{align}
for any $\varphi=(\varphi_1, \varphi_2), \psi=(\psi_1, \psi_2)\in H^1(\mathbb{S}^2)\times H_0^{1}(\mathbb{S}^2, e^{4\bar{u}})$.
If $\varphi_1\in C^2(\mathbb{S}^2)$ and $\varphi_2\in C^2_c(\mathbb{S}^2\setminus\{N,S\})$, then
\begin{equation}
\overline{\mathcal{B}}[\varphi, \psi]= -\langle\mathcal{T} \varphi, \psi\rangle_{L^2(\mathbb{S}^2, 1\times e^{4\bar{u}})}
=-\langle\mathcal{T} _1 \varphi, \psi_1\rangle_{L^2(\mathbb{S}^2)}
- \langle \mathcal{T}_2 \varphi, \psi_2\rangle_{L^2(\mathbb{S}^2, e^{4\bar{u}} )}.
\end{equation}
It was proved in \cite{HKWX} that $\overline{\mathcal{B}}$  is symmetric and nonnegative
on $H^1(\mathbb{S}^2)\times H_0^{1}(\mathbb{S}^2, e^{4 \bar{u}})$, and zero is the least and simple eigenvalue of the problem
\begin{equation}
-\mathcal{T} \varphi= \mu \varphi  \quad\text{in }\mathbb S^2\setminus\{N,S\},\quad\quad
\varphi_2(N)=\varphi_2(S)=0.
\end{equation}
Let
\begin{equation}
0=\mu_1<\mu_2\le \dots\le \mu_i\le \dots \to \infty
\end{equation}
be the collection of eigenvalues, and $\varphi^{(i)}$ be the normalized eigenfunction corresponding to $\mu_i$.

We present a result on the asymptotic behavior of solutions to the linearized equation
\begin{equation} \label{eq:LL-0}
Lw =f \quad \mbox{in }B_1 \setminus\Gamma, \quad w_2=0 \quad \mbox{on }\Gamma \cap B_1,
\end{equation}
as $x\to 0$. The weak solutions will be understood in the integration by parts sense.
Their regularity and local derivative estimates can be deduced from \cite[Proposition 3.9]{HKWX}.

\begin{prop} \label{prop:linear-asymptotic}
Assume that $f=(f_1, f_2)\in C^{1}(B_1 \setminus\Gamma)$ satisfies 
\begin{equation}
|\nabla^k( r^2f_1- \bar{f}_1) |+e^{2u} |\nabla^k ( r^2 f_2- \bar{f}_2)|\le K r^{\gamma-k}\quad\text{in }B_1 \setminus\Gamma,\quad\quad k=0,1,
\end{equation}
for some positive constants $K$, $\gamma$, and a $C^1$-function $\bar{f}=(\bar{f}_1,\bar{f}_2)$ on $\mathbb{S}^2$. Let $w=(w_1,w_2)$ be a weak solution of \eqref{eq:LL-0} satisfying $(w_1,e^{2u}w_2)\in L^\infty (B_1)$.
Then for any $\varphi\in H^1(\mathbb{S}^2)\times H_0^{1}(\mathbb{S}^2, e^{4 \bar{u}})$ with $\mathcal T\varphi=0$ we have
\begin{equation}
\langle\bar{f}, \varphi\rangle_{L^2(\mathbb{S}^2, 1\times e^{4\bar{u}})}=0,
\end{equation}
and there exists a function $\bar{w}=(\bar{w}_1, \bar{w}_2)\in H^1(\mathbb{S}^2)\times H_0^{1}(\mathbb{S}^2, e^{4 \bar{u}})$ with
\begin{equation}
\mathcal{T}\bar{w}=\bar{f}\quad\text{in }\mathbb S^2\setminus\{N, S\}.
\end{equation}
Moreover, there is a universal constant
$0<\bar \beta\le \min\{\beta, \gamma\}$ such that for $k=0,1$ and any $\alpha\in (0,1)$ it holds that
\begin{align}\label{eq-linear-limit-decay}
|\nabla ^k(  w_1-\bar{w}_1)|(x) + (\sin \theta)^{k-3-\alpha } |\nabla ^k (  w_2-\bar{w}_2) |(x) \le C |x| ^{\bar \beta-k}\quad\text{in } B_1 \setminus \Gamma,
\end{align}
where $C$ is a constant depending only on $K$, $\alpha$, $\gamma$, and the $L^\infty$-norms of $w_1$, $e^{2u}w_2$, $r^2 f_1$, and $r^2 e^{2u}f_2$ in $B_{1}$.
\end{prop}

Although $w$ in Proposition \ref{prop:linear-asymptotic} is assumed to be bounded in $B_1$,
in some applications only its bound in $B_1\setminus B_\delta$ is known, for some $\delta>0$. In order to apply
Proposition \ref{prop:linear-asymptotic}, we need to estimate the bound of $w$ in $B_1$ in terms of that in $B_1\setminus B_\delta$.

\begin{prop} \label{prop:linear-boundness}
Let $f$ and $w$ be as in Proposition \ref{prop:linear-asymptotic}.
Then there exists $0<\delta<1/2$
depending only on $\gamma$, $a$, as well as $\bar{C}$ and $\beta$ in \eqref{qaohoihqhq}, such that 
\begin{align}
\begin{split}
&\||w_1|+e^{2u}|w_2|\|_{L^\infty(B_{1/2})} \\
&\qquad\le C\big(\||w_1|+e^{2u}|w_2|\|_{L^\infty(B_{1/2}\setminus B_{\delta})} +\| |r^2f_1|+e^{2u} |r^2 f_2| \|_{L^\infty(B_1)} +K\big).
\end{split}
\end{align}
Moreover, \eqref{eq-linear-limit-decay} holds for
$C$ depending only on $K$, $\alpha$, $\gamma$ and the $L^\infty$-norms of $w_1$ and $e^{2u}w_2$ in $B_{1}\setminus B_\delta$.
\end{prop}

To prove Propositions \ref{prop:linear-asymptotic} and \ref{prop:linear-boundness}, it is convenient to adopt the cylindrical coordinate $t=-\ln r$. Then the Euclidean metric becomes
\begin{equation}
(\ud x )^2= \ud r^2+r^2 g_0=  e^{-2t} (\ud t^2+g_0). 
\end{equation}
In what follows we will set
\begin{equation}\label{eq:product-metric}
\hat g=\ud t^2+g_0,
\end{equation}
which is the standard product metric on $\mathbb R_{+}\times \mathbb S^2$.
Denote by $\nabla_{\hat g}$ and $\Delta_{\hat g}$ the gradient operator and the Laplace-Beltrami operator with respect to the metric $\hat g$ on $\R_{+}\times \mathbb S^2$, in particular
\begin{equation}
\nabla_{\hat g}=(\pa_t, \nabla_{g_0}), \quad\quad \Delta_{\hat g}=\pa_t^2+\Delta_{g_0}.
\end{equation}

Recall that $(\bar{u}, \bar{v})$ is the tangent map of $\Phi$ and $L$ is the linearized harmonic map operator at $\Phi$. In the new coordinates this takes the form
\begin{align}
\begin{split}
e^{-2t}L_{1} \varphi &= (\pa_t^2-\pa_t +\Delta_{g_0}) \varphi_1
- 8 e^{4u } |\nabla_{\hat g}  v|^2 \varphi_1 -4 e^{4u } \nabla_{\hat g} v  \cdot \nabla_{\hat g} \varphi_2,\\
e^{-2t}L_{2}\varphi& = (\pa_t^2-\pa_t +\Delta_{g_0}) \varphi_2
+ 4 \nabla_{\hat g} u \cdot \nabla_{\hat g} \varphi_2 + 4    \nabla_{\hat g} v  \cdot \nabla_{\hat g} \varphi_1.
\end{split}
\end{align}
Write
\begin{equation}
e^{-2t} L\varphi= \mathcal{ L}  \varphi +\mathcal{Q} \varphi,
\end{equation}
where $\mathcal{ L}=(\mathcal{ L}_1, \mathcal{ L}_2)$ and $\mathcal{Q}=(\mathcal{Q}_1, \mathcal{Q}_2)$ are given by
\begin{align}
\begin{split}
\mathcal{L}_{1}  \varphi&=( \pa_t^2-\pa_t +\Delta_{g_0}) \varphi_1
- 8 e^{4\bar{u}} |\nabla_{g_0}  \bar{v}|^2 \varphi_1 -4 e^{4\bar{u}} \nabla_{g_0} \bar{v} \cdot \nabla_{g_0} \varphi_2,\\
\mathcal{ L}_{2}  \varphi&= (\pa_t^2-\pa_t +\Delta_{g_0}) \varphi_2
+4 \nabla_{g_0} \bar{u}\cdot \nabla_{g_0} \varphi_2 + 4    \nabla_{g_0} \bar{v} \cdot \nabla_{g_0} \varphi_1,
\end{split}
\end{align}
and
\begin{align}
\begin{split}
\mathcal{Q}_{1} \varphi &= (8 e^{4\bar{u}} |\nabla_{g_0}  \bar{v}|^2 - 8 e^{4u} |\nabla_{\hat g}  v|^2 )\varphi_1 
+(4 e^{4\bar{u}} \nabla_{g_0} \bar{v} \cdot \nabla_{g_0} \varphi_2
-4 e^{4u} \nabla_{\hat g} v \cdot \nabla_{\hat g} \varphi_2), \\
\mathcal{Q}_{2} \varphi &= 4 \nabla_{\hat g} u\cdot \nabla_{\hat g} \varphi_2
+ 4    \nabla_{\hat g} v \cdot \nabla_{\hat g} \varphi_1-4 \nabla_{g_0} \bar{u}\cdot \nabla_{g_0} \varphi_2
- 4    \nabla_{g_0} \bar{v} \cdot \nabla_{g_0} \varphi_1.
\end{split}
\end{align}
The operator $\mathcal{T}$ introduced earlier is the spherical part of $\mathcal L$.
We point out that the coefficients of $\mathcal L$ do not depend on $t$. The next result asserts that $\mathcal Qw$ is a perturbative term.

\begin{lem}\label{lemma-Q-decay} 
Let $\bar{C}$ and $\beta$ be as in \eqref{qaohoihqhq}. For any $w=(w_1,w_2)\in C^{1}(B_1 \setminus\Gamma)$, the estimate
\begin{equation}
|\mathcal Q_1w|+e^{2u}|\mathcal Q_2w|\le Ce^{-\beta t}\big(|w_1|+|\nabla_{\hat g}w_1|+e^{2u}|\nabla_{\hat g}w_2|\big)
\end{equation}
is valid with a constant $C$ depending only on $\bar{C}$.
\end{lem}

\begin{proof} 
A straightforward calculation yields
\begin{align}
\begin{split}
|\mathcal Q_1w|+e^{2u}|\mathcal Q_2w|&\le C\big(|u-\bar{u}|+|\nabla_{\hat g}(u-\bar{u})|
+e^{2u}|\nabla_{\hat g}(v-\bar{v})|\big)\\
&\qquad\cdot\big(|w_1|+|\nabla_{\hat g}w_1|+e^{2u}|\nabla_{\hat g}w_2|\big).
\end{split}
\end{align}
The desired result follows from \eqref{qaohoihqhq}, reformulated on the cylinder.
\end{proof}

We now prove Propositions \ref{prop:linear-asymptotic} and \ref{prop:linear-boundness} in the new coordinates. Note that the $f$ appearing in these propositions differs from that of the following result, by a scaling factor of $e^{-2t}$.

\begin{prop}\label{prop:linear-asymptotic-cylinder}
Assume that $f=(f_1, f_2)\in C^1(\mathbb{R}_{+} \times\left(\mathbb{S}^2 \setminus\{N,S\}\right))$ satisfies 
\begin{equation}
|\nabla_{\hat g}^k( f_1- \bar{f}_1) |
+e^{2u} |\nabla_{\hat g}^k (f_2- \bar{f}_2)|\le K e^{-\gamma t}\quad\text{in }\mathbb{R}_{+} \times\left(\mathbb{S}^2 \setminus\{N,S\}\right),\quad\quad k=0,1,
\end{equation}
for some positive constants $K$, $ \gamma$, and a $C^1$-function $\bar{f}=(\bar{f}_1,\bar{f}_2)$ on $\mathbb{S}^2$. 
Let $w$ satisfying $(w_1,e^{2u}w_2)\in L^{\infty}(\mathbb{R}_{+} \times \mathbb{S}^2)$ be a weak solution of
\begin{equation}\label{onaofinoihjhj}
e^{-2t}Lw =f \quad\text{in }\mathbb R_+\times \mathbb S^2 ,\quad\quad w_2 =0 \quad\text{on } \mathbb{R}_{+}\times\{N,S\}.
\end{equation}
Then for any $\varphi\in H^1(\mathbb{S}^2)\times H_0^{1}(\mathbb{S}^2, e^{4 \bar{u}})$ with $\mathcal T\varphi=0$ we have
\begin{equation}
\langle\bar{f}, \varphi\rangle_{L^2(\mathbb{S}^2, 1\times e^{4\bar{u}})}=0,
\end{equation}
and there exists a function $\bar{w}=(\bar{w}_1, \bar{w}_2)\in H^1(\mathbb{S}^2)\times H_0^{1}(\mathbb{S}^2, e^{4 \bar{u}})$
with
\begin{equation}
\mathcal{T}\bar{w}=\bar{f}\quad\text{in }\mathbb S^2 \setminus \{N,S\}. 
\end{equation}
Moreover for all
$0<\bar{\beta}\le \min\{\beta, \gamma\}$ with $\bar{\beta}^2+\bar{\beta}<\mu_2$, $k=0,1$, and any $\alpha\in (0,1)$ it holds that
\begin{align}\label{eq-linear-limit-decay-cylinder}
\begin{split}
&\sup_{\mathbb S^2 \setminus\{N,S\}}\big(|\nabla_{\hat{g}}^k(  w_1(t)-\bar{w}_1)   | + (\sin \theta)^{k-3-\alpha } |\nabla_{\hat{g}}^k(w_2(t)-\bar{w}_2) |\big) \\
&\qquad\le C e^{-\bar \beta t}
\big\{\sup_{\mathbb R_+\times\mathbb (S^2 \setminus\{N,S\})}\big(|w_1 | + e^{2u}|w_2 |\big)
+\sup_{\mathbb R_+\times\mathbb (S^2 \setminus\{N,S\})}\big(|f_1 | + e^{2u}|f_2 |\big)+K\big\},
\end{split} 
\end{align}
where $C$ is a constant depending only on $K$, $\alpha$, $\gamma$, $\beta$, and $\bar{\beta}$.
\end{prop}

\begin{proof} 
The proof consists of two steps. For brevity we will make use of the notation 
\begin{equation}
\|\cdot\|_{L^2}=\|\cdot\|_{L^2(\mathbb{S}^2; 1\times e^{4\bar{u}} )},\quad
\langle\cdot,\cdot\rangle_{L^2}=\langle\cdot,\cdot\rangle_{L^2(\mathbb{S}^2; 1\times e^{4\bar{u}} )},\quad
A=\sup_{t\in[0,\infty)}\big(\|w(t)\|_{L^2}
+\|f(t)\|_{L^2}\big).
\end{equation}

{\it Step 1.} By a local estimate for $w$ (see the proof of \cite[Proposition 3.9]{HKWX}), for any $t\geq 1$ and $k=0,1$ we have 
\begin{equation}\label{eq:time-slice-1}
\sup_{\mathbb{S}^2 \setminus\{N,S\}}\!\!\!\big(|\nabla^k_{\hat g} w_1(t)|+(\sin\theta)^{k-3-\alpha} |\nabla^k_{\hat g} w_2(t) |\big)
\le C\Big(\int_{t-1}^{t+1}\!\!\!\big(\|w(s)\|^2_{L^2}
+\|f(s)\|^2_{L^2}\big)ds\Big)^{1/2}\!\!\le CA.
\end{equation}
It follows that Lemma \ref{lemma-Q-decay} implies 
\begin{equation}
|\mathcal{Q}_{1} w| +e^{2u}|\mathcal{Q}_{2} w| \le
CA e^{-\beta t},\quad\quad t\geq 1.
\end{equation}
By setting $\bar{\gamma}=\min\{\gamma, \beta\}$ we then have that for $t\geq 1$, 
\begin{align}\label{eq-assumption-f-Laplace-special-tilde}
\big\|(f-\bar{f}) - \mathcal{Q} w\big\|_{L^2}\le C(A+K)e^{-\bar{\gamma} t}.
\end{align}
Now write equations \eqref{onaofinoihjhj} as
\be \label{eq:reduced-linear}
w''-w'+\mathcal{T} w=\bar{f}+(f-\bar{f}) - \mathcal{Q} w  \quad \mbox{in }\R_+\times \mathbb{S}^2,
\ee
where the prime notation indicates $t$-derivatives. Taking an inner product with the eigenfunction $\varphi^{(1)}$ in $L^2(\mathbb{S}^2, 1\times e^{4\bar{u}})$ then yields
\begin{equation}
\frac{d^2\,}{dt^2}\langle w, \varphi^{(1)}\rangle_{L^2} -\frac{d\,}{dt}\langle w, \varphi^{(1)}\rangle_{L^2}
=\langle \bar{f}, \varphi^{(1)}\rangle_{L^2}+\langle f-\bar{f}-\mathcal{Q} w, \varphi^{(1)}\rangle_{L^2},
\end{equation}
which is a second order ODE for $\langle w, \varphi^{(1)}\rangle_{L^2}$.
The first term on the right-hand side is constant and the second term decays exponentially as $t\to\infty$
according to \eqref{eq-assumption-f-Laplace-special-tilde}. Therefore
\begin{equation}
\langle w, \varphi^{(1)}\rangle_{L^2}=c_1+c_2e^t-t\langle \bar{f}, \varphi^{(1)}\rangle_{L^2}
+\int_t^\infty (1-e^{t-s})\langle f(s)-\bar{f}-\mathcal{Q} w(s), \varphi^{(1)}\rangle_{L^2} d s,
\end{equation}
for some constants $c_1$ and $c_2$. Since $\langle w, \varphi^{(1)}\rangle_{L^2}$ is bounded, we must have $c_2=0$ and $\langle \bar{f}, \varphi^{(1)}\rangle_{L^2}=0$. Hence, there exists a unique $\bar{w}_*=(\bar{w}_{*1}, \bar{w}_{*2})\in H^1(\mathbb{S}^2)\times H_0^{1}(\mathbb{S}^2, e^{4 \bar{u}})$ satisfying
\begin{equation}\label{eq-special-w}
\mathcal{T}\bar{w}_*=\bar{f}\quad \text{on }\mathbb S^2 
\quad\text{and}\quad \langle \bar{w}_*, \varphi^{(1)}\rangle_{L^2}=0. 
\end{equation}
Moreover, note that
\begin{equation}\label{eq-Laplace-sup-norm-tilde-0}
\big|\langle w(t)-\bar{w}_*, \varphi^{(1)}\rangle_{L^2}-c_1\big|
\le C(A+K)e^{-\bar{\gamma} t},\quad\quad t\geq 1.
\end{equation}

{\it Step 2.} We will now establish \eqref{eq-linear-limit-decay-cylinder}, which
follows from standard ODE analysis. 
Set
\begin{align} \label{eq:Laplace-linear-solution-special-hat}
\widehat w(t)=w(t)-\bar{w}_*-\langle w(t)-\bar{w}_*, \varphi^{(1)}\rangle_{L^2}\varphi^{(1)},\end{align}
and
\begin{equation}
\widehat f(t)=(f(t)-\bar{f}) - \mathcal{Q} w(t)-\langle (f(t)-\bar{f}) - \mathcal{Q} w(t), \varphi^{(1)}\rangle_{L^2}\varphi^{(1)}.
\end{equation}
Then
\begin{equation}\label{eq:Laplace-linear-special-hat}
\widehat{w}''-\widehat{w}'+\mathcal T\widehat  w=\widehat  f  \quad \text{in }\mathbb R_+\times \mathbb{S}^2,
\end{equation}
and
\begin{equation}\label{eq-assumption-f-Laplace-special-hat}
\|\widehat  f(t)\|_{L^2}\le C(A+K)e^{-\bar{\gamma} t},\quad\quad t\geq 1.
\end{equation}
Moreover $\langle \widehat w(t), \varphi^{(1)}\rangle_{L^2}=0$ and hence
\begin{equation}\label{eq-lower-energy-Laplace-special-hat}
\overline{\mathcal{B}}[\widehat w(t), \widehat w(t)]\ge \mu_2
\|\widehat w(t)\|_{L^2}^2,
\end{equation}
where $\overline{\mathcal{B}}$ is the bilinear form defined by \eqref{eq:bilinear-form}.
Taking an inner product of the equation \eqref{eq:Laplace-linear-special-hat}
with $-\widehat w(t)$ in $L^2(\mathbb{S}^2, 1\times e^{4\bar{u}})$ produces
\begin{equation}
\langle \widehat w(t), -\widehat{w}''(t)+\widehat{w}'(t)\rangle_{L^2} + \overline{\mathcal{B}}[\widehat w(t), \widehat w(t)]
=-\langle \widehat w(t), \widehat  f(t)\rangle_{L^2}.
\end{equation}
For any $t\ge 0$ set
\begin{equation}\label{eq-def-y-Laplace-special}
y(t)=\|\widehat w(t)\|_{L^2}
=\Big(\int_{\mathbb{S}^2} \big(\widehat w_1^2(t)+e^{4\bar{u}}\widehat{w}_2^2(t) \big)\, d vol_{g_0}\Big)^{1/2},
\end{equation}
and note that $y$ is differentiable at $t>0$ if $y(t)>0$. At such $t$ we have
\begin{equation}
y(t)y'(t)=\langle \widehat w(t), \widehat  w'(t)\rangle_{L^2},\quad y(t)y''(t)+[y'(t)]^2=\langle \widehat w(t), \widehat  w''(t)\rangle_{L^2}+\langle \widehat w'(t), \widehat  w'(t)\rangle_{L^2}.
\end{equation}
The Cauchy-Schwarz inequality implies that if $y(t)>0$ then $|y'(t)|\leq\| \widehat{w}'(t)\|_{L^2}$
and hence
\begin{equation}
y(t)y''(t)\ge \langle \widehat w(t), \widehat  w''(t)\rangle_{L^2}.
\end{equation}
It follows that for $t\geq 1$ inequalities \eqref{eq-assumption-f-Laplace-special-hat}
and \eqref{eq-lower-energy-Laplace-special-hat} yield
\begin{equation}
-y(t)y''(t)+y(t)y'(t)+\mu_2y(t)^2\le y(t)\|\widehat f(t)\|_{L^2}
\le C(A+K)y(t)e^{-\bar{\gamma} t},
\end{equation}
and thus when $y(t)>0$ we obtain
\begin{equation}\label{eq-ODE-ineq-y-Laplace-special}
-y''(t)+y'(t)+\mu_2y(t)\le C(A+K)e^{-\bar{\gamma} t},\qquad t\geq 1.
\end{equation}

Consider $0<\bar{\beta}\le \bar\gamma$ with $\bar{\beta}^2+\bar{\beta}<\mu_2$.
For some positive constant $B$ to be determined, set $z(t)=Be^{-\bar{\beta} (t-1)}$ and observe that
\begin{equation}
-z''(t)+z'(t)+\mu_2z(t)=B(\mu_2-\bar{\beta}^2-\bar{\beta})e^{-\bar{\beta} (t-1)}.
\end{equation}
We choose $B$ large enough such that
\begin{equation}
B(\mu_2-\bar{\beta}^2-\bar{\beta})>C(A+K)\quad\text{and}
\quad B>\|\widehat w(1)\|_{L^2}.
\end{equation}
Hence, $y(1)<z(1)$ and if $y(t)>0$ then
\begin{equation}\label{eq-ODE-ineq-Laplace-special}
-(y-z)''(t)+(y-z)'(t)+\mu_2(y-z)(t)< 0,\qquad t\geq 1.
\end{equation}
We now claim that $y(t)\le z(t)$ for any $t\geq 1$.
If $y-z$ has a local nonnegative maximum at some $t_*>1$ then
\begin{equation}
(y-z)(t_*)\ge 0, \quad (y-z)'(t_*)=0, \quad (y-z)''(t_*)\le 0,
\end{equation}
and in particular, $y(t_*)\ge z(t_*)>0$, contradicting \eqref{eq-ODE-ineq-Laplace-special} at $t_*$.
Thus, $y-z$ does not have a local nonnegative maximum on $(0,\infty)$.
Assume $(y-z)(t_0)>0$ for some $t_0>1$.
Then $y(t_0)>0$ and \eqref{eq-ODE-ineq-Laplace-special} holds at $t_0$.
If $(y-z)(t)<(y-z)(t_0)$ for some $t>t_0$,
then $y-z$ has a local positive maximum in $(1,t)$, leading to a contradiction.
Therefore, $(y-z)(t)\ge (y-z)(t_0)>0$ for all $t>t_0$.
In particular, $y(t)> z(t)>0$ for all $t>t_0$ so that \eqref{eq-ODE-ineq-Laplace-special} holds whenever $t>t_0$.
Next observe that for any $s>t>t_0$, a similar argument shows that $(y-z)(s)\ge (y-z)(t)$.
In other words, $y-z$ is nondecreasing on $(t_0, \infty)$.
Hence $(y-z)'\ge 0$ on $(t_0, \infty)$, and by \eqref{eq-ODE-ineq-Laplace-special} we have
\begin{equation}
(y-z)''(t)>(y-z)'(t)+\mu_2(y-z)(t)\ge \mu_2(y-z)(t_0), \qquad t>t_0.
\end{equation}
It follows that $y(t)\to\infty$, which contradicts the boundedness of $w$ on $\mathbb R_+\times\mathbb S^2$.
Thus $y(t)\le z(t)$ for all $t\geq 1$, and hence
\begin{equation}
\|\widehat w(t)\|_{L^2}\le Be^{-\bar{\beta}( t-1)},\quad\quad t\geq 1.
\end{equation}

The proof of \cite[Proposition 3.9]{HKWX} applied to equation \eqref{eq:Laplace-linear-special-hat} in $(t-1,t+1)\times \mathbb S^2$, for $t>2$ and $k=0,1$, now implies
\begin{align}\label{eq-Laplace-sup-norm-hat}
\begin{split}
\text{ }\qquad\sup_{\mathbb{S}^2 \setminus\{N,S\}}\!\!\big(|\nabla^k_{\hat g} \widehat w_1(t)|+(\sin\theta)^{k-3-\alpha} |\nabla^k_{\hat g} \widehat w_2(t) |\big)
& \le C\Big(\int_{t-1}^{t+1}\!\!\big(\|\widehat w(s)\|^2_{L^2}
+\|\widehat f(s)\|^2_{L^2}\big)ds\Big)^{1/2}\\
&\le C(A+K)e^{-\bar{\beta} t}.
\end{split}
\end{align}
Note that by \eqref{eq:Laplace-linear-solution-special-hat} we have
\begin{equation}
w(t)-\bar{w}_*-c_1\varphi^{(1)}
=\widehat w(t)+\big[\langle w(t)-\bar{w}_*, \varphi^{(1)}\rangle-c_1\big]\varphi^{(1)}.
\end{equation}
Then for $k=0,1$ the estimates \eqref{eq-Laplace-sup-norm-tilde-0} and \eqref{eq-Laplace-sup-norm-hat} produce
\begin{align}
\begin{split}
&\sup_{\mathbb{S}^2 \setminus\{N,S\}}\big(|\nabla^k_{\hat g} (w_1(t)-\bar{w}_{*1}-c_1\varphi^{(1)}_1)|
+(\sin\theta)^{k-3-\alpha} |\nabla^k_{\hat g}  (w_2(t)-\bar{w}_{*2}-c_1\varphi^{(1)}_2) |\big)\\
&\qquad\le C(A+K)e^{-\bar{\beta} t}.
\end{split}
\end{align}
This establishes the desired result with
$\bar{w}=\bar{w}_*+c_1\varphi^{(1)}.$
\end{proof}

\begin{prop} \label{prop:linear-boundness-cylinder}
Let $f$ and $w$ be as in Proposition \ref{prop:linear-asymptotic-cylinder}.
Then there exists $t_*>10$
depending only on $\gamma$, $a$, as well as $\bar{C}$ and $\beta$ in \eqref{qaohoihqhq}, such that
\begin{align}
\begin{split}
&\||w_1|+e^{2u}|w_2|\|_{L^\infty(\mathbb R_+\times \mathbb S^2)}  \\
&\qquad
\le C\big(\||w_1|+e^{2u}|w_2|\|_{L^\infty([0, t_*]\times \mathbb S^2)}+\| |f_1|+e^{2u} | f_2| \|_{L^\infty(\mathbb R_+\times \mathbb S^2)} +K\big).
\end{split}
\end{align}
Moreover, \eqref{eq-linear-limit-decay-cylinder} holds with the $L^\infty$-norms of $w_1$ and $e^{2u}w_2$ on $[0, \infty)\times \mathbb S^2$ replaced by that on $[0, t_*]\times \mathbb S^2$.
\end{prop}

\begin{proof} 
For brevity we will write $\|\cdot\|_{L^2}=\|\cdot\|_{L^2(\mathbb{S}^2; 1\times e^{4\bar{u}} )}.$
Let $\bar{w}_*$ be as in \eqref{eq-special-w} and set
\begin{equation}\label{eq:Laplace-linear-solution-special-tilde1}
\widetilde w(t)=w(t)-\bar{w}_*,\qquad \widetilde f(t)=(f-\bar{f}) - \mathcal{Q} w.
\end{equation}
Then
\begin{equation}\label{eq:Laplace-linear-special-tilde1}
\widetilde{w}''-\widetilde{w}'+\mathcal T\widetilde w= \widetilde f  \quad \text{in }\mathbb R_+\times \mathbb S^2.
\end{equation}
We will expand $\widetilde w$ and  $\widetilde f$ according to the eigenfunctions $\{\varphi^{(i)}\}$ as
\begin{equation}
\widetilde w(t)= \sum_{i=1}^\infty \widetilde w^{(i)}(t)\varphi^{(i)},
\quad \widetilde f(t)= \sum_{i=1}^\infty \widetilde f^{(i)}(t)\varphi^{(i)}.
\end{equation}
The regularity of $w$ obtained from \cite[Proposition 3.9]{HKWX} guarantees that the expansion can be differentiated term by term so that
\begin{equation}
\frac{d^2}{dt^2}\widetilde w^{(i)}- \frac{d}{dt}\widetilde w^{(i)} -\mu_i \widetilde w^{(i)} = \widetilde  f^{(i)}  \quad \text{in }\mathbb R_+.
\end{equation}
For $i\ge 1$ define
\begin{equation}
\lda_i^+= \frac{1}{2}\big(1+ \sqrt{1+4\mu_i}\big), \quad \lda_i^-= \frac{1}{2}\big(1- \sqrt{1+4\mu_i}\big),
\end{equation}
and note that $\lda_i^+\ge 1$, $\lda_i^-\le 0$, as well as $\lda_1^+=1$, $\lda_1^-= 0$.
We will treat $i=1$ differently from $i\ge 2$. Since $\widetilde w^{(i)}$ is bounded on $\mathbb R_+$,
there exists a constant $c_i^-$ such that
\begin{equation}
\widetilde w^{(i)}(t)=c_i^- e^{\lda_i^- t}+\widetilde w_*^{(i)}(t),
\end{equation}
where for $i=1$ the remaining function takes the form
\begin{equation}\label{eq:ode-representation-1}
\widetilde w_*^{(1)}(t) =  - \int_t^{\infty} e^{t-s}\widetilde  f^{(1)}(s)  \, d s
+  \int_{t}^{\infty} \widetilde  f^{(1)}(s) \, d s,
\end{equation}
and for $i\ge 2$ it becomes
\begin{equation}\label{eq:ode-representation-2}
\widetilde w_*^{(i)}(t) =\frac{1}{\lda_i^- -\lda_i^{+}} \int_t^{\infty} e^{\lda_i^+(t-s) }\widetilde  f^{(i)}(s) \, d s
+ \frac{1}{\lda_i^- -\lda_i^{+}} \int_2^{t} e^{\lda_i^-(t-s) }\widetilde  f^{(i)}(s) \, d s.
\end{equation}
Correspondingly, we may write
\begin{equation}
\widetilde  w(t)=\widetilde{w}_0(t)+\widetilde w_*(t),\quad
\widetilde{w}_0(t)=\sum_{i=1}^\infty c_i^- e^{\lda_i^- t}\varphi^{(i)},
\quad \widetilde {w}_*(t)=\sum_{i=1}^\infty \widetilde w_*^{(i)}(t)\varphi^{(i)}.
\end{equation}

By the explicit expression of $\widetilde w_0(t)$ for $t\ge 2$, it follows that
\begin{equation}
\|\widetilde{w}_0(t)\|_{L^2}^2 =\sum_{i=1}^\infty (c_i^-)^2 e^{2\lda_i^- t}
\le \sum_{i=1}^\infty (c_i^-)^2 e^{2\cdot 2\lda_i^-}
=\|\widetilde{w}_0(2)\|_{L^2}^2.
\end{equation}
Moreover, the triangle inequality then gives
\begin{align}
\begin{split}
\| \widetilde w(t) \|_{L^2}
&\le  \| \widetilde{w}_0(t) \|_{L^2}+\| \widetilde{w}_*(t) \|_{L^2}\
\le  \| \widetilde{w}_0(2) \|_{L^2}+\| \widetilde{w}_*(t) \|_{L^2}\\
&\le  \| \widetilde{w}(2) \|_{L^2}+\| \widetilde{w}_*(2) \|_{L^2}+\| \widetilde{w}_*(t) \|_{L^2}.
\end{split}
\end{align}
For $t\geq 2$ this, together with $w(t)=\widetilde{w}(t)+\bar{w}_*$, implies
\begin{align}\label{eq-estimate-bound-intermediate}
\begin{split}
\| w(t) \|_{L^2}
&\le\| {w}(2) \|_{L^2}+2\| \bar{w}_* \|_{L^2}+\| \widetilde{w}_*(2) \|_{L^2}+\| \widetilde{w}_*(t) \|_{L^2}\\
&\le \| {w}(2) \|_{L^2}+C\| \bar{f} \|_{L^2}+\| \widetilde{w}_*(2) \|_{L^2}+\| \widetilde{w}_*(t) \|_{L^2}.
\end{split}
\end{align}

We next estimate $\widetilde {w}_*$.
By \eqref{eq:ode-representation-2} and the H\"older inequality, we have for $i\ge 2$ and $t\ge 2$ that
\begin{align}
\begin{split}
|\widetilde w_*^{(i)}(t) |^2
&\le  \frac{C}{(\lda_i^+)^3}  \int_t^{\infty} e^{\lda_i^+(t-s) } |\widetilde f^{(i)}(s)|^2  \, d s
+ \frac{C}{(\lda_i^+)^3}  \int_2^{t} e^{\lda_i^-(t-s) } |\widetilde  f^{(i)}(s) |^2 \, d s\\
&\le  C  \int_t^{\infty} e^{t-s} |\widetilde f^{(i)}(s)|^2  \, d s
+ C  \int_2^{t} e^{\lda_2^-(t-s) } |\widetilde  f^{(i)}(s) |^2 \, d s,
\end{split}
\end{align}
where in the second inequality $\lda_i^+\ge \lda_2^+\ge 1$ and $\lda_i^-\le \lda_2^-<0$ were used.
Similarly, for $i=1$ and $t\ge 2$ we find that \eqref{eq:ode-representation-1} yields
\begin{equation}
|\widetilde w_*^{(1)}(t) |^2
\le  C  \int_t^{\infty} e^{t-s} |\widetilde f^{(1)}(s)|^2  \, d s
+ C \Big(\int_{t}^{\infty} |\widetilde  f^{(1)}(s)| \, d s\Big)^{2}.
\end{equation}
Since
\begin{equation}
\| \widetilde w_*(t) \|_{L^2} ^2
=\sum_{i=1}^\infty|\widetilde w_*^{(i)}(t) |^2, \quad
\| \widetilde f(t) \|_{L^2} ^2
=\sum_{i=1}^\infty|\widetilde f^{(i)}(t) |^2,
\end{equation}
for $t\geq 2$ we have 
\begin{equation}
\| \widetilde w_*(t) \|_{L^2} ^2
\le  C  \int_t^{\infty} e^{t-s}\| \widetilde f(s) \|_{L^2} ^2  \, d s
+ C  \int_2^{t} e^{\lda_2^-(t-s) } \| \widetilde f(s) \|_{L^2} ^2\, d s
 + C \Big(\int_{t}^{\infty} |\widetilde  f^{(1)}(s)| \, d s\Big)^{2},
\end{equation}
and hence
\begin{equation}
\| \widetilde w_*(t) \|_{L^2} ^2
\le  C \sup_{s\in [2,\infty)} \| \widetilde f(s) \|_{L^2} ^2
+ C \Big(\int_{t}^{\infty} |\widetilde  f^{(1)}(s)| \, d s\Big)^{2}.
\end{equation}
It is obvious that
\begin{equation}
\| \widetilde f(s) \|_{L^2}\le \| f(s)-\bar{f} \|_{L^2}+\| \mathcal{Q}w(s) \|_{L^2}
\le CKe^{-\gamma s}+\| \mathcal{Q}w(s) \|_{L^2},
\end{equation}
and
\begin{equation}
|\widetilde  f^{(1)}(s)|\le \| \widetilde f(s) \|_{L^2}\le CKe^{-\gamma s}+\| \mathcal{Q}w(s) \|_{L^2},
\end{equation}
so that
\begin{equation}
\| \widetilde w_*(t) \|_{L^2}
\le  C\Big(K
+ \sup_{s\in [2,\infty)} \| \mathcal{Q}w(s) \|_{L^2}
+\int_{t}^{\infty}\| \mathcal{Q}w(s) \|_{L^2} \, d s\Big).
\end{equation}
Furthermore, by Lemma \ref{lemma-Q-decay} we have
\begin{equation}
\| \mathcal{Q}w(s) \|_{L^2}\le Ce^{-\beta s}\big(\| w(s) \|_{L^2}+\| \nabla_{\hat{g}} w(s) \|_{L^2}\big),
\end{equation}
and a local estimate for $w$ (see Proposition 3.9 of \cite{HKWX}) when $s\ge 2$ produces
\begin{align}
\begin{split}
\|\nabla_{\hat{g}} w(s) \|^2_{L^2}&\le C\int_{s-1}^{s+1}(\| f(\tau) \|^2_{L^2}+\| w(\tau) \|^2_{L^2}) \, d\tau\\
&\le C\sup_{[s-1,s+1]}\| f(\tau) \|^2_{L^2}+C\int_{s-1}^{s+1}\| w(\tau) \|^2_{L^2} \, d \tau,
\end{split}
\end{align}
so that
\begin{align*}\| \mathcal{Q}w(s) \|_{L^2}\le Ce^{-\beta s}\Big(\sup_{[s-1,s+1]}\| f(\tau) \|_{L^2}
+\| w(s) \|_{L^2}+\Big(\int_{s-1}^{s+1}\| w(\tau) \|^2_{L^2} \, d \tau\Big)^{1/2}\Big).
\end{align*}
Next note that a simple substitution gives
\begin{align}\label{aoinoinihh}
\begin{split}
\| \widetilde w_*(t) \|_{L^2}
&\le  C\Big(K+ \sup_{s\in [1,\infty)} \| f(s) \|_{L^2}+ \sup_{s\in [1,\infty)} e^{-\beta s}\|w(s) \|_{L^2}\\
&\quad+\int_{t}^{\infty}e^{-\beta s}\|w(s) \|_{L^2} \, d s
+\int_{t}^{\infty}e^{-\beta s}\Big(\int_{s-1}^{s+1}\| w(\tau) \|^2_{L^2} \, d \tau\Big)^{1/2}\, d s\Big)\\
&\le  C\big(K+ \sup_{s\in [1,\infty)} \| f(s) \|_{L^2}+ \sup_{s\in [1,\infty)} e^{-\frac\beta2 s}\|w(s) \|_{L^2}\big);
\end{split}
\end{align}
this holds for any $t\ge 2$. 

By substituting \eqref{aoinoinihh} into \eqref{eq-estimate-bound-intermediate} we obtain
\begin{equation}
\| w(t) \|_{L^2}
\le C\big(\| {w}(2) \|_{L^2}+K+ \sup_{s\in [1,\infty)} \| f(s) \|_{L^2}+ \sup_{s\in [1,\infty)} e^{-\frac\beta2 s}\|w(s) \|_{L^2}\big).
\end{equation}
Next, let $t_*>1$ and break the interval of the last term on the right-hand side into $[1, t_*]$ and $[t_*, \infty)$ to find
\begin{align}
\begin{split}
\sup_{t\in [t_*,\infty)}\| w(t) \|_{L^2}
&\le C\big(\|{w}(2) \|_{L^2}+K+ \sup_{s\in [1,\infty)} \| f(s) \|_{L^2}+ \sup_{s\in [1,t_*]} e^{-\frac\beta2 s}\|w(s) \|_{L^2}\big)\\
&\qquad+Ce^{-\frac\beta2 t_*}\sup_{s\in [t_*,\infty)} \|w(s) \|_{L^2}.
\end{split}
\end{align}
Choosing $t^*>2$ large enough such that $C e^{-\beta t^* /2} <1/2$ then produces
\begin{equation}
\sup_{t\in [t_*,\infty)}\| w(t) \|_{L^2}
\le C\big(K+ \sup_{s\in [1,\infty)} \| f(s) \|_{L^2}+ \sup_{s\in [1,t_*]}\|w(s) \|_{L^2}\big).
\end{equation}
Furthermore by the local estimate \cite[Proposition 3.9]{HKWX}, for any $t>3$ and $k=0,1$, we have
\begin{equation}\label{foanoinoinhjj}
|\nabla_{\hat g}^k w_1(t)|+(\sin\theta)^{k-3-\alpha}|\nabla_{\hat g}^k w_2(t)|
\le C\big(K+ \sup_{s\in [1,\infty)} \| f(s) \|_{L^2}+ \sup_{s\in [1,t_*]}\|w(s) \|_{L^2}\big).
\end{equation}
Following the proof of Proposition \ref{prop:linear-asymptotic-cylinder}, using \eqref{foanoinoinhjj} in place of \eqref{eq:time-slice-1}, yields the desired \eqref{eq-linear-limit-decay-cylinder} with the $L^\infty$-norms of $w_1$ and $e^{2u}w_2$ on $[0, \infty)\times \mathbb S^2$ replaced by that on $[0, t_*]\times \mathbb S^2$.
\end{proof}

\begin{rem}
Given the propositions above, one can improve \eqref{qaohoihqhq} by identifying higher order limits of $(u,v)$ as $x\to 0$. In particular, we find that $\beta>0$ depends only on $\mu_2$.
\end{rem}

\section{Differentiability of Singular Harmonic Maps With Respect to Parameters}
\label{sec:diff-dp} \setcounter{equation}{0}

We now study the harmonic maps discussed in Section \ref{secflow} which depend on parameters, and provide a proof of Theorem \ref{thm:smoothness} and Corollary \ref{cor:linest}.  Let $\mathbf{z}_0=(z_{0,1},\dots, z_{0,N})$ with $z_{0,1}<z_{0,2} <\dots<z_{0,N}$ denote the collection of $z$-coordinates for points $p^0_i=(0,0,z_{0,i})\in\mathbb{R}^3$, and set 
$\delta_{0}:=\frac{1}{8}\min_{1\le i\le N-1}\{z_{0,i+1}-z_{0,i}\}$. Points in $\mathbb{R}^3$ will be denoted by
$y=(y_1,y_2,y_3)$. Let $\bar \chi_1$ be a smooth function on $[0,\infty)$ with support in $[0,\sqrt{2}\delta_0)$, $0\le \bar \chi_1\le 1$, and $\bar \chi_1(t)= 1$ for $t\le \delta_0$, and define
\begin{equation}
\bar \chi(y)=\bar \chi_1(\rho(y)) \cdot \bar \chi_1(|y_3|)
\end{equation}
where $\rho(y)= \sqrt{y_1^2+y_2^2}$. Consider the map $\Pi_{\mathbf{z}}:\R^3 \to \R^3 $ given by
\begin{equation}
y\mapsto  y+\sum_{i=1}^N (z_i-z_{0,i})\bar \chi(y_1,y_2,y_3-z_{0,i}) e_3=:x,
\end{equation}
where $e_3=(0,0,1)$. Note that it is a diffeomorphism when $\mathbf{z}\in B_{\delta_1}^N(\mathbf{z}_0)$ for some small $\delta_1>0$, and that $\Pi_{\mathbf{z}}(p_i^0) =p_{i}$ for every $1\le i\le N$. Furthermore,
the coefficients of the pull back of the Euclidean metric $g_{\mathbf{z}}=\Pi_{\mathbf{z}}^* \pmb{\delta}=(g_{\mathbf{z}})_{ab}dy^a dy^b$ take the form
\begin{equation}\label{eq:metric-perturbation}
(g_{\mathbf{z}})_{ab}(y)=\pmb{\delta}_{ab}+ \sum_{i=1}^N\Big((z_i-z_{0,i})(\pmb{\delta}_{a3}\pa_b \bar \chi  +\pmb{\delta}_{b3}\pa_a \bar \chi)  +(z_i-z_{0,i})^2\pa_{a} \bar{\chi} \pa_{b} \bar{\chi}\Big)(y_1,y_2,y_3-z_{0,i}).
\end{equation}
Notice that
\begin{equation}
(g_{\mathbf{z}})_{ab}=\pmb{\delta}_{ab} \quad \text{on } \cup_{i=1}^N B_{\delta_0}(p_i^0) \cup (\mathbb{R}^3\setminus \cup_{i=1}^N B_{2\delta_0}(p_i^0) ) ,
\end{equation}
and 
\begin{equation}\label{anfoininihnmojj}
\Delta_{g}\ln\rho=\Delta_{\Pi_{\mathbf{z}}^* \pmb{\delta}}\left(\Pi_{\mathbf{z}}^* \ln\rho\right)=\Pi_{\mathbf{z}}^* \left(\Delta_{\pmb{\delta}}\ln \rho\right)=0\quad\text{in }\mathbb{R}^3.
\end{equation}
Moreover, we may choose $0<\va_0\le \min \{\delta_0,\delta_1\}$ small enough so that, in the sense of positive definite matrices, $\frac{1}{2} \pmb{\delta}\le  g_{\mathbf{z}}\le 2\pmb{\delta}$ for $\mathbf{z}\in B_{\va_0}(\mathbf{z}_0)$. For convenience, in what follows we will forgo the subscript $\mathbf{z}$ and simply write $g$ in place of $g_{\mathbf{z}}$.

For simplicity, we shall only consider the case when one component $z_{i_0}$ of $\mathbf{z}$ is varying, so that $z_{i}=z_{0,i}$ for $i\neq i_0$. It will become evident that the general case can be proved similarly, albeit with a tedious but straightforward adaptation. Without loss of generality we may assume that $z_{0,i_0}=0$, and denote $s:=z_{i_0}\in (-\delta_0, \delta_0)$. To streamline the notation set
\begin{equation}\label{eq:definition-u-s}
\Phi(s, y)=(u(s,y),v(s,y)):=\Phi_{\mathbf{z}}(\Pi_{\mathbf{z}}(y)),
\end{equation}
where $\Phi_{\mathbf{z}}$ is the harmonic map defined in Section \ref{secflow}. In the $y$ coordinates, the harmonic map system becomes
\begin{equation}\label{eq:perturbation}
\Delta_{g} u- 2e^{4u} |\nabla_{g} v|^2 =0,\quad\quad
\Delta_{g} v+4 \nabla_{g} u \cdot\nabla_{g} v=0.
\end{equation}
Let $\sigma$ be a smooth function on $\mathbb R^3\setminus \{p_1^0,\dots, p_N^0\}$ such that
\begin{equation}
\sigma(y)= \begin{cases}
\text{dist}_{\pmb{\delta}}(y, \{p_1^0,\dots, p_N^0\} ) &\text{if }\text{dist}_{\pmb{\delta}}(y, \{p_1^0,\dots, p_N^0\} ) \le 2 \delta_0,\\
|y| &\text{if }|y| \ge  2 \sum_{i=1}^{N} |z_i|+1,
\end{cases}
\end{equation}
and
\begin{equation}
\frac{1}{2}\text{dist}_{\pmb{\delta}}(y, \{p_1^0,\dots, p_N^0\} ) \le \sigma(y)\le 2 \text{dist}_{\pmb{\delta}}(y, \{p_1^0,\dots, p_N^0\} ).
\end{equation}

\begin{lem} \label{lem:compactness-1}
Let $\Phi=\Phi(s)=(u(s), v(s))$ be given by \eqref{eq:definition-u-s}, for any $s\in [-\varepsilon_0, \varepsilon_0]$.
Then 
\be \label{eq:global-distance-bound}
d_{\mathbb{H}^2} (\Phi(s),\Phi(0))\le \Lda\quad\text{in }\mathbb R^3\setminus \Gamma, 
\ee
where $\Lda$ is a positive constant depending only on $\va_0$ and $\{p_1^0,\dots, p_N^0\}$.
Moreover for any $\mathbf{r}\ge 2(|z_{0,1}|+|z_{0,N}|)$, $\alpha \in (0,1)$, and $k=1,2,3$ we have 
\be \label{eq:tangent-uniform-1}
|\nabla_g ^k \Big(u(s)+\ln \frac {\rho} {\sigma} \Big) |
+ \Big(\frac{\rho}{\sigma }\Big)^{k-3-\alpha} |\nabla_g ^k v(s)|
\le C(\mathbf{r})\sigma ^{-k}  \quad \mbox{in }B_{\mathbf{r}}\setminus\Gamma,
\ee
where $C(\mathbf{r})$ is a constant independent of $s$.
\end{lem}

\begin{proof} 
Both estimates are consequences of slight modifications of the proofs from previous results, where the flat metric $\pmb\delta$ is replaced by $g$. In particular, the first conclusion follows from Proposition 2.1 of \cite{CLW} (see also \cite{WeinsteinHadamard}). The second one follows from Proposition 3.8 and Theorem 2.3 of \cite{HKWX}, together with standard estimates in regions away from the axis. 
\end{proof}

\begin{rem}
Notice that although the distance between $\Phi(s)$ and $\Phi(0)$ is defined only on $\mathbb{R}^3\setminus \Gamma$, according to \cite[Theorems 2.1 and 2.3]{HKWX} it lies within $C^{3,\alpha}(\mathbb R^3\setminus \{p_1^0,\dots, p_N^0\})$ for any $\alpha\in(0,1)$.
\end{rem}

By Theorem 2.1 of \cite{HKWX} (see also \eqref{qaohoihqhq} above), for each $s\in [-\va_0,\va_0]$ and each $p_i$, $1\le i\le N$ there exists a unique tangent map $\bar \Phi_i=(\bar u_{i},\bar v_{i})$ from $\mathbb{S}^2\setminus \{N,S\}$ to $\mathbb{H}^2$, such that
for any $l+k\le 3$ we have  
\be\label{eq:HKWX-main-1}
\begin{split}
\sup_{\mathbb{S}^2 \setminus\{N,S\}} \Big(& |(r_i \pa_{r_i})^l \nabla_{\mathbb{S}^2} ^k (u(s,r_i,\theta_i)-\bar u_{i}(s,\theta_i) ) |
\\& +e^{(3+\alpha-k)\bar u_{i}} |(r_i \pa_{r_i})^l \nabla_{\mathbb{S}^2} ^k (v(s,r_i,\theta_i) -\bar v_i(s,\theta_i)  ) |\Big)
\le C_1(s)r_i^{\beta}
\end{split}
\ee
in $B_{\delta_0}(p_i^0)\setminus \{p_i^0\}$, where $\beta>0$ is independent of $s$, $r_i= |y-p_i^0|$, and $\sin\theta_i= \tfrac{\rho(y)}{r_i}$.
Furthermore, $(\bar u_{i}+\ln \frac{\rho}{r_i},\bar v_{i})\in C^{3,\alpha } (\mathbb{S}^2)$
for any $\alpha \in (0,1)$, and $\bar v_{i}=c_{i}$ at the south pole $S$ while $\bar v_{i}=c_{i+1}$ at the north pole $N$; see \eqref{eq:tange-map} for the explicit form. We will now prove that the dependence of $C_1(s)$ on $s$ can be removed.

\begin{lem} \label{lem:interior-uniform}
Let $\Phi=\Phi(s)$ be given by \eqref{eq:definition-u-s}, for any $s\in [-\varepsilon_0, \varepsilon_0]$.
Then the constant $C_1(s)$ in \eqref{eq:HKWX-main-1} is uniformly bounded for $s\in [-\va_0, \va_0]$, and hence may be chosen independent of $s$.
\end{lem}

\begin{proof}
Take an arbitrary sequence $s_n\in [-\va_0, \va_0]$, and
let $\Phi(s_n)$ be the corresponding solutions of \eqref{eq:perturbation}.
By \eqref{eq:global-distance-bound} and \eqref{eq:tangent-uniform-1}, the Arzel\`a-Ascoli theorem may be applied to extract a subsequence, still denoted by $\Phi(s_n)$, such that 
\begin{equation}
\lim_{n,m\to \infty}d_{\mathbb{H}^2} (\Phi(s_n), \Phi(s_{m})) =0
\quad \text{in }\overline B_{\delta_0}(p_i^0)\setminus (B_{\delta_0/2}(p_i^0) \cup \Gamma).
\end{equation}
By \eqref{eq:global-distance-bound}, we have
$d_{\mathbb{H}^2} (\Phi(s_n), \Phi(s_{m})) \le C$ in $B_{\delta_0}(p_n^0)$. 
Note also that the negative curvature of the target space yields 
\begin{equation}
\Delta \sqrt{1+d_{\mathbb{H}^2} (\Phi(s_n), \Phi(s_{m}))^2} \ge 0  \quad \text{in }B_{\delta_0}(p_i^0) \setminus \Gamma.
\end{equation}
Therefore, the maximum principle \cite[Lemma 8]{WeiDuke} implies
\begin{equation}
\max_{B_{\delta_0}(p_i^0)\setminus\{p_i^0\}}d_{\mathbb{H}^2} (\Phi(s_n), \Phi(s_{m}))
\le \max_{\pa B_{\delta_0}} d_{\mathbb{H}^2} (\Phi(s_n), \Phi(s_{m})).
\end{equation}
These estimates show that the constants $C_*$ and $C$ from Lemmas 5.6 and 5.7 of \cite{HKWX}, are independent of $n$. It then follows from the proof of \cite[Theorem 5.8]{HKWX} that $C_1(s_{n})$ is uniformly bounded.
The desired result may now be obtained from a simple contradiction argument.
Suppose that there exists a sequence $s_n\in [-\va_0, \va_0]$ 
such that $C_1(s_n)\to \infty$ as $n\to\infty$.
The argument above shows that $C_1(s_n)$ is bounded, leading to a contradiction.
\end{proof}

Let $\mathbf{r}_0>8(|z_{0,1}|+|z_{0,N}|)$ be a large number. According to Theorem 2.3 of \cite{HKWX}, for each $s$ and any $\alpha\in(0,1)$
there exist constants $\nu_0$, $\nu_1$, $\nu_2$, $C$, and spherical harmonics $Y_1$, $Y_2$ of degrees $1$ and $2$ such that within $B_{\mathbf{r}_0}^c$ we have
\begin{equation}\label{eq:HKWX-main-2}
\begin{split}
&\big|u (s, r, \theta)+\ln \rho -\nu_0-\nu_1 r^{-1}-Y_1 r^{-2}-Y_2 r^{-3}\big|_{C^3(\mathbb{S}^2)} \leq Cr^{-3-\beta_0},\\
&\big|v (s, r, \theta) - \mathcal{J}_0\cos \theta (3-\cos^2 \theta )-\nu_2 r^{-1} \sin^ 4 \theta \big |_{C^3(\mathbb{S}^2)}
\leq Cr^{-1-\beta_0}(\sin\theta)^{3+\alpha},
\end{split}
\end{equation}
where $r=|y|$, $\mathcal{J}_0$ is a multiple of total angular momentum, and $\beta_0\in(0,1)$ is independent of $s$.  Moreover, the corresponding asymptotics are valid for the $r$-derivatives of $(u,v)$.
We point out that the constants $\nu_0$, $\nu_1$, $\nu_2$, $C$ and functions
$|Y_1|$, $|Y_2|$ have uniform bounds for $s\in [-\va_0,\va_0]$, which depend only on $\Lda$ and $\mathbf{r}_0$.

\begin{lem}\label{lem:exterior-uniform}
Let $\Phi=\Phi(s)$ be given by \eqref{eq:definition-u-s}, for any $s\in [-\varepsilon_0, \varepsilon_0]$.
Then $\nu_0=0$ for all $s\in [-\va_0,\va_0]$ and
\begin{equation}
d_{\mathbb{H}^2} (\Phi(s,y),\Phi(0,y)) \le C r^{-1} \quad \mbox{in }B_{\mathbf{r}_0}^c,
\end{equation}
where $C$ is a positive constant independent of $s$.
\end{lem}

\begin{proof}
For each fixed $s$, consider the original $x$ coordinates.
In this setting, recall the construction of $\Phi$ in \cite{CLW}.
In the proof of \cite[Proposition 2.1]{CLW}, a model (or reference) map $\tilde{\Phi}=(\tilde{u}, \tilde{v})$ is constructed such that $\tilde{v}=c_i$ on the interval $\Gamma_i\subset \Gamma$, $1\leq i\leq N$ and
$\tilde{\Phi}$ coincides, for all large values of $|x|$, with the extreme Kerr solution having potential constant value $c_{N+1}$ for $x_3>z_{N}$
and value $c_1$ for $x_3<z_{1}$ on the $x_3$-axis.
Next, one solves the Dirichlet problem for the harmonic map system in a sequence of exhausting domains $\om_n= \{x\mid\rho(x)>1/n, |x|<n \}$, with boundary values given by the above model map; this generates a sequence of harmonic maps $\Phi_n$ in $\Omega_n$. It is then shown that there is a uniform hyperbolic distance bound between $\Phi_n$ and $\tilde{\Phi}$ on compact subsets.
The desired solution $\Phi$ is obtained by sending $n$ to infinity and passing to a subsequence. 
Furthermore, by Theorem 4.1 of \cite{KhuriWeinstein} (with motivation from \cite{SchoenZhou}) we have
\begin{equation}
\left(\int_{\R^3} d_{\mathbb{H}^2} (\Phi, \tilde{\Phi} )^6 dx\right)^{1/3}  <\infty.
\end{equation}
On the other hand, for large $|x|$ the expansions \eqref{eq:HKWX-main-2} imply
\begin{equation}\label{afoniaoinionihhh}
\cosh \left(2d_{\mathbb{H}^2} (\Phi, \tilde{\Phi} )\right) =\cosh \left(2(u -\tilde{u})\right)+2 e^{2(u +\tilde{u})} (v -\tilde{v})^2
= \cosh [2\nu_0+O(|x|^{-1})] +O(|x|^{-2}),
\end{equation}
from which it follows that $\nu_0=0$.
\end{proof}

We will begin the study of regularity for $\Phi(s)$, with respect to the $s$-parameter, by establishing continuity in an $L^6$-sense. 

\begin{prop}\label{prop:integral-continuity}
Let $\Phi=\Phi(s)$ be given by \eqref{eq:definition-u-s}, for any $s\in [-\varepsilon_0, \varepsilon_0]$.
Then for any $s_1,s_2\in [-\va_0,\va_0]$ we have
\begin{equation}
\left(\int_{\R^3} d_{\mathbb{H}^2} (\Phi(s_1), \Phi(s_2))^6dy\right)^{1/6} \le C|s_1-s_2|^{1/2},
\end{equation}
where $C$ is a constant independent of $s_1$ and $s_2$.
\end{prop}

\begin{proof}
Since $\rho(y)=\rho(\Pi_{\mathbf{z}}(y))$, we may renormalize the first term of the harmonic maps with the same expression in $x$ or $y$-coordinates by $U= u +\ln \rho$, and define the renormalized energy in $y$-coordinates with
\begin{equation}
\mathcal{E}_s(\Phi(s))= \int_{\R^3} (|\nabla_{g(s)} U(s)|^2+  e^{4u(s) } |\nabla_{g(s)} v(s) |^2 )\, dvol_{g(s)}.
\end{equation}
Using that $\Phi(s_j)$ is harmonic with respect to $g(s_j)$, $j=1,2$ the estimate of \cite[Theorem 4.1]{KhuriWeinstein} (see also \cite[Theorem 1.1]{SchoenZhou}) gives
\begin{align}
\begin{split}
&\mathcal{E}_{s_1} (\Phi(s_2)) - \mathcal{E}_{s_1} (\Phi(s_1)) \ge C\left(\int_{\R^3} d_{\mathbb{H}^2} (\Phi(s_1), \Phi(s_2))^6dy\right)^{1/3},\\
&  \mathcal{E}_{s_2} (\Phi(s_1)) - \mathcal{E}_{s_2} (\Phi(s_2))\ge C \left(\int_{\R^3} d_{\mathbb{H}^2} (\Phi(s_1), \Phi(s_2))^6dy\right)^{1/3},
\end{split}
\end{align}
where $C$ is a positive constant depending only on $\va_0$, independent of $s_1$ and $s_2$.
We point out that these estimates were proved in the original coordinates $x$, however in light of \eqref{anfoininihnmojj}
the same arguments may be applied in the current $y$-coordinate setting. Next observe that
\begin{align}
\begin{split}
&\mathcal{E}_{s_1} (\Phi(s_2)) - \mathcal{E}_{s_1} (\Phi(s_1)) +  \mathcal{E}_{s_2} (\Phi(s_1)) - \mathcal{E}_{s_2} (\Phi(s_2))\\
&\qquad
= \mathcal{E}_{s_1} (\Phi(s_2)) - \mathcal{E}_{s_2} (\Phi(s_2)) +  \mathcal{E}_{s_2} (\Phi(s_1))-  \mathcal{E}_{s_1} (\Phi(s_1)).
\end{split}
\end{align}
By a straightforward computation the terms on the right-hand side may be estimated by
\begin{align}
\begin{split}
&(|\nabla_{g(s_1)} U|^2+  e^{4u } |\nabla_{g(s_1)} v |^2 ) \sqrt{\det g(s_1)}
 -(|\nabla_{g(s_2)} U|^2+  e^{4u } |\nabla_{g(s_2)} v |^2 ) \sqrt{\det g(s_2)}\\
&
= (|\nabla_{g(s_1)} U|^2+  e^{4u } |\nabla_{g(s_1)} v |^2 ) (\sqrt{\det g(s_1)}- \sqrt{ \det g(s_2) })\\& \quad +
[ (|\nabla_{g(s_1)} U|^2+  e^{4u } |\nabla_{g(s_1)} v |^2 ) - (|\nabla_{g(s_2)} U|^2+  e^{4u } |\nabla_{g(s_2)} v |^2 )] \sqrt{\det g(s_2)}\\
& \le  C (|\nabla_{g(s_1)} U|^2+  e^{4u } |\nabla_{g(s_1)} v |^2 ) |s_1-s_2| \cdot \chi_{\{\text{supp}(g(s_1)-g(s_2))\}}\\&
\quad + C |(g^{ij}(s_1)-g^{ij}(s_2)) \pa_i U\pa_j U| +e^{4u} |( g^{ij}(s_1)-g^{ij}(s_2)) \pa_i v \pa_j v  |,
\end{split}
\end{align}
where $\chi_{\{\text{supp}(g(s_1)-g(s_2))\}}$ is the indicator function of the support set of $g(s_1)-g(s_2)$.
Since $g(s_1)-g(s_2)$ is supported in $B_{2\delta_0}\setminus B_{\delta_0}$, we may apply \eqref{eq:tangent-uniform-1}
to obtain the desired result.
\end{proof}

We are now able to establish the H\"older continuity of $\Phi$ with respect to $s$.

\begin{prop}\label{prop:local-continuity}
Let $\Phi=\Phi(s)$ be given by \eqref{eq:definition-u-s}, for any $s\in [-\varepsilon_0, \varepsilon_0]$. Then for any $s_1, s_2\in [-\varepsilon_0, \varepsilon_0]$ we have 
\begin{equation}\label{aofnoinighh}
d_{\mathbb{H}^2} (\Phi(s_1,y),\Phi(s_2,y)) \le C  (1+|y|)^{-1} |s_1-s_2|^{1/2} \quad\text{in }\mathbb R^3\setminus \Gamma, 
\end{equation}
and for $k=0,1,2,3$ as well as any $\alpha \in (0,1)$ it holds that
\begin{align}\label{oifnoiqntoinqionh}
\begin{split}
&|\nabla_g ^k (u (s_1)-u (s_2)) |(y) + \Big(\frac{\rho}{\sigma}\Big)^{k-3-\alpha } |\nabla_g ^k (v (s_1)-v (s_2))  |(y)\\
&\qquad
\le C \sigma ^{-k}   (1+|y|)^{-1}  |s_1-s_2|^{1/2} \quad\text{in }\mathbb R^3\setminus \Gamma, 
\end{split}
\end{align}
where $C$ is a constant independent of $s_1$ and $s_2$, and $g=g(s)$ for any $s\in[-\varepsilon_0,\varepsilon_0]$. 
\end{prop}

\begin{proof}
Recall that $p_{i_0}^0$ is the origin and $z_{0,1}<\dots<z_{0,N}$.
It follows from Proposition \ref{prop:integral-continuity} above and (the proof of) Proposition 3.9 of \cite{HKWX} that
for any $\mathbf{r}>2(|z_{0,1}|+|z_{0,N}|)$ and $\alpha\in (0,1)$ there exists a constant $C$, depending only on $\mathbf{r}$, $\delta_0$, and the boundary data $c_1,\dots, c_{N+1}$ such that for all $s_1,s_2\in [-\va_0,\va_0]$ and $k=0,1,2,3$ we have
\begin{align}
\begin{split}
&|\nabla_g ^k (u (s_1)-u (s_2)) | + \Big(\frac{\rho}{\sigma}\Big)^{k-3-\alpha } |\nabla_g ^k (v (s_1)-v (s_2))  |\\
&\qquad \le
C |s_1-s_2|^{1/2} \quad \mbox{in }B_{\mathbf{r}}\setminus \left(\Gamma\cup_{i=1}^N B_{\delta_0}(p_i^0)\right).
\end{split}
\end{align}
With the aid of \eqref{afoniaoinionihhh}, this in particular implies 
\begin{equation}
\sup_{B_{\mathbf{r}}\setminus \cup_{i=1}^N B_{\delta_0}(p_i^0) } d_{\mathbb{H}^2} (\Phi(s_1,y),\Phi(s_2,y))  \le C  |s_1-s_2|^{1/2}.
\end{equation}
Furthermore, since the metrics $g(s_1)=g(s_2)=\pmb{\delta}$ agree in $B_{\delta_0}\cup B_{2\delta_0}^c$, it follows that
\begin{equation}
\Delta \sqrt{1+d_{\mathbb{H}^2} (\Phi(s_1,y),\Phi(s_2,y))^2} \ge 0 \quad \mbox{in } B_{\delta_0}\cup B_{2\delta_0}^c.
\end{equation}
The first conclusion then follows from a maximum principle argument as in the proof of Lemma \ref{lem:interior-uniform}, together with Lemma \ref{lem:exterior-uniform}. This inequality \eqref{aofnoinighh}, in turn may be used to obtain the second conclusion from the proof of \cite[Proposition 3.9]{HKWX}.
\end{proof}

In the remainder of this section, we will study higher order regularity with respect to $s$.
Let $g=g(s)$ be the metric as in \eqref{eq:metric-perturbation}.
In view of \eqref{eq:perturbation} and $\Phi=(u , v )$, define
\begin{equation}
\mathcal{F}_1(s,\Phi) = \Delta_{g} u- 2e^{4u} |\nabla_{g} v|^2, \qquad
\mathcal{F}_2(s,\Phi) =\Delta_{g} v+4 \nabla_{g} u\cdot  \nabla_{g} v,
\end{equation}
and write $\mathcal{F}(s,\Phi)= (\mathcal{F}_1(s,\Phi),\mathcal{F}_2 (s,\Phi))$.
Let $L_s=(L_{s,1}, L_{s,2})$ be the linearized operator of $\mathcal{F}$ at $\Phi$.
A straightforward computation yields
\begin{align} \label{eq:linearized-operator}
\begin{split}
L_{s,1}\varphi &= \Delta_{g} \varphi_1 - 8 e^{4u} |\nabla_{g}  v|^2 \varphi_1 -4 e^{4u} \nabla_{g} v \cdot \nabla_{g} \varphi_2,\\
L_{s,2}\varphi&= 
\Delta_{g} \varphi_2+4  \nabla_{g} u \cdot \nabla_{g} \varphi_2+4    \nabla_{g} v \cdot \nabla_{g} \varphi_1,
\end{split}
\end{align}
for $\varphi=(\varphi_1, \varphi_2) \in C_c^1(\R^3)\times C_c^1(\R^3\setminus \Gamma)$. Define
\begin{align}
\begin{split}
\mathcal{B}_s[\varphi,\psi]= &\int_{\R^3} \big(\nabla_g \varphi_1\cdot \nabla_g \psi_1
+8 e^{4u} |\nabla_{g}  v|^2 \varphi_1\psi_1 +e^{4u} \nabla_{g} \varphi_2\cdot \nabla_g \psi_2\big )\,  dvol_g   \\
& +   4\int_{\R^3}  e^{4u} \big( ( \nabla_{g} v \cdot \nabla_{g} \varphi_2)  \psi_1
-  ( \nabla_{g} v \cdot \nabla_{g} \varphi_1)  \psi_2  \big)\,  dvol_g.
\end{split}
\end{align}
This is the quadratic form associated with $L_s$.

\begin{prop}\label{prop:coercive-1}
Let $\Phi=\Phi(s)$ be given by \eqref{eq:definition-u-s}, for any $s\in [-\varepsilon_0, \varepsilon_0]$.
Then for all $\varphi=(\varphi_1, \varphi_2)$ with $\varphi_1\in C_c^1(\R^3)$ and $\varphi_2 \in C_c^1(\R^3\setminus \Gamma)$ we have
\begin{equation}
\mathcal{B}_s[ \varphi, \varphi]\ge C\left (\int_{\R^3} (|\varphi_1| +e^{2u } |\varphi_2|)^6dy\right)^{1/3},
\end{equation}
where $C$ is a positive constant independent of $s$.
\end{prop}

\begin{proof}  
For any $t>0$ there exists $t_*\in (0,t)$ such that
\begin{equation}
\mathcal{E}_s(\Phi+t\varphi )- \mathcal{E}_s(\Phi) = t \frac{d}{dt}\mathcal{E}_s(\Phi+t\varphi )\Big|_{t=0}
+ \frac{t^2}{2} \frac{d^2}{d t^2} \mathcal{E}_s(\Phi+t\varphi )\Big|_{t=t_*}=  \frac{t^2}{2} \frac{d^2}{d t^2} \mathcal{E}_s(\Phi+t\varphi )\Big|_{t=t^*}.
\end{equation}
Moreover, as in the proof of Proposition \ref{prop:integral-continuity} the estimate of \cite[Theorem 4.1]{KhuriWeinstein} applies to the current setting and yields
\begin{equation}
\mathcal{E}_s(\Phi+t\varphi )- \mathcal{E}_s(\Phi)\ge   C\left (\int_{\R^3} d_{\mathbb{H}^2} (\Phi+t\varphi, \Phi)^6 dy\right)^{1/3},
\end{equation}
where $C$ is a positive constant independent of $s$.
Therefore
\begin{equation}
\mathcal{B}_s[ \varphi, \varphi] = \lim_{t\to 0} \frac{1}{2} \frac{d^2}{d t^2} \mathcal{E}_s(\Phi+t\varphi )\Big|_{t=t_*}
 \ge  \liminf_{t \to 0} C\left (\int_{\R^3} \Big( \frac{d_{\mathbb{H}^2} (\Phi+t\varphi, \Phi)}{t}\Big)^6dy\right)^{1/3}.
\end{equation}
Next note that the hyperbolic distance expression \eqref{afoniaoinionihhh} implies
\begin{equation}
d_{\mathbb{H}^2} (\Phi+t\varphi, \Phi) = \frac{1}{2} \ln (\zeta+ \sqrt{\zeta^2-1}),\qquad
\zeta= \cosh (2t \varphi_1)+2 e^{2(2u +t \varphi_1 )} t^2 \varphi_2^2.
\end{equation}
Since $\zeta= 1+ 2 t^2 (\varphi_1^2+  e^{4u }  \varphi_2^2 )  +O(t^3)$ we obtain
\begin{equation}
\zeta+ \sqrt{\zeta^2-1} = 1+ 2 t  \sqrt{\varphi_1^2+  e^{4u }  \varphi_2^2 } +O(t^2) ,
\end{equation}
and hence
\begin{equation}
\lim_{t \to 0} \frac{d_{\mathbb{H}^2} (\Phi+t\varphi, \Phi)}{t} =    \sqrt{\varphi_1^2+  e^{4u }  \varphi_2^2 }.
\end{equation}
The stated result then follows easily. 
\end{proof}

For each $s\in [-\varepsilon_0, \varepsilon_0]$,
denote by $\mathcal{H}$ the closure of $C_c^1(\R^3)\times  C_c^1(\R^3\setminus \Gamma)$ under the norm
\begin{equation}
\|\varphi \|_{\mathcal{H}}= \sqrt{\mathcal{B}_s[ \varphi, \varphi]}.
\end{equation}
By Proposition \ref{prop:coercive-1} we have the following Sobolev type inequality
\be\label{eq:sobolev-type}
\left (\int_{\R^3} (|\varphi_1| +e^{2u } |\varphi_2|)^6\, d y\right)^{1/6}  \le C\| \varphi\|_{\mathcal{H}} \quad \mbox{for any }\varphi\in \mathcal{H},
\ee
where $C$ is a positive constant independent of $s$.
We point out that the space $\mathcal H$ is independent of $s$ and
the norms $\|\varphi\|_{\mathcal H}$ corresponding to different $s\in [-\va_0,\va_0]$ are comparable.

Consider the linear equation with boundary condition 
\be \label{eq:model-linear}
L_{s} w= f \quad \mbox{in }\R^3 \setminus\Gamma,\qquad w_2=0 \quad \mbox{on }\Gamma\setminus \{p_1^0,\dots, p_N^0\}.
\ee
Given $f=(f_1,f_2)$ with $f_1, f_2 \in L^1_{loc}(\mathbb{R}^3)$, we say that $w\in \mathcal{H}$ is a weak solution of \eqref{eq:model-linear} if
\begin{equation}
\mathcal{B}_s[w, \varphi]=-\int_{\R^3} (f_1 \varphi_1 + e^{4u }f_2 \varphi_2) dvol_g,
\end{equation}
for all $\varphi=(\varphi_1, \varphi_2)\in C_c^1(\R^3)\times C_c^1(\R^3\setminus \Gamma)$.

\begin{lem}\label{lem:linear-energy} 
Suppose that $w\in \mathcal{H}$ is a weak solution of \eqref{eq:model-linear} with $f$ as above and satisfying
\begin{equation}
|f_1(y)|+e^{2u (y)} |f_2(y)| \le C_1 \sigma(y)^{-2}   (1+|y|)^{-1} \quad \mbox{for }y\in \R^3\setminus\Gamma. 
\end{equation}
Then
\begin{equation}
\|w\|_{\mathcal{H}}\le C C_1,
\end{equation}
where $C$ is a constant independent of $s\in [-\va_0,\va_0]$ and $w$.
\end{lem}

\begin{proof} 
By the density of compact support functions in $\mathcal{H}$, together with H\"older's inequality, and the definition of weak solutions we have 
\begin{align}
\begin{split}
\mathcal{B}_s[w, w]&= -\int_{\R^3} (f_1 w_1 + e^{4u }f_2 w_2) \,  dvol_g\\&
\le CC_1 \int_{\R^3} \sigma(y) ^{-2}   (1+|y|)^{-1}   (|w_1| + e^{2u }|w_2|) \, dy \\&
\le CC_1 \left(\int_{\R^3} \sigma^{-\frac{12}{5}} (1+|y|)^{-\frac{6}{5}} \,d y\right )^{\frac{5}{6}}
\left(\int_{\R^3}  (|w_1| + e^{2u }|w_2|)^6 \,d y\right )^{\frac{1}{6}}   \\&
\le CC_1 \left(\int_{\R^3}  (|w_1| + e^{2u }|w_2|)^6 \,d y\right )^{\frac{1}{6}} .
\end{split}
\end{align}
The lemma then follows from \eqref{eq:sobolev-type} immediately.
\end{proof}

It is now possible to improve Proposition \ref{prop:local-continuity} to obtain Lipschitz regularity in $s$.

\begin{prop}\label{prop:lip}  
Let $\Phi=\Phi(s)$ be given by \eqref{eq:definition-u-s}, for any $s\in [-\varepsilon_0, \varepsilon_0]$.
Then for any $s_1, s_2\in  [-\varepsilon_0, \varepsilon_0]$ we have 
\begin{equation}
 d_{\mathbb{H}^2} (\Phi(s_1,y),\Phi(s_2,y)) \le C (1+|y|)^{-1} |s_1-s_2|\quad\text{in }\mathbb R^3\setminus\Gamma,
\end{equation}
and for $k=0,1,2,3$ as well as any $\alpha \in (0,1)$ it holds that
\begin{align}
\begin{split}
&|\nabla_g ^k (u (s_1)-u (s_2)) |(y) + \Big(\frac{\rho}{\sigma}\Big)^{k-3-\alpha } |\nabla_g ^k (v (s_1)-v (s_2))|(y)\\
&\qquad
\le C \sigma ^{-k}(1+|y|)^{-1}|s_1-s_2|\quad\text{in }\mathbb R^3\setminus\{p_1^0,\dots, p_N^0\},
\end{split}
\end{align}
where $C$ is a constant independent of $s_1$ and $s_2$, and $g=g(s)$ for any $s\in[-\varepsilon_0,\varepsilon_0]$. 
\end{prop}

\begin{proof}
Take any $s, \tau$ such that $s, s+\tau\in [-{\va_0},{\va_0}]$. We will first derive an equation satisfied by
$\Phi(s+\tau)-\Phi(s)$. Note that $\mathcal{F}(s+\tau,\Phi(s+\tau))=\mathcal{F}(s,\Phi(s))=0$, and thus
\begin{align}
\begin{split}
0&=\mathcal{F}(s+\tau,\Phi(s+\tau))-  \mathcal{F}(s,\Phi(s)) \\
&=  \mathcal{F}(s+\tau,\Phi(s+\tau))- \mathcal{F}(s,\Phi(s+\tau)) + \mathcal{F}(s,\Phi(s+\tau))-\mathcal{F}(s,\Phi(s))\\
&\qquad -L_s\big(\Phi(s+\tau)-\Phi(s)\big)+L_s\big(\Phi(s+\tau)-\Phi(s)\big).
\end{split}
\end{align}
This may be rewritten as
\begin{equation}\label{eq:equation-difference}
L_s\big(\Phi(s+\tau)-\Phi(s)\big)=P(s,\tau; y)+Q(s,\tau; y),
\end{equation}
where
\begin{align}
\begin{split}
P(s,\tau; y)&=-\big[\mathcal{F}(s+\tau,\Phi(s+\tau))-  \mathcal{F}(s,\Phi(s+\tau))\big], \\
Q(s, \tau; y)&= -\big[\mathcal{F}(s,\Phi(s+\tau))-\mathcal{F}(s,\Phi(s)) -L_s\big(\Phi(s+\tau)-\Phi(s)\big)\big].
\end{split}
\end{align}
We point out that $P$ involves $g(s+\tau)-g(s)$ and hence is supported in $B_{2\delta_0}\setminus B_{\delta_0}$, while
$Q$ is quadratic in $\Phi(s+\tau)-\Phi(s)$. Next set $P=(P_{1}, P_{2})$ and $Q=(Q_{1}, Q_{2})$, and write $g=g(s)$ for brevity.
Then a straightforward computations yields
\begin{align}
\begin{split}
Q_{1} (s,\tau; y)
&=2e^{4u (s)} \big(e^{4(u (s+\tau)-u (s))} -1- 4(u (s+\tau)-u (s)) \big) |\nabla_{g} v (s)|^2 \\
& \qquad +2e^{4u (s)} \big(e^{4(u (s+\tau)-u (s))} -1\big)\big (|\nabla_{g} v (s+\tau)|^2- |\nabla_{g} v (s)|^2 \big)\\
& \qquad + 2e^{4u (s)} |\nabla_{g} v (s+\tau)- \nabla_g v  (s) |^2,
\end{split}
\end{align}
and
\begin{equation}
Q_{2}(s,\tau; y)
= -4 (  \nabla_{g} u(s+\tau ) - \nabla_{g} u(s) )\cdot ( \nabla_{g} v(s+\tau)-  \nabla_{g} v(s)).
\end{equation}

For any $s_1, s_2\in  [-\varepsilon_0, \varepsilon_0]$ 
write $s_1=s$ and $s_2=s+\tau$, then
according to \eqref{eq:equation-difference} we find that $\Phi(s_2)-\Phi(s_1)$ is a solution
of \eqref{eq:model-linear} at $s=s_1$ with
\begin{equation}
f= P(s_1,s_2-s_1;\cdot)+Q(s_1,s_2-s_1;\cdot).
\end{equation}
By the explicit expressions of $Q_{1}$ and $Q_{2}$,
Proposition \ref{prop:local-continuity}, as well as \eqref{eq:tangent-uniform-1} and \eqref{eq:HKWX-main-2} we have
\begin{equation}
|Q_{1}(s_1,s_2-s_1; y)| +e^{2u (s_1)} |Q_{2}(s_1,s_2-s_1;y)|  \le C|s_1-s_2| \sigma^{-2}(1+|y|)^{-2}
\end{equation}
on $\mathbb{R}^3 \setminus\Gamma$. A similar estimate holds for $P$, since $P_{1}$, $P_{2}$ involve $g(s_1)-g(s_2)$ and are therefore supported in $B_{2\delta_0}\setminus B_{\delta_0}$. Denoting $f=(f_1, f_2)$, it follows that
\begin{equation}
|f_1(y)| +e^{2u (y)} |f_2(y)|  \le C|s_1-s_2| \sigma^{-2}   (1+|y|)^{-2} \quad\text{in }\mathbb{R}^3 \setminus\Gamma.
\end{equation}
By Lemma \ref{lem:linear-energy}, noting that $\Phi(s_1)-\Phi(s_2)\in\mathcal{H}$ in light of \eqref{eq:HKWX-main-2} and \eqref{oifnoiqntoinqionh}, we get
\begin{equation}
\|\Phi(s_1)-\Phi(s_2) \|_{\mathcal{H}} \le C |s_1-s_2|.
\end{equation}
Inequality \eqref{eq:sobolev-type} then produces
\begin{equation}
\left (\int_{\R^3} (|u (s_1)-u (s_2)| +e^{2u (s_1)} |v (s_1)-v (s_2)|)^6\, d y\right)^{1/6}
\le C |s_1-s_2|,
\end{equation}
and hence the distance formula \eqref{afoniaoinionihhh} yields 
\begin{equation}
\left(\int_{\R^3} d_{\mathbb{H}^2} (\Phi(s_1), \Phi(s_2))^6dy\right)^{1/6} \le C|s_1-s_2|.
\end{equation}
This is an improvement of Proposition \ref{prop:integral-continuity} from the power $1/2$ to 1 for $|s_1-s_2|$.
We can now repeat the proof of Proposition \ref{prop:local-continuity} to obtain the desired result.
\end{proof}

We will next prove differentiability of $\Phi(s)$ with respect to $s$.

\begin{prop}\label{prop:differentiable-1}
Let $\Phi=\Phi(s)$ be given by \eqref{eq:definition-u-s}, for any $s\in [-\varepsilon_0, \varepsilon_0]$.
Then $\nabla^k_g \Phi(s)$ is differentiable in $s$ except on the axis for $k=0,1,2,3$, and for any $\alpha \in (0,1)$ it holds that
\begin{align}\label{eq:partial-s-est-1}
|\nabla_g ^k \partial_su (s) |(y) + \Big(\frac{\rho}{\sigma}\Big)^{k-3-\alpha } |\nabla_g ^k \partial_sv (s) |(y)
\le C \sigma ^{-k}   (1+|y|)^{-1} \quad\text{in }\mathbb{R}^3 \setminus\Gamma,
\end{align}
where $C$ is a constant independent of $s$.
\end{prop}

\begin{proof} 
We will adopt the notation from the proof of Proposition \ref{prop:lip}. 
Consider the difference quotient
\begin{equation}
D_\tau \Phi(s)=\frac{1}{\tau} [\Phi(s+\tau)- \Phi(s)],
\end{equation}
for $\tau\neq0$.
By Proposition \ref{prop:lip}, for $k=0,1,2,3$ and any $s, s+\tau\in [-\va_0, \va_0]$ we have
\begin{equation}\label{eq:partial-s-est-1z}
|\nabla_g ^k D_\tau u (s) |(y) + \Big(\frac{\rho}{\sigma}\Big)^{k-3-\alpha } |\nabla_g ^k D_\tau v (s) |(y)
\le C \sigma ^{-k}   (1+|y|)^{-1} \quad \text{in }\mathbb{R}^3 \setminus\Gamma.
\end{equation}
We will prove that the limits of $\nabla ^k D_\tau u (s)$ and $\nabla ^k D_\tau v (s)$ exist as $\tau\to0$, therefore showing that $\nabla^k_g u (s)$ and $\nabla^k_g v (s)$ are differentiable in $s$ and satisfy \eqref{eq:partial-s-est-1}.

Dividing \eqref{eq:equation-difference} by $\tau$ produces
\begin{equation}
L_sD_\tau \Phi(s)=\frac{1}{\tau}P(s,\tau; y)+\frac{1}{\tau}Q(s,\tau; y).
\end{equation}
For any $s$, $\tau$, and $\tau'\neq 0$ with $s, s+\tau, s+\tau'\in [-\va_0, \va_0]$ define
$w(s, \tau, \tau')=D_\tau \Phi(s)-D_{\tau'} \Phi(s)$, and observe that a simple subtraction yields
\begin{equation}\label{eq-equation-w-difference}
L_s\big[w(s, \tau, \tau')\big]=f(s, \tau, \tau'),
\end{equation}
where
\begin{equation}
f(s, \tau, \tau')=\frac{1}{\tau}P(s,\tau; \cdot)-\frac{1}{\tau'}P(s,\tau'; \cdot)
+\frac{1}{\tau}Q(s,\tau; \cdot)-\frac{1}{\tau'}Q(s,\tau'; \cdot).
\end{equation}
The two differences on the right-hand side will be estimated below. In particular, 
by writing $f=(f_1, f_2)$ we claim that
\begin{equation}\label{eq-estimate-f-difference}
|f_1(s, \tau, \tau')|(y) +e^{2u (y)} |f_2(s, \tau, \tau')|(y)  \le C|\tau-\tau'| \sigma^{-2}   (1+|y|)^{-2}\quad\text{in }\mathbb{R}^3 \setminus\Gamma.
\end{equation}
With this, Lemma \ref{lem:linear-energy} applies to give
\begin{equation}
\|w(s, \tau, \tau')\|_{\mathcal{H}} \le C |\tau-\tau'|.
\end{equation}
Moreover, after setting $w=(w_1, w_2)$ we find that \eqref{eq:sobolev-type} implies
\begin{equation}
\left (\int_{\R^3} (|w_1(s, \tau, \tau')| +e^{2u (s)} |w_2(s, \tau, \tau')|)^6\, d y\right)^{1/6}
\le C |\tau-\tau'|,
\end{equation}
and hence by the definition of $w$ it follows that
\begin{equation}
\left(\int_{\R^3} d_{\mathbb{H}^2} (D_\tau \Phi(s), D_{\tau'} \Phi(s))^6dy\right)^{1/6} \le C|\tau-\tau'|.
\end{equation}

We now proceed as in the proof of Proposition \ref{prop:local-continuity}.
Take any $\mathbf{r}$ sufficiently large, $\delta$ sufficiently small, and any $\alpha\in (0,1)$.
By (the proof of) Proposition 3.9 of \cite{HKWX}, for $k=0,1,2,3$ we have
\begin{align}\label{eq-estimate-difference-R-delta}
\begin{split}
&|\nabla_g ^k [D_\tau u (s)-D_{\tau'} u (s)] | + \Big(\frac{\rho}{\sigma}\Big)^{k-3-\alpha } |\nabla_g ^k [D_\tau v (s)-D_{\tau'} v (s)]  |\\
&\qquad \le
 C |\tau-\tau'| \quad \mbox{in }B_{\mathbf{r}}\setminus\left(\Gamma \cup_{i=1}^N B_{\delta}(p_i^0)\right),
\end{split}
\end{align}
where $C$ is a constant depending on $\mathbf{r}$ and $\delta$, independent of $s$.
Hence, for $k=0,1,2,3$ the limits of $D_\tau \nabla ^ku (s)$ and $D_\tau \nabla ^kv (s)$ exist as $\tau\to0$,
and thus $\nabla ^ku (s)$ and $\nabla ^kv (s)$ are differentiable with respect to $s$ in $B_{\mathbf{r}}\setminus\left(\Gamma \cup_{i=1}^N B_{\delta}(p_i^0)\right)$ for any $\mathbf{r}, \delta>0$.

We now establish \eqref{eq-estimate-f-difference}.
Consider first the term
\begin{equation}
Q_2(s,\tau; \cdot)=-4\tau^2   \nabla_{g} D_\tau u(s) \cdot  \nabla_{g} D_\tau v(s),
\end{equation}
and observe that
\begin{align}
\begin{split}
\frac{1}{4\tau}Q_2(s,\tau; \cdot)-\frac{1}{4\tau'}Q_2(s,\tau'; \cdot)
&=-\tau   \nabla_{g} D_\tau u(s) \cdot  \nabla_{g} D_\tau v(s)+\tau'   \nabla_{g} D_{\tau'} u(s) \cdot  \nabla_{g} D_{\tau'} v(s)\\
&=-\nabla_{g}\left(u(s+\tau)-u(s+\tau')\right)\cdot\nabla_{g} D_{\tau}v(s)\\
&\quad\text{ }+(\tau-\tau')\nabla_{g} D_{\tau'}u(s)\cdot \nabla_{g}D_{\tau}v(s)\\
&\quad\text{ }+\nabla_{g}\left(v(s+\tau')-v(s+\tau)\right)\cdot\nabla_{g}D_{\tau'}u(s).
\end{split}
\end{align}
Then Proposition \ref{prop:lip} and \eqref{eq:partial-s-est-1z} imply
\begin{equation}
e^{2u (y)} \Big|\frac{1}{\tau}Q_2(s,\tau; y)-\frac{1}{\tau'}Q_2(s,\tau'; y)\Big|  \le C|\tau-\tau'| \sigma^{-2}   (1+|y|)^{-2}
\quad\text{in }\mathbb{R}^3 \setminus\Gamma.
\end{equation}
Estimates for $Q_1$, $P_1$, and $P_2$ may be achieved similarly, yielding the desired conclusion.
\end{proof}

For a later purpose, we shall improve \eqref{eq-estimate-difference-R-delta} so that for any large $\mathbf{r}$, $\alpha\in(0,1)$, and $k=0,1,2,3$ the following estimate is valid
\begin{align}\label{eq-estimate-difference-R}
\begin{split}
&|\nabla_g ^k [D_\tau u (s)-D_{\tau'} u (s)] | + \Big(\frac{\rho}{\sigma}\Big)^{k-3-\alpha } |\nabla_g ^k [D_\tau v (s)-D_{\tau'} v (s)]  |\\
&\qquad \le
 C \sigma^{-k}|\tau-\tau'| \quad \mbox{in }B_{\mathbf{r}}\setminus\Gamma,   
\end{split}
\end{align}
where $C$ is a constant depending only on $\mathbf{r}$, independent of $s$.
To see this, first rewrite \eqref{eq-estimate-difference-R-delta} as
\begin{equation}
|\nabla_g ^k w_1(s, \tau, \tau') | + \Big(\frac{\rho}{\sigma}\Big)^{k-3-\alpha } |\nabla_g ^k w_2(s, \tau, \tau')|
\le  C |\tau-\tau'| \quad \mbox{in }B_{\mathbf{r}}\setminus\left(\Gamma \cup_{i=1}^N B_{\delta}(p_i^0)\right).
\end{equation}
By applying Proposition \ref{prop:linear-boundness} to equation \eqref{eq-equation-w-difference}
near each puncture $p_i^0$, we find that there exists a $\delta>0$ sufficiently small such that 
\begin{align}
\begin{split}
&\sup_{B_{\mathbf{r}}\setminus \Gamma}\Big(|w_1(s, \tau, \tau') | + \Big(\frac{\rho}{\sigma}\Big)^{-2 } |w_2(s, \tau, \tau')|\Big)\\
&\qquad \le \sup_{B_{\mathbf{r}}\setminus\left(\Gamma \cup_{i=1}^N B_{\delta}(p_i^0)\right)}\Big(|w_1(s, \tau, \tau') | + \Big(\frac{\rho}{\sigma}\Big)^{-2 } |w_2(s, \tau, \tau')|\Big)\\
&\qquad\qquad+\sup_{B_{\mathbf{r}}\setminus \Gamma}\Big(|\sigma^2 f_1(s, \tau, \tau') | + \Big(\frac{\rho}{\sigma}\Big)^{-2} |\sigma^2 f_2(s, \tau, \tau')|\Big)\\
&\qquad \le  C |\tau-\tau'|.
\end{split}
\end{align}
This inequality, together with the proof of Proposition 3.9 of \cite{HKWX}, then yields \eqref{eq-estimate-difference-R}.
We next study the asymptotic behavior of $\pa_s \Phi(s)$ as $y\to\{p_1^0,\dots, p_N^0\}$.

\begin{prop}\label{prop:limit-derivative-1}
Let $\Phi=\Phi(s)$ be given by \eqref{eq:definition-u-s} and $\bar{\Phi}_i(s)=(\bar{u}_{i}(s),\bar{v}_{i}(s))$
be the map satisfying \eqref{eq:HKWX-main-1}, for any $s\in [-\varepsilon_0, \varepsilon_0]$ and any puncture $p_i^0$. Then $\nabla^k_{\mathbb S^2}\bar{\Phi}_i(s)$ is differentiable
in $s$ except at the north and south pole for $k=0,1,2,3$, and for any $\alpha \in (0,1)$ we have
\begin{equation}\label{eq:partial-s-est-sphere-1}
|\nabla_{\mathbb S^2}^k \partial_s\bar{u}_{i}(s) |+ (\sin\theta_i)^{k-3-\alpha } |\nabla_{\mathbb S^2}^k \partial_s\bar{v}_{i}(s) |
\le C  \quad\text{on }\mathbb S^2 \setminus\{N,S\},
\end{equation}
and for $l+k\leq 3$ it holds that 
\begin{align}\label{eq-estimate-derivative-s-1}
\begin{split}
&\sup_{\mathbb{S}^2  \setminus\{N,S\}}\big(|(r_i\partial_{r_i})^l \nabla_{\mathbb{S}^2}^k [\pa_s u (s,r_i,\theta_i)-\partial_s\bar{u}_{i}(s,\theta_i)]|\\
&\qquad+e^{(3+\alpha-k)u (s)}|(r_i\partial_{r_i})^l \nabla_{\mathbb{S}^2}^k [\pa_s v (s,r_i,\theta_i)-\partial_s\bar{v}_{i}(s,\theta_i)]|\big)
\le C r_i^{\bar\beta}
\end{split}
\end{align}
in $B_{\delta_0}(p_i^0)\setminus\{p_i^0\}$, where $C$ and $\bar\beta$ are positive constants independent of $s$, $r_i=|y-p_i^0|$, and $\sin\theta_i=\tfrac{\rho(y)}{r_i}$.
\end{prop}

\begin{proof}
The proof consists of three steps. Throughout, we will fix a puncture $p_i^0$.

{\it Step 1.} We will show that $\nabla^k_{\mathbb S^2}\bar{\Phi}_i(s)$ is differentiable
in $s$ for $k=0,1,2,3$, and that \eqref{eq:partial-s-est-sphere-1} holds.
To this end recall that for $k=0,1,2,3$, any $s_1, s_2\in  [-\varepsilon_0, \varepsilon_0]$, and any $\alpha \in (0,1)$,
Proposition \ref{prop:lip} implies
\begin{align}
\begin{split}
&|\nabla_g ^k [u (s_1)-u (s_2)] |(y) + \Big(\frac{\rho}{\sigma}\Big)^{k-3-\alpha } |\nabla_g ^k [v (s_1)-v (s_2)]  |(y)\\
&\qquad
\le C \sigma ^{-k}   (1+|y|)^{-1}  |s_1-s_2|\quad\text{in }\mathbb{R}^3 \setminus\Gamma.   
\end{split}
\end{align}
By letting $y\to p_i^0$ and using \eqref{eq:HKWX-main-1} we obtain 
\begin{align}
\begin{split}
&|\nabla_{\mathbb S^2} ^k [\bar{u}_{i}(s_1)-\bar{u}_{i}(s_2)] |
+ (\sin\theta)^{k-3-\alpha } |\nabla_{\mathbb S^2} ^k [\bar{v}_{i}(s_1)-\bar{v}_{i}(s_2)]  |\\
&\qquad
\le C   |s_1-s_2|\quad\text{on }\mathbb S^2 \setminus\{N,S\}.
\end{split}
\end{align}
In other words, $\nabla_{\mathbb S^2} ^k \bar{u}_{i}(s)$ is Lipschitz in $s$ for $k=0,1,2,3$.
This is the counterpart of Proposition \ref{prop:lip} for $\bar{\Phi}_{i}=(\bar{u}_{i}, \bar{v}_{i})$.
Note that $\bar{\Phi}_{i}$ is a harmonic map from $\mathbb S^2\setminus \{N,S\}$ to $\mathbb H^2$.
By proceeding as in the proof of Proposition \ref{prop:differentiable-1}, we obtain the desired result.
In fact, the proof in the present case is easier since $\mathbb S^2$ is compact.

{\it Step 2.} We will show that there exists a map $\bar{\Phi}^{(1)}_i(s)=(\bar{u}^{(1)}_{i}(s), \bar{v}^{(1)}_{i}(s))$ with
both component functions in $C^{3,\alpha}(\mathbb{S}^2)$ for any $\alpha\in(0,1)$, while the second satisfies $\bar{v}^{(1)}_{i}(s)=0$ at $N$ and $S$, such that for $l+k\leq 3$ it holds that
\begin{align}\label{eq-estimate-derivative-s-1-main}
\begin{split}
&\sup_{\mathbb{S}^2  \setminus\{N,S\}}\big(|(r_i \partial_{r_i})^l \nabla_{\mathbb{S}^2}^k [\pa_s u (s,r_i,\theta_i)-\bar{u}^{(1)}_{i}(s,\theta_i)]|\\
&\qquad+e^{(3+\alpha-k)u (s)}|(r_i\partial_{r_i})^l \nabla_{\mathbb{S}^2}^k [\pa_s v (s,r_i,\theta_i)-\bar{v}^{(1)}_{i}(s,\theta_i)]|\big)
\le C r_i^{\bar\beta}
\end{split}
\end{align}
in $B_{\delta_0}(p_i^0)\setminus\{p_i^0\}$, where $C$ and $\bar{\beta}$ are independent of $s$.
To prove \eqref{eq-estimate-derivative-s-1-main} differentiate \eqref{eq:perturbation} with respect to $s$, and observe that
in view of \eqref{eq:linearized-operator} we have
\begin{equation}\label{eq-equation-derivative-1}
L_{s} \pa_s \Phi=\mathcal{P}_1,
\end{equation}
where $\mathcal{P}_1=(\mathcal{P}_{1,1}, \mathcal{P}_{1,2})$ with
\begin{align}
\begin{split}
\mathcal{P}_{1,1}&= - \pa_s g^{ab} \pa_{ab}u -\pa_s (\pa_a g^{ab} +g^{ab} \pa_a \ln \sqrt{\det g} ) \pa_b u +2 e^{4u } \pa_s g^{ab } \pa_a v  \pa_b v ,\\
\mathcal{P}_{1,2}&= - \pa_s g^{ab} \pa_{ab }v -\pa_s (\pa_a g^{ab} +g^{ab} \pa_a \ln \sqrt{\det g} ) \pa_b v -4 \pa_s g^{ab } \pa_a u  \pa_b v .
\end{split}
\end{align}
Note that $\mathcal{P}_1$ is supported in $B_{2\delta_0}(p_i^0)\setminus B_{\delta_0}(p_i^0)$ so that
\begin{equation}
L_s \pa_s \Phi= 0 \quad \mbox{in }B_{\delta_0}(p_i^0)\setminus\Gamma, \quad\quad \partial_s v=0\quad\mbox{on }\Gamma\cap B_{\delta_0}(p_i^0).
\end{equation}
Utilizing \eqref{eq:partial-s-est-1}, Proposition \ref{prop:linear-asymptotic} as well as the proof of Proposition 3.8 of \cite{HKWX} for higher order derivatives, we obtain \eqref{eq-estimate-derivative-s-1-main}.

{\it Step 3.} We will now prove that $\partial_s\bar{\Phi}_{i}=\bar{\Phi}^{(1)}_i$. Since
$D_{\tau'} \Phi(s)\to \partial_s\Phi(s)$ as $\tau'\to0$ by Proposition \ref{prop:differentiable-1}, it follows that letting $\tau'\to0$ and taking $k=0$ in \eqref{eq-estimate-difference-R} produces
\begin{equation}
|D_\tau u (s)-\partial_s u (s) | + \Big(\frac{\rho}{\sigma}\Big)^{-3-\alpha } |D_\tau v (s)-\partial_s v (s) |
\le  C |\tau| \quad \mbox{in }B_{\mathbf{r}}\setminus\Gamma.
\end{equation}
Sending $y\to p_i^0$ then yields
\begin{equation}
|D_\tau \bar{u}(s)-\bar{u}_{i}^{(1)}(s) |
+ (\sin\theta_i)^{-3-\alpha }|D_\tau \bar{v}(s)-\overline{v}_{i}^{(1)}(s)|
\le  C |\tau| \quad \mbox{on }\mathbb S^2 \setminus\{N,S\}.
\end{equation}
Finally, by $\tau\to0$ we obtain $\partial_s\bar{\Phi}_{i}=\bar{\Phi}^{(1)}_i$.
\end{proof}


Lastly, we will establish the smooth dependence of $\Phi(s)$ on $s$.

\begin{thm}\label{thm:smooth-dep}
Let $\Phi=\Phi(s)$ be given by \eqref{eq:definition-u-s} and $\bar{\Phi}_i(s)=(\bar{u}_{i}(s),\bar{v}_{i}(s))$
be the map satisfying \eqref{eq:HKWX-main-1}, for any $s\in [-\varepsilon_0, \varepsilon_0]$ and any puncture $p_i^0$.
Then for $k=0,1,2,3$, the maps $\nabla^k_g\Phi(s)$ and
$\nabla^k_{\mathbb S^2}\bar{\Phi}_i(s)$ are smooth in $s$, except on the axis.
Moreover for these $k$, any $j\ge 1$, and any $\alpha \in (0,1)$ we have
\begin{align}\label{eq:partial-s-est-j}
|\nabla_g ^k \partial_s^ju (s) |(y) + \Big(\frac{\rho}{\sigma}\Big)^{k-3-\alpha } |\nabla_g ^k \partial_s^jv (s) |(y)
&\le C \sigma ^{-k}   (1+|y|)^{-1} \quad\text{in }\mathbb R^3 \setminus\Gamma,\\
\label{eq:partial-s-est-sphere-j}
|\nabla_{\mathbb S^2}^k \partial_s^j\bar{u}_{i}(s) |+ (\sin\theta_i)^{k-3-\alpha } |\nabla_{\mathbb S^2}^k \partial_s^j\bar{v}_{i}(s) |
&\le C  \quad\text{on }\mathbb S^2\setminus\{N,S\},
\end{align}
and for $l+k\leq 3$ and $0<r_i\le \delta_0$ it holds that
\begin{align}\label{eq-estimate-derivative-s-j}
\begin{split}
&\sup_{\mathbb{S}^2  \setminus\{N,S\}}\big(|(r_i\partial_{r_i})^l \nabla_{\mathbb{S}^2}^k [\pa_s^j u (s,r_i,\theta_i)-\partial_s^j\bar{u}_{i}(s,\theta_i)]|\\
&\qquad+e^{(3+\alpha-k)u (s)}|(r_i \partial_{r_i})^l \nabla_{\mathbb{S}^2}^k [\pa_s^j v (s,r_i,\theta_i)-\partial_s^j\bar{v}_{i}(s,\theta_i)]|\big)
\le C r_i^{\bar\beta},
\end{split}
\end{align}
where $C$ and $\bar\beta$ are positive constants independent of $s$, and $r_i=|y-p_i^0|$ with $\sin\theta_i=\tfrac{\rho(y)}{r_i}$.
\end{thm}

\begin{proof} 
By Propositions \ref{prop:differentiable-1} and \ref{prop:limit-derivative-1}
we have that $\nabla^k_g\Phi(s)$ and
$\nabla^k_{\mathbb S^2}\bar{\Phi}_i(s)$ are differentiable in $s$,
and that \eqref{eq:partial-s-est-j}-\eqref{eq-estimate-derivative-s-j} hold for $j=1$.
We now prove that $\nabla^k_g\Phi(s)$ and
$\nabla^k_{\mathbb S^2}\bar{\Phi}_i(s)$ are twice differentiable in $s$,
and that \eqref{eq:partial-s-est-j}-\eqref{eq-estimate-derivative-s-j} hold for $j=2$.
The proof consists of several steps. Recall that $\partial_s\Phi$ satisfies \eqref{eq-equation-derivative-1}.

{\it Step 1.} We claim that for $k=0,1,2,3$ any $s_1, s_2\in  [-\varepsilon_0, \varepsilon_0]$,
and any $\alpha \in (0,1)$ it holds that
\begin{align}
\begin{split}
&|\nabla_g ^k [\partial_su (s_1)-\partial_su (s_2)] |(y) + \Big(\frac{\rho}{\sigma}\Big)^{k-3-\alpha } |\nabla_g ^k [\partial_sv (s_1)-\partial_sv (s_2)]  |(y)\\
&\qquad
\le C \sigma ^{-k}   (1+|y|)^{-1}  |s_1-s_2|\quad\text{in }\mathbb{R}^3 \setminus\Gamma, 
\end{split}
\end{align}
where $C$ is a positive constant independent of $s_1$ and $s_2$. The proof is similar to that of Proposition \ref{prop:lip}, with $\Phi$ there replaced by $\partial_s\Phi$.

{\it Step 2.} We claim that $\nabla^k_g\partial_s\Phi(s)$ is differentiable
in $s$ for $k=0,1,2,3$, and that \eqref{eq:partial-s-est-j} holds for $j=2$.
The proof is similar to that of Proposition \ref{prop:differentiable-1}, again with $\Phi$ replaced by $\partial_s\Phi$.

{\it Step 3.} We claim that $\nabla^k_{\mathbb S^2}\partial_s\bar{\Phi}_i(s)$ is differentiable
in $s$ for $k=0,1,2,3$, and that \eqref{eq:partial-s-est-sphere-j} holds for $j=2$.
The proof is similar to Step 1 in the proof of Proposition \ref{prop:limit-derivative-1}.

{\it Step 4.} We will show that there exists a map $\bar{\Phi}^{(2)}_i(s)=(\bar{u}^{(2)}_{i}(s), \bar{v}^{(2)}_{i}(s))$ with both component functions in
$C^{3,\alpha}(\mathbb{S}^2)$ for any $\alpha\in(0,1)$, while the second satisfies $\bar{v}^{(2)}_{i}(s)=0$ at $N$ and $S$,
such that for $l+k\leq 3$ and $0<r_i\le \delta_0$ it holds that
\begin{align}\label{eq-estimate-derivative-s-2-main}
\begin{split}
&\sup_{\mathbb{S}^2 \setminus\{N,S\}}\big(|(r_i\partial_{r_i})^l \nabla_{\mathbb{S}^2}^k [\pa^2_s u (s,r_i,\theta_i)-\bar{u}^{(2)}_{i}(s,\theta_i)]|\\
&\qquad+e^{(3+\alpha-k)u (s)}|(r_i\partial_{r_i})^l \nabla_{\mathbb{S}^2}^k [\pa^2_s v (s,r_i,\theta_i)-\bar{v}^{(2)}_{i}(s,\theta_i)]|\big)
\le C r_i^{\bar\beta},
\end{split}
\end{align}
where $C$ and $\bar{\beta}$ are independent of $s$. To this end differentiate \eqref{eq-equation-derivative-1} with respect to $s$, or
differentiate \eqref{eq:perturbation} with respect to $s$ twice to obtain
\begin{equation}\label{eq-equation-derivative-2}
L_{s} \pa^2_s \Phi=\mathcal{P}_{2},
\end{equation}
where $\mathcal{P}_{2}=(\mathcal{P}_{2,1}, \mathcal{P}_{2,2})$ with
\begin{align}
\begin{split}
\mathcal{P}_{2,1}&= \pa_s \mathcal{P}_{1,1} - \pa_s g^{ab} \pa_{ab }\pa_s u -\pa_s (\pa_a g^{ab} +g^{ab} \pa_a \ln \sqrt{\det g} ) \pa_b \pa_s u  \\ &
 \qquad +8 \pa_s (e^{4u} |\nabla_{g}  v|^2)  \pa_s u   +4 \pa_s (e^{4u} \nabla_{g} v  \cdot \nabla_{g} )  \pa_s v ,\\
\mathcal{P}_{2,2}&=\pa_s  \mathcal{P}_{1,2}- \pa_s g^{ab} \pa_{ab }\pa_s v -\pa_s (\pa_a g^{ab}
+g^{ab} \pa_a \ln \sqrt{\det g} ) \pa_b \pa_s v \\
& \qquad
   -4 \pa_s (   \nabla_{g} u \cdot \nabla_{g} )  \pa_s v   -4 \pa_s (   \nabla_{g} v \cdot \nabla_{g} )  \pa_s u .
\end{split}
\end{align}
We may then proceed similarly to Step 2 in the proof of Proposition \ref{prop:limit-derivative-1} to obtain the desired conclusion.

{\it Step 5.} We claim that $\partial_s^2\bar{\Phi}_{i}=\bar{\Phi}^{(2)}_i$. The proof is similar to Step 3 in the proof of Proposition \ref{prop:limit-derivative-1}.
This completes the proof for $j=2$. 

The general case may be treated by induction.
\end{proof}

\begin{proof}[Proof of Theorem \ref{thm:smoothness}] Assuming all components of $\mathbf{z}$ are allowed to vary, the same conclusions as above remain valid. In particular the map $B_{\varepsilon}^N(\mathbf{z}_0)\rightarrow C_{loc}^{3,\varsigma}(\mathbb{R}^3 \setminus \cup_{i=1}^N \{p_{i}^0\},\mathbb{R}^2)$ given by 
$\mathbf{z}\mapsto (U_{\mathbf{z}}\circ\Pi_{\mathbf{z}}, v_{\mathbf{z}}\circ\Pi_{\mathbf{z}})$, and the map $B_{\varepsilon}^N(\mathbf{z}_0)\rightarrow C^{3,\varsigma}(\mathbb{S}^2 ,\mathbb{R}^2)$ given by $\mathbf{z}\mapsto (\bar{U}_{\mathbf{z},i}, \bar{v}_{\mathbf{z},i})$ for each $i=1,\dots,N$, are 
smooth in $s$ for some small $\va>0$. Since $\Pi_{\mathbf{z}}$ is a diffeomorphism, the first conclusion of Theorem \ref{thm:smoothness} holds. 

We will now consider estimates for the derivatives with respect to $\mathbf{z}$. Observe that for each $j=1,\dots,N$, a simple differentiation along with $\Pi_{\mathbf{z}}(y)=x$ produces
\begin{align}
\begin{split}
 \partial_{z_j}  (\Phi_{\mathbf{z}}\circ \Pi_{\mathbf{z}}(y))&=\partial_{z_j}  \Phi_{\mathbf{z}} (x)+ \nabla\Phi_\mathbf{z}(x) \cdot  \partial_{z_j}  \Pi_{\mathbf{z}}(y)\\&
 = \partial_{z_j}  \Phi_{\mathbf{z}} (x) +\pa_{x_3} \Phi_\mathbf{z}(x) \bar \chi(y_1,y_2,y_3-z_{0,j}).
\end{split}
\end{align}
Let $\zeta_j(x)=\bar \chi(\Pi_{\mathbf{z}}^{-1}(x)-z_{0,j}e_3)$ and note that this function is supported in $B_{2\delta_0}(p_j^0)$, and set $\tilde \sigma(x)=\sigma(\Pi_{\mathbf{z}}^{-1}(x))$.
It follows from Theorem \ref{thm:smooth-dep} that for $k=0,1,2,3$ and any $\alpha\in(0,1)$ we have
\begin{equation}
|\nabla^k( \partial_{z_j} u_{\mathbf{z}} + \pa_{x_3} u_{\mathbf{z}} \zeta_j ) |(x) + \Big(\frac{\rho}{\tilde \sigma}\Big)^{k-3-\alpha } |\nabla ^k ( \partial_{z_j} v_{\mathbf{z}} + \pa_{x_3} v_{\mathbf{z}} \zeta_j )|(x)
\le C \tilde \sigma ^{-k}   (1+|x|)^{-1} 
\end{equation}
in $\mathbb R^3\setminus\Gamma$, 
as well as 
\begin{equation}
|\nabla_{\mathbb S^2}^k \partial_{z_j}\bar{u}_{\mathbf{z}, i} |+ (\sin\theta_i)^{k-3-\alpha } |\nabla_{\mathbb S^2}^k \partial_{z_j}\bar{v}_{\mathbf{z},i} |
\le C  \quad\text{on }\mathbb S^2 \setminus\{N,S\}, 
\end{equation}
and for $l+k\leq 3$ and $0<r_i\leq \delta_0$ it holds that 
\begin{align}
\begin{split}
&\sup_{\mathbb{S}^2 \setminus\{N,S\}}\big(|(r_i \partial_{r_i})^l \nabla_{\mathbb{S}^2}^k [( \partial_{z_j} u_{\mathbf{z}} + \pa_{x_3} u_{\mathbf{z}} \zeta_j ) -\partial_{z_j}\bar{u}_{\mathbf{z}, i}]|\\
&\qquad+e^{(3+\alpha-k)u_{\mathbf{z}}}|(r_i \partial_{r_i})^l \nabla_{\mathbb{S}^2}^k [( \partial_{z_j} v_{\mathbf{z}} + \partial_{x_3} v_{\mathbf{z}} \zeta_j ) -\partial_{z_j}\bar{v}_{\mathbf{z},i}]|\big)
\le C r_i ^{\bar\beta},
\end{split}
\end{align}
for some constants $C$ and $\bar\beta>0$.
\end{proof} 

\begin{proof}[Proof of Corollary \ref{cor:linest}] By a direct computation
 \begin{equation}
 (\dot U_t,\dot v_t)=\sum_{j=1}^N \pa_{z_j}\left(U_{\mathbf{z}},v_{\mathbf{z}}\right) \dot z_j. 
 \end{equation}
In each ball $B_{\va/2}(p_i)$  we have the following expansion by \eqref{eq:thm3.1-2} of Theorem \ref{thm:smoothness}, namely 
\begin{equation}
 \pa_{z_j}\left(U_{\mathbf{z}},v_{\mathbf{z}}\right)  =- \pmb{\delta}_{ij}\left(\pa_{x_3} U_{\mathbf{z}},\pa_{x_3} v_{\mathbf{z}}\right)+\pa_{z_j} \left(\bar U_{\mathbf{z},i}, \bar v_{\mathbf{z},i}\right) +O(r_i^{\bar \beta}) .
\end{equation}
Next observe that
\begin{equation}
\pa_{z_j} \left(\bar U_{\mathbf{z},i}, \bar v_{\mathbf{z},i}\right)\dot z_j =\pa_t \left(\bar U_{\mathbf{z}(t),i}, \bar v_{\mathbf{z}(t),i}\right)
=\pa_{b_i}\left(\bar U_{\mathbf{z},i}, \bar v_{\mathbf{z},i}\right)\dot b_i,
\end{equation}
and by \cite[Theorems 2.1 and 2.2]{HKWX} (see \eqref{3.1}-\eqref{3.3}) it holds that
\begin{align}
\begin{split}
\left(\pa_{x_3} U_{\mathbf{z}},\pa_{x_3} v_{\mathbf{z}}\right)&=\left(\tfrac{x_3-z_i}{r_i^2},0\right)+\pa_{\theta_i} \left(\bar U_{\mathbf{z},i}, \bar v_{\mathbf{z},i}\right) \partial_{x_3} \theta_{i} +O(r_{i}^{\beta_i-1})\\&
=\left(\tfrac{x_3-z_i}{r_i^2},0\right)-\frac{\rho}{r_i^2}\pa_{\theta_i} \left(\bar U_{\mathbf{z},i}, \bar v_{\mathbf{z},i}\right)  +O(r_{i}^{\beta_i-1}).
\end{split}
\end{align}
Therefore
\begin{equation}
 \left(\dot U_t,\dot v_t \right) = \left[-\left(\tfrac{x_3-z_i}{r_i^2},0\right)+\frac{\rho}{r_i^2}\pa_{\theta_i} (\bar U_{\mathbf{z},i}, \bar v_{\mathbf{z},i})\right]\dot z_i+\pa_{b_i}\left(\bar U_{\mathbf{z},i}, \bar v_{\mathbf{z},i}\right)\dot b_i+O(r_{i}^{\beta_i-1}). 
\end{equation}
Regarding $\dot v_t$, the remainder $O(r_{i}^{\beta_i-1})$ can be refined to $O(r_{i}^{\beta_i-1} (\sin\theta_i)^{3})$ as we have the weighted estimates. Hence \eqref{abcd} is proved. The remaining two estimates of this corollary may be established similarly, and we omit the details.
\end{proof}

\appendix

\section{Miscellaneous Computations}

\begin{lem} \label{lem:r1-r2}
Let $r_1$, $r_2$ be the respective Euclidean distances to the (Cartesian) points $(0,0,\mp2\eta)\in\mathbb{R}^3$ where $\eta>0$, and set
$D=\{(\rho,z,\phi)\in\mathbb{R}^3\mid z>0, \> r_2>\eta\}$. Then within $D$ we have
\begin{equation}    \label{eq:r1-r2}
\frac{r_1}{r_2}\leq 5.
\end{equation}
\end{lem}

\begin{proof}
On $D$, consider the function
\begin{equation}
f(\rho,z)=\frac{r_1^2}{r_2^2}=\frac{\rho^2+(z-2\eta)^2+8\eta z}{r_2^2}=1+\frac{8\eta z}{r_2^2}.
\end{equation}
Clearly $f$ is a decreasing function of $\rho$, and hence $f(\rho,z)\leq f(\bar\rho,z)$ for $(\bar\rho,z) \in \tilde{\partial} D$ where $\tilde{\partial} D$ is the left-hand portion of the projected boundary of $D$ within the $\rho z$-half plane. This left-hand portion of the projected boundary consists of three parts within the half plane:
$\text{I}=\{ \rho=0, \> 0\leq z\leq \eta\}$, $\text{II}=\{r_2=\eta, \> \eta\leq z\leq 3\eta\}$, and $\text{III}=\{\rho=0, \> 3\eta\leq z<\infty\}$. As a function of $z$, we find that $f$ is increasing on $\text{I}\cup\text{II}$ and decreasing on $\text{III}$. Therefore, its maximum value on the closure of $D$ is achieved at the $z$-axis when $z=3\eta$, and~\eqref{eq:r1-r2} follows.
\end{proof}

\begin{cor} \label{cor:r1-r2}
Let $r_1$, $r_2$ be the respective Euclidean distances to two distinct points $p_1,p_2\in\R^3$, and set $4\eta=|p_1-p_2|$ to be the distance between them. Then
\begin{equation}
\frac15 \leq \frac{r_1}{r_2} \leq 5
\end{equation}
outside of the two balls of radius $\eta$ centered at $p_1$ and $p_2$.
\end{cor}

\begin{proof}
Without loss of generality it may be assumed that the points $p_1$ and $p_2$ are located on the $z$-axis, and that the origin is located at the midpoint between them. Lemma~\ref{lem:r1-r2} then states that $r_1/r_2\leq 5$ in the upper half-space $z>0$ outside $B_{\eta}(p_2)$, and hence $r_2/r_1\geq 1/5$ there. Also clearly $r_2<r_1$ in the upper half-space, so that $r_1/r_2\geq1$ there. By symmetry we obtain  $r_1/r_2\geq 1/5$ in the lower half-space outside $B_\eta(p_1)$, and also $r_1/r_2\leq 1$ there. Combining these inequalities yields the desired result.
\end{proof}

\begin{lem} \label{lem:r0-ri}
Let $\delta,\eta>0$ and $p_0,p_1,p_2 \in \mathbb{R}^3$ be such that $B_\eta(p_i)\subset B_{\delta/4}(p_0)$, $i=1,2$ with $|p_1-p_2|=4\eta$ and $p_0=(p_1+p_2)/2$. Set $r_i$ to be the Euclidean distance to $p_i$, $i=0,1,2$. Then for any $\lambda\in[0,1]$ we have
\begin{equation}\label{eq:r0-ri}
|\lambda \ln r_1 + (1-\lambda)\ln r_2 - \ln r_0| \leq \ln 2,\quad\quad\quad
|\nabla\ln r_i| \leq \frac{2}{r_0},
\end{equation}
outside $B_{\delta/2}(p_0)$.
\end{lem}

\begin{proof}
Let $i=1,2$. By the triangle inequality 
\begin{equation}
r_0\leq r_i+2\eta<r_i+\frac\delta4,
\end{equation}
and thus dividing by $r_i$ produces
\begin{equation}\label{goqanoigoiq}
\frac{r_0}{r_i}\leq 1+\frac{\delta/4}{r_i}\leq 2
\end{equation}
outside $B_{\delta/2}(p_0)$, since in this region $r_i\geq\delta/4$. Similarly
\begin{equation}
r_i\leq r_0 + 2\eta< r_0 + \frac{\delta}{4},
\end{equation}
and hence
\begin{equation}
\frac{r_i}{r_0} \leq 1 + \frac{\delta/4}{r_0}\leq \frac{3}{2}
\end{equation}
outside $B_{\delta/2}(p_0)$.
Therefore in this region
\begin{equation}
|\ln r_i - \ln r_0| \leq \ln 2,
\end{equation}
and~\eqref{eq:r0-ri} follows by the triangle inequality after rewriting $\ln r_0$ as a convex combination 
\begin{equation}
|\lambda \ln r_1 + (1-\lambda)\ln r_2 - \ln r_0|\leq \lambda|\ln r_1 -\ln r_0|
+(1-\lambda)|\ln r_2 -\ln r_0|\leq \ln 2.
\end{equation}

Lastly, for the gradient bound observe that \eqref{goqanoigoiq} implies 
$1/r_i\leq 2/r_0$ outside $B_{\delta/2}(p_0)$. The desired estimate now follows.
\end{proof}

\begin{lem}\label{aoijfoainoinhk}
Let $\Phi_i=(u_i,v_i)$, $i=1,2$ be extreme multi-Kerr harmonic maps with the same potential constants on $\Gamma_a^+=[a,\infty)\subset z$-axis, for $a>0$ sufficiently large so that there are no punctures on this axis interval. Then $d_{\mathbb{H}^2}(\Phi_1,\Phi_2)\rightarrow 0$ as $r\rightarrow \infty$ with $0\leq\theta\leq\pi/2$.
\end{lem}

\begin{proof}
Recall the hyperbolic distance formula \cite[Lemma 3]{Wei90} given by
\begin{equation}
\cosh\left(2d_{\mathbb{H}^2}(\Phi_1,\Phi_2)\right)=\cosh\left(2(u_1 -u_2)\right)+2e^{2(u_1 +u_2)}(v_1 -v_2)^2.
\end{equation}
Observe that the expansions of \cite[Theorem 2.3]{HKWX}, on the complement of ball $B_a(0)$ with $0\leq \theta\leq\pi/2$, imply that
\begin{equation}\label{fpaioifnh}
u_1 - u_2 =O(r^{-1}), \qquad v_1 -v_2 =O(\sin^{3} \theta),
\end{equation}
as $r\rightarrow\infty$. Here we have utilized the fact that the constant term in the expansion of $u_i$ vanishes, see Lemma \ref{lem:exterior-uniform}. Therefore
\begin{equation}
\cosh\left(2d_{\mathbb{H}^2}(\Phi_1,\Phi_2)\right)=1+O(r^{-1})
\end{equation}
in this region, and the desired result follows.
\end{proof}

\begin{rem}\label{foanfoinhghhh}
We note that an analogous result holds for $\Gamma_a^{-}=(-\infty,-a]$ and $\pi/2\leq\theta\leq\pi$, using the same asymptotics \eqref{fpaioifnh}. Similarly, the result is valid for $0<\theta_1\leq\theta\leq\theta_2<\pi$ without hypotheses concerning potential constants, where we only need to use the first inequality in \eqref{fpaioifnh} and boundedness of $v_1$, $v_2$.
\end{rem}

\end{document}